\documentclass[10pt]{amsart}
\usepackage{enumitem}
\usepackage{amscd}
\usepackage{mathtools}
\usepackage{amssymb}
\usepackage{MnSymbol}
\usepackage[all]{xy}
\usepackage{fge}
\usepackage{hyperref}
\usepackage{graphics}

\usepackage[textsize=tiny]{todonotes}

\RequirePackage[T1]{fontenc}
\RequirePackage{amsfonts,latexsym,amssymb}

\let\cal\mathcal

\usepackage{bbm}
\usepackage{hyperref}

\newtheorem{theorem}[equation]{Theorem}
 \newtheorem{lemma}[equation]{Lemma}
 \newtheorem{proposition}[equation]{Proposition}
 \newtheorem{corollary}[equation]{Corollary}

\theoremstyle{definition}
\newtheorem{definition}[equation]{Definition}

\newtheorem{remark}[equation]{Remark}

\newtheorem{example}[equation]{Example}
\theoremstyle{remark}
\newtheorem*{acknowledgments}{Acknowledgments}

\def\invlim{\mathop{\vtop{\ialign{##\crcr$\hfill{\lim}\hfil$\crcr
\noalign{\kern1pt\nointerlineskip}\leftarrowfill\crcr\noalign
{\kern -3pt}}}}\limits}
\def\dirlim{\mathop{\vtop{\ialign{##\crcr$\hfill{\lim}\hfil$\crcr
\noalign{\kern1pt\nointerlineskip}\rightarrowfill\crcr\noalign
{\kern -3pt}}}}\limits} 
\def\lomapr#1{\smash{\mathop{\relbar\joinrel\longrightarrow}\limits^{#1}}}

\def\phi{\varphi} 
\def\epsilon{\varepsilon}

\newcommand{\ovk}{\overline{K} }

\newcommand{\dr}{\operatorname{dR} }

 \newcommand{\colim}{\operatorname{colim} }

 \newcommand{\proeet}{\operatorname{pro\acute{e}t} } 
 \newcommand{\eet}{\operatorname{\acute{e}t} }
  \newcommand{\Spa}{\operatorname{Spa} }

 \newcommand{\Hom}{{\rm{Hom}} }
  \newcommand{\uHom}{\underline{{\rm{Hom}} }}
  \newcommand{\Hhom}{{\cal Hom}}

\newcommand{\st}{\operatorname{st} }

 \newcommand{\crr}{\operatorname{cr} }
  
   \newcommand{\LL}{\operatorname{L} }
     \newcommand{\FF}{\operatorname{FF} }

 \newcommand{\sff}{{\mathcal{F}}}

 \newcommand{\sh}{{\mathcal{H}}}
 \newcommand{\sg}{{\mathcal{G}}}
 \newcommand{\stt}{{\mathcal{T}}}
 \newcommand{\sv}{{\mathcal{V}}}
 
 \newcommand{\sbb}{{\mathcal{B}}}
 \newcommand{\scc}{{\mathcal{C}}}

 \newcommand{\sn}{{\mathcal N}}
 \newcommand{\so}{{\mathcal O}}
 
 \newcommand{\se}{{\mathcal{E}}}
 \newcommand{\sa}{{\mathcal{A}}}
 
 \newcommand{\sx}{{\mathcal{X}}}
  
 \newcommand{\sss}{{\mathcal{S}}}
\newcommand{\sd}{{\mathcal{D}}} 
\newcommand{\sm}{{\mathcal{M}}} 
\newcommand{\spp}{{\mathcal{P}}}

 \newcommand{\wt}{\widetilde}
 \newcommand{\wh}{\widehat}
   \numberwithin{equation}{section}

\def\R{{\mathrm R}}

  \def\B{{\bf B}}
\def\Q{{\bf Q}} \def\Z{{\bf Z}}

\def\epsilon{\varepsilon}

\numberwithin{equation}{section}
\begin{document}
\title[Topological Vector Spaces]
 {Topological Vector Spaces}
 \author{Pierre Colmez} 
\address{CNRS, IMJ-PRG, Sorbonne Universit\'e, 4 place Jussieu, 75005 Paris, France}
\email{pierre.colmez@imj-prg.fr} 
\author{Wies{\l}awa Nizio{\l}}
\address{CNRS, IMJ-PRG, Sorbonne Universit\'e, 4 place Jussieu, 75005 Paris, France}
\email{wieslawa.niziol@imj-prg.fr}
 \date{\today}
\thanks{The authors' research was supported in part by the  Simons Collaboration on Perfection in Algebra, Geometry, and Topology and by the grant ANR-19-CE40-0015-02 COLOSS}
\begin{abstract}
 Motivated by applications to duality theorems 
for $p$-adic pro-\'etale cohomology of rigid analytic spaces, we study the category of Topological Vector Spaces in the setting of condensed mathematics.
 We prove full-faithfulness results for  complexes of $\underline{\Q}_p$-sheaves on the pro-\'etale site  and for perfect complexes on the Fargues-Fontaine curve. 
  \end{abstract}
\maketitle
\tableofcontents
\section{Introduction} Let $p$ be a prime and let $K$ be a complete discrete valuation field with a perfect residue field, of mixed characteristic. Let $\ovk$ be an algebraic closure of $K$ and let $C=\wh{\ovk}$ be the $p$-adic completion of $\ovk$.  

 Topological Vector Spaces  were introduced in \cite{CN5} to geometrize $p$-adic period morphisms for rigid analytic varieties, i.e., to make them vary over the category of perfectoid affinoids over $C$.
 Once geometrized these period morphisms acquire rigidity that could be exploited to extract 
actual comparison theorems from the long exact sequence provided by the basic comparison theorems of~\cite{CN4}.
 Topological Vector Spaces were defined as enriched presheaves on sympathetic spaces over $C$ with values in locally convex vector spaces over $\Q_p$ (abelianized appropriately).
 The enrichment condition meant that we consider only presheaves with restriction maps sensitive to the topology on the mapping spaces between sympathetic spaces.

  In this paper we study the condensed version of this category.
 This was mostly motivated  by the application to duality results in \cite{CGN2}, 
where the main examples of Topological Vector Spaces ({\rm TVS}, for short) come from cohomology of  quasi-coherent sheaves on the Fargues-Fontaine curve hence live in the condensed universe.
Another application is the proof~\cite{CGLN}, for partially proper varieties,
 of the $C_{\rm st}$-conjecture from~\cite{CGN3} which allows to reconstruct
the de Rham and Hyodo-Kato cohomologies from the compactly supported $p$-adic pro-\'etale  cohomology.

  In these  applications we need to  retain the topology on
the values of geometrized cohomology objectwise
 so that we can perform objectwise
computations in solid vector spaces:  continuous
duality, solid tensor products  (see \cite[proof of Th. 6.8]{CGN2}, 
\cite[Prop.~5.1]{CGLN}). Moreover, 
for the proof of the $C_{\rm st}$-conjecture in~\cite{CGLN}, one also has to add an action of Galois; this leads
to the notion of ATVS, where ``A'' stands for ``Arithmetic'', which we study in detail in loc.\,cit.

 In the condensed version of TVS, locally convex topological vector spaces are replaced by condensed  (or solid) ${\Q}_p$-modules and the sympathetic spaces are replaced by strictly totally disconnected spaces.
 Roughly, this means that instead of  imposing that the \'etale cohomology of these spaces with values in  $\Q_p$   is trivial in degrees~$\geq 1$  we impose this triviality 
 for  any  \'etale sheaf.
  The enrichment condition puts us in the setting of ${\rm Cond}$-enriched presheaves once we equip mapping spaces with a natural condensed structure. We also call these topologically enriched presheaves.
 The key point of enrichment in this language is the enriched Yoneda Lemma (see Section 3.1.5), which allows us to work homologically with topologically enriched  presheaves in a way akin to algebraic presheaves.
 We note that this fails for non-enriched topological presheaves.

 In \cite{CGN2}, \cite{CGLN}, we need to transport duality and
reconstruction results from the Fargues--Fontaine curve
to TVS and ATVS. This requires
compatibility with derived internal Hom's. It is here that 
topological enrichment plays a key role: The enriched Yoneda Lemma
provides the key tool for proving the fully-faithfulness
and internal-Hom comparisons established below, 
which allow us to combine computations on the
Fargues--Fontaine curve with objectwise solid
functional analysis.

   Topological Vector Spaces are targets of two natural functors: one from Vector Spaces\footnote{That is, $\underline{\Q}_p$-sheaves on ${\rm Perf}_{C, \proeet}$, 
   the pro-\'etale site of perfectoid affinoids over $C$. } ({\rm VS}, for short) and the other one from quasi-coherent sheaves on the Fargues-Fontaine curve. 
 The main result of the paper is the following theorem\footnote{
 Here ${\rm sPerf}_C$ denotes the category of strictly totally disconnected affinoid perfectoid spaces over $C$. We write $\sd(S_{\proeet},\Q_p)$ for the derived $\infty$-category of $\underline{\Q}_p$-sheaves on $S_{\proeet}$, and $\underline{\sd}(S,\Q_p)$ for the derived $\infty$-category of ${\rm Cond}$-enriched presheaves of $\underline{\Q}_p$-modules on ${\rm sPerf}_S$. See Section \ref{lato25S} for the detailed conventions.}:
\begin{theorem}   Let $S\in {\rm sPerf}_C$. 
\begin{enumerate}
\item {\rm(Enriched fully-faithfulness)}  The canonical functor from Vector Spaces to Topological Vector Spaces
$$
\R\pi_*: \sd(S_{\proeet},\Q_p)\to \underline{\sd}(S,\Q_p)
$$
is fully faithful. More precisely,  for   $\sff,\sg\in \sd(S_{\proeet},\Q_p)$,   the canonical   morphism in $\underline{\sd}(S,\Q_p)$
$$
\R\pi_*\R\Hhom_{S,\Q_p}(\sff,\sg)\to \R\Hhom_{S^{\rm top},\Q_p}(\R\pi_*\sff,\R\pi_*\sg)
$$
is a quasi-isomorphism. 
\item  {\rm(Fargues-Fontaine fully-faithfulness)}  The functor from solid quasi-coherent  sheaves on the Fargues-Fontaine curve to Topological Vector Spaces
$$
\R\tau_*: {\rm QCoh}(X_{\FF, S^{\flat}})\to \underline{\sd}(S,\Q_{p})
$$
is fully faithful when restricted to perfect complexes. That is, for $\sff,\sg\in {\rm Perf}(X_{\FF, S^{\flat}})$, the natural morphism in $\underline{\sd}(S,\Q_p)$
$$\R\tau_*\R\Hhom_{{\rm QCoh}(X_{\FF,S^{\flat}})}(\sff,\sg)\to \R\Hhom_{S^{\rm top},\Q_{p}}(\R\tau_*\sff,\R\tau_*\sg)
$$
is a quasi-isomorphism. 
\item  {\rm (Compatibility of the algebraic and topological projections)} The functor from nuclear sheaves on the Fargues-Fontaine curve\footnote{In the sense of Andreychev \cite{And21}.}  to  Vector Spaces
$$
\R\tau^{\prime}_*: {\rm Nuc}(X_{\FF, S^{\flat}})\to {\sd}(S_{\proeet},\Q_{p})
$$
is compatible with the functor $\R\tau_*$ when restricted to nuclear sheaves. That is, the
 following diagram commutes
$$
\xymatrix{
 {\rm Nuc}(X_{\FF,S^{\flat}})\ar[r]^-{\R\tau^{\prime}_*}\ar[rd]^{\R\tau_*} &  \sd(S_{\proeet},{\Q}_p)\ar[d]^{\R\pi_*}\\
  & \underline{\sd}(S,{\Q}_p).
  }
  $$
\end{enumerate}
\end{theorem}
The first claim of the theorem  states that the canonical functor  from Vector Spaces to Topological Vector Spaces is fully faithful (in the enriched sense). This follows from the fact that algebraic Yoneda sheaves associated to strictly totally disconnected spaces map to the topological ones and we have Yoneda Lemmas in both settings. The second claim is a topological analog of a result of Ansch\"utz-Le Bras from \cite{AnLB} and follows from the latter, the third claim of the theorem.

\begin{remark} (1) The above theorem allows us to compute extensions
of Banach--Colmez spaces equivalently on the
Fargues--Fontaine curve, in VS's, or in TVS's.  Hence, in $\underline{\sd}(S,\Q_{p})$:
\begin{align*}
& \R\Hhom_{S^{\rm top},\Q_p}(\Q_p,\Q_p)  \stackrel{\sim}{\leftarrow}\Q_p,\quad \R\Hhom_{S^{\rm top},\Q_p}(\Q_p,{\mathbb G}_a)\stackrel{\sim}{\leftarrow}{\mathbb G}_a,\\
&  \R\Hhom_{S^{\rm top},\Q_p}({\mathbb G}_a,{\mathbb G}_a)  \simeq{\mathbb G}_a\oplus{\mathbb G}_a(-1)[-1],\quad 
  \R\Hhom_{S^{\rm top},\Q_p}({\mathbb G}_a,\Q_p)  \simeq{\mathbb G}_a(-1)[-1].
\end{align*}

 (2) In the applications  in \cite{CGN2} and \cite{CGLN} we work mostly with the solid version of {\rm TVS}'s, which allows us to do computations 
objectwise in the category of solid $\Q_p$-vector spaces. 
 
 (3) There is also a solid version of {\rm VS}'s introduced by Fargues-Scholze  in \cite{FS} (see also \cite{ALBM}). In the Appendix we list some facts about the  relationship between these versions of {\rm VS}'s and {\rm TVS}'s.
\end{remark}
 \subsubsection*{Organization of the paper.}  
We discuss three categories of (pre)sheaves on the category of perfectoid affinoids over $C$ that expand the
category of Banach-Colmez spaces in various directions, by removing finiteness conditions:

-- The {\rm VS}--category (where {\rm VS} stands for Vector Spaces), whose objects are sheaves of 
$\underline{\Q}_p$-vector spaces (thus without topology) on perfectoid affinoids over $C$ equipped with pro-\'etale topology. This topology "encodes" a topology on the {\rm VS}'s. We will also use the term VS for the derived category.

-- The {\rm NTVS}--category (where {\rm NTVS} stands for Naive Topological Vector Spaces), 
whose objects are (pre)sheaves of topological $\underline{\Q}_p$-vector spaces; ("topological" in the sense of Clausen-Scholze) on the category of perfectoid affinoids over $C$ equipped with pro-\'etale topology. There is a canonical projection from the {\rm VS}-category to the {\rm NTVS}-category: the objectwise topology is induced from the topology 
encoded by  the pro-\'etale site.

-- The {\rm TVS}--category (where {\rm TVS} stands for Topological Vector Spaces), which is the  category of topologically enriched 
 presheaves in  {\rm NTVS}'s: we use  the natural topology on the $\Hom$-spaces between perfectoid spaces to impose a continuity property on the presheaf functors. To make this work we need to restrict the presheaf functors from perfectoid affinoids to strictly totally disconnected ones. There is a canonical projection from the {\rm VS}-category to the {\rm TVS}-category: the objectwise topology is induced from the topology 
encoded by  the pro-\'etale site and such objects are canonically enriched.

  The key point to keep in mind is that the categories {\rm VS}  and {\rm TVS} have  Yoneda Lemmas and suitable  generating (projective) Yoneda objects correspond under the canonical projection from {\rm VS} to {\rm TVS}. 
  
  The category {\rm VS} was studied before; here we introduce the categories {\rm NTVS} and {\rm TVS}. This is done at a rather pedestrian pace allowing us to get familiar with these new objects. Section 2 is devoted to {\rm VS}'s and {NTVS}'s, Section 3 to {\rm TVS}'s, and Section 4 to fully-faithfulness.

\begin{acknowledgments}  We are  immensely  grateful to Akhil Mathew for many hours of discussions on the  constructions presented in this paper. 
 We would also like to thank 
Johannes Ansch\"utz for helpful comments on a draft of this paper and  Ko Aoki, Guido Bosco, Juan Esteban Rodriguez Camargo,  Dustin Clausen, David Hansen, Shizhang Li, Zhenghui Li,  Lucas Mann,  Emanuel Reinecke, Peter Scholze, and Bogdan Zavyalov for  discussions concerning the content of this paper. 

Parts of this paper were written during the  authors' stay at the Hausdorff Research Institute for Mathematics in Bonn, in the Summer of 2023, and at IAS at Princeton in the Spring 2024. We would like to thank these Institutes for their support and hospitality. 

The original submission was prepared without AI assistance.
During the revision, GPT-6 (OpenAI) and Claude Opus~5.5
(Anthropic) were used to assist with checking and revising
mathematical arguments, exploring alternative proofs,
and improving the exposition.
\end{acknowledgments}

 \subsubsection*{Notation and conventions.} Let $p$ be a fixed prime, $\Q_p$ -- the $p$-adic rational numbers. 
 Let $K$ be a complete discrete valuation field with a perfect residue field, of mixed characteristic. Let $\ovk$ be an algebraic closure of $K$ and let $C=\wh{\ovk}$ be the $p$-adic completion of $\ovk$.  
We will denote by $ \B_{\crr}, \B_{\st},\B_{\dr}$ the crystalline, semistable, and  de Rham period rings of Fontaine. 

 The category of affinoid perfectoid spaces over an affinoid perfectoid space $S$ over $C$ will be denoted by  ${\rm Perf}_S$. The pro-\'etale site of $S$ will be  the category  ${\rm Perf}_S$ endowed with the quasi-pro-\'etale topology and will be denoted by $S_{\proeet}$.   

   We denote by  ${\rm Cond}$, ${\rm CondAb}$, ${\rm Solid}$ the condensed sets, condensed abelian groups, and solid abelian groups, respectively. To simplify arguments, we fix a cut-off cardinal $\kappa$ and assume that  perfectoid spaces and condensed sets are $\kappa$-small.

\section{Naive Topological Vector Spaces}\label{naive1}  We introduce and study here the category of Naive Topological Vector Spaces. This is a category of topological sheaves, a simpler version of the category of Topological Vector Spaces that we are mostly interested in retaining however many of its properties hence allowing us a gentler introduction to our main object of study. 

   We start with the introduction of the topological pro-\'etale site and examples of sheaves on it. Then we introduce the {\rm NTVS}-category of sheaves with values in condensed and solid abelian groups. We prove that they form Grothendieck abelian categories, study their $\infty$-derived categories, the solidification functors, and monoidal structures. Finally, we extend all of this to the categories of $\underline{\Q}_p$-modules.
\subsection{Topological pro-\'etale site}\label{defnaive} In this section, we introduce  pro-\'etale topological sheaves. 
\subsubsection{Totally disconnected spaces} This brief review is based on \cite[Sec. 7]{SchD} and \cite{MW}. 
\begin{itemize}
\item A perfectoid space $X$ is {\em  totally disconnected}  (see \cite[Def. 7.1, Lemma 7.2, Lemma 7.3]{SchD}) if it is qcqs and every  open cover splits. Equivalently, if all connected components of $X$ are of the form
${\rm Spa}(K,K^+)$ for some perfectoid field $K$ with an open and bounded valuation subring $ K^+\subset K$.
 Also equivalently, 
if for all sheaves of abelian groups $\sff$ on $X$ , one has $H^i(X,\sff)=0$, for $i>0$. 
 By \cite[Lemma 7.5]{SchD}, a totally disconnected space $X$ is always affinoid. 

  A perfectoid space $X$ is {\em w-local} (see \cite[Def. 7.4]{SchD}) if it is  totally disconnected  and the subset $X^c\subset X$ of closed points is closed.
  
 \item  A perfectoid space $X$ is {\em strictly totally disconnected} (see \cite[Def. 7.15]{SchD}) if it is qcqs and every \'etale cover splits. Equivalently, if every connected component is of the form ${\rm Spa}(K,K^+)$, where $K$ is algebraically closed. We note that, for every profinite set \(T\), the associated affinoid perfectoid \(\widehat T\) (defined in Section \ref{naive-521}) is strictly totally disconnected and that, by \cite[Lemma 7.19]{SchD}, the product $X\times \wh{T}$ of a strictly totally disconnected space and a profinite set is strictly totally disconnected. 
 
   More generally (see \cite[Def. 7.17]{SchD}), a {\em w-strictly local } perfectoid space is a w-local perfectoid space $X$ such that for
all $x \in X$, the completed residue field $K(x)$ is algebraically closed.

 \item  A  {\em w-contractible space} (see \cite[Def. 1.1]{MW}) is a w-strictly local perfectoid space $X$ such that $\pi_0(X)$
is extremally disconnected. By \cite[Lemma 1.2]{MW}, if $X$ is a diamond,  the w-contractible spaces in  $X_{\proeet}$ form a basis of
$X_{\proeet}$. Moreover, if $U \in X_{\proeet}$ is w-contractible then  every pro-\'etale cover of $ U$ admits a section. In particular, for any sheaf $\sff$ of abelian groups
on $X_{\proeet}$ we have $H^i(U_{\proeet}, \sff) = 0$, for $i >0$, and $X_{\proeet}$ is replete.
\end{itemize}
\subsubsection{Topological pro-\'etale site} \label{naive-521} Let $S$  be an affinoid perfectoid  over $C$.  Let $S_{\proeet}^{\rm top}:=S_{\proeet}\times *_{\proeet}$ be the topological pro-\'etale site of $S$: its objects are pairs $(Y,T),$ where $Y\in{\rm Perf}_S,$ and $T$ is  a profinite set, and pro-\'etale coverings are given by maps $(Y^{\prime}, T^{\prime})\to (Y,T)$, where $Y^{\prime}\to Y$ is a  pro-\'etale covering and $T^{\prime}\to T$ is a surjective map.  We note that sheaves of sets on $S^{\rm top}_{\proeet}$ form the category ${\rm Sh}(S_{\proeet},{\rm Cond})$ of sheaves on $S_{\proeet}$ with values in ${\rm Cond}$. Similarly for sheaves of abelian groups: we get the category ${\rm Sh}(S_{\proeet},{\rm CondAb})$. 

  We have maps of sites 
$$
S_{\proeet}\lomapr{\pi} S_{\proeet}^{\rm top}\lomapr{\eta} S_{\proeet}
$$
given  by $Y\times \wh{T}\leftmapsto (Y,T)$,
$(Y,*)\leftmapsto Y$, respectively. We note that $\eta\pi={\rm Id}$.  Here $\wh{T}={\rm Spa}(\scc(T,C),\scc(T,\so_C))$ is the adic space (an affinoid perfectoid) associated to $T$. (We will drop the hat from $T$ if this does not cause confusion). 
We define the following pushforward functors $\pi_*: {\rm Sh}(S_{\proeet})\to {\rm Sh}(S^{\rm top}_{\proeet}), \eta_*: {\rm Sh}(S^{\rm top}_{\proeet})\to {\rm Sh}(S_{\proeet})$: 
\begin{align*}
& \pi_*\sff: \{(Y,  {T})\mapsto \sff(Y\times {T})\},\quad \sff\in{\rm Sh}(S_{\proeet});\\
& \eta_*\sff: \{Y\mapsto  \sff(Y,*)\},\quad  \sff\in{\rm Sh}(S^{\rm top}_{\proeet}).
\end{align*}
We have $\eta_*\pi_*={\rm Id}_*$. Thus the functor $\pi_*$ is faithful. 

If $X$ is a topological space, we use the standard notation $\underline{X}$ 
to denote the associated condensed set, i.e., the sheaf
$T\mapsto{\cal C}(T,X)$ on $*_{\proeet}$. 

    \begin{example} \label{ex1} We present some 
  examples of topological sheaves on $S\in {\rm Perf}_C$.
  
 (1)  {\em The constant sheaf $\underline{\Q}_p$.} Recall that the constant sheaf $\underline{\Q}_p$ on $S_{\proeet}$ is defined via 
 $$\underline{\Q}_p: \{Y\mapsto {\scc(|Y|,\Q_p)}\}.$$ One can see that, a priori just a presheaf,  it is actually a sheaf (see  \cite[Lemma 3.14]{MW}). We define the presheaf on $S_{\proeet}^{\rm top}$
 $$\underline{\Q}_p: \{Y\mapsto \underline{\scc(|Y|,\Q_p)}\},$$
 where  $\scc(|Y|,\Q_p)$ is equipped with the compact open topology.
 We claim that  $\pi_*\underline{\Q}_p\simeq\underline{\Q}_p$ (as presheaves for now). Indeed, we compute:
\begin{align*}
& \pi_*\underline{\Q}_p: \{(Y,T)\mapsto \underline{\Q}_p(Y\times {T})=\scc(|Y\times {T}|,\Q_p)\},\\
& \underline{\Q}_p: \{(Y,T)\mapsto \scc(T, {{\scc(|Y|,\Q_p)}})\}.
\end{align*}
But, since $|Y|\times {T}\stackrel{\sim}{\to}|Y\times {T}|$,  the natural map 
$$
\scc(|Y\times {T}|,\Q_p)\to \scc(T, { {\scc(|Y|,\Q_p)}})
$$
is a bijection\footnote{Since $\Q_p$ is Hausdorff, the exponential law of mapping space with compact-open topology holds for locally compact (not necessarily
Hausdorff spaces).}.   It follows that $\underline{\Q}_p$ is a sheaf (since so is $\pi_*\underline{\Q}_p$). 

We note that the value $\underline{\Q}_p(Y)\in{\rm Solid}$. Indeed, if $Y$ is a strictly totally disconnected\footnote{Recall that this means that any {\em \'etale} covering has a section.} perfectoid space then we have (see the proof of \cite[Cor. 3.15]{MW} for an argument)
 $$\underline{\Q}_p(Y)=\underline{\scc(|Y|,\Q_p)}\stackrel{\sim}{\leftarrow} \underline{\scc(\pi_0(Y),\Q_p)},
 $$ a Banach space since $\pi_0(Y)$ is profinite. Our claim now follows because solid abelian groups are closed under taking kernels.

  If $W$ is  a complete Hausdorff  locally convex topological vector space  over $\Q_p$, we can analogously 
  define constant sheaves $\underline{W}$ on $S_{\proeet}^{\rm top}$. 
  We have that $\underline{W}(Y)\in {\rm Solid}$.   In what follows, if this does not cause confusion, we simply write $W$ in place of $\underline{W}$. 
  Similarly, a profinite set $T$ defines a sheaf $\underline{T}$ on $S_{\proeet}$ and on $S_{\proeet}^{\rm top}$; on $S_{\proeet}$ it is represented by the associated perfectoid space $S\times\wh{T}$. 
  
 (2) {\em The sheaf ${\mathbb G}_a$.}  On $S_{\proeet}$ and $S_{\proeet}^{\rm top}$ we have the presheaves 
 $${\mathbb G}_a: \{Y={\rm Spa}(R_Y,R^+_Y)\mapsto {R_Y}\},\quad {\mathbb G}^{\rm top}_a: \{Y={\rm Spa}(R_Y,R^+_Y)\mapsto \underline{R}_Y\},
 $$ respectively,  where $R_Y$ is equipped with its natural  Banach topology.  The first presheaf is a sheaf by \cite[Th. 8.7]{SchD}. 
 We claim that  $\pi_*({\mathbb G}_a)\simeq {\mathbb G}^{\rm top}_a$ (as presheaves). Indeed, we have
 \begin{align*}
 \pi_*{\mathbb G}_a:\{(Y,T)\mapsto {\mathbb G}_a(Y\times {T})\},\quad 
  {{\mathbb G}^{\rm top}_a}: \{(Y,T)\mapsto \scc(T,R_Y)\}.
 \end{align*}
 But, since $R_Y$ is a Banach space, by \cite[Cor. 10.5.4]{PGS},  we have a natural isomorphism
 $$
 \scc(T,R_Y)\stackrel{\sim}{\leftarrow} {\mathbb G}_a(Y\times {T})=\underline{R}_Y{\otimes}^{\Box}_{\Q_p}\scc(T,\Q_p).
 $$
 It follows that ${\mathbb G}^{\rm top}_a$ on $S^{\rm top}_{\proeet}$ is a sheaf (as so is $\pi_*{\mathbb G}_a$). We will skip the superscript $(-)^{\rm top}$ if this does not cause confusion.

  (3) {\em Period sheaves.}  For $i\geq 0$, on $S_{\proeet}^{\rm top}$ we have the presheaves 
 \begin{equation}\label{periods1}
 \mathbb{B}^+_{\dr}/t^i :\{Y\mapsto { \B}^+_{\dr}(Y)/t^i\}, \quad \mathbb{B}^{+,\phi=p^i}_{\crr}: \{Y\mapsto {\B}^{+,\phi=p^i}_{\crr}(Y)\},
 \end{equation}
 where the (objectwise)  period rings are equipped with their canonical topologies. 
  These are sheaves of (solid) abelian groups. Their sheaf
property follows from that of the usual pro-\'etale
period sheaves and the identifications \eqref{warsz10}
below. We have the fundamental exact sequence of sheaves 
  \begin{equation}\label{warsz11}
  0\to \underline{\Q}_p(i)\to\mathbb{B}^{+,\phi=p^i}_{\crr}\to  \mathbb{B}^+_{\dr}/t^i\to0
  \end{equation}
  This sequence is also exact after evaluation at $(Y,T)$
when $Y$ is strictly totally disconnected and $T$ is
profinite, since
$H^1_{\proeet}(Y\times T,\underline{\Q}_p(i))=0$
(see the proof of Lemma~\ref{acyclic1}).

 Similarly, we have period sheaves $\mathbb{B}^+_{\dr}/t^i$ and $\mathbb{B}^{+,\phi=p^i}_{\crr}$ on $S_{\proeet}$ (by forgetting the topology on the period rings in \eqref{periods1}). We claim that there are natural isomorphisms
\begin{equation}\label{warsz10}
\pi_*\mathbb{B}^+_{\dr}/t^i\stackrel{\sim}{\rightarrow}\mathbb{B}^+_{\dr}/t^i, \quad \pi_*\mathbb{B}^{+,\phi=p^i}_{\crr}\stackrel{\sim}{\rightarrow}\mathbb{B}^{+,\phi=p^i}_{\crr}.
 \end{equation}
 Indeed, we have 
 $$
 \pi_*(\mathbb{B}^+_{\dr}/t^i)=\{(Y,T)\mapsto \mathbb{B}^+_{\dr}(Y\times T)/t^i\}, \quad \pi_*\mathbb{B}^{+,\phi=p^i}_{\crr}
 =\{(Y,T)\mapsto \mathbb{B}^+_{\crr}(Y\times T)^{\phi=p^i}\}
 $$
 and the first map in \eqref{warsz10} can be defined as the composition (see \cite[Cor. 6.6]{Sch13}):
 \begin{align*}
  \mathbb{B}^+_{\dr}(Y\times T)/t^i =  \mathbb{B}^+_{\dr}(R_{Y\times T})/t^i \stackrel{\sim}{\to}\mathbb{B}^+_{\dr}(\scc(T, R_Y))/t^i\stackrel{\sim}{\to} \scc(T,  \mathbb{B}^+_{\dr}( R_Y)/t^i)=\scc(T,  \mathbb{B}^+_{\dr}(Y)/t^i).
 \end{align*}
 Here  the third map  is constructed by  going back to the definitions of  period rings; it is seen to be an isomorphism by d\'evissage reducing it to the case of $i=1$, which was shown above. We argue 
  similarly for $\mathbb{B}^{+,\phi=p^i}_{\crr}$.
  \end{example}

  The following lemma shows that  the sheaves $\underline{\Q}_p, {\mathbb G}_a$ and the period sheaves 
  $\mathbb{B}^+_{\dr}/t^i$, $\mathbb{B}^{+,\phi=p^i}_{\crr}$ are acyclic for the functor $\pi_*$:
   \begin{lemma} \label{acyclic1}
Let $\sff\in\{\underline{\Q}_p, {\mathbb G}_a,  \mathbb{B}^+_{\dr}/t^i, \mathbb{B}^{+,\phi=p^i}_{\crr}\}$ be a sheaf on $S_{\proeet}$. We have 
 $\R^i\pi_*\sff=0$, for $i>0$.  
 In particular, we have a  natural quasi-isomorphism 
 $\pi_*\sff\stackrel{\sim}{\to} \R\pi_*\sff$
in $\sd(S_{\proeet},{\rm CondAb})$.
 \end{lemma}
 \begin{proof}
 It suffices to show that 
 \begin{equation}\label{vanish1}
 H^i_{\proeet}(Y\times T,\sff)=0, \quad i>0, 
 \end{equation}
  for a strictly totally disconnected perfectoid space $Y$ over $C$ and a profinite set $T$.  
(We note that $Y\times T$ is a strictly totally disconnected perfectoid.) 

  If $\sff={\mathbb G}_a$ \eqref{vanish1} follows  from Tate acyclicity for perfectoid affinoid spaces for the pro-\'etale topology (see \cite[Prop. 8.8]{SchD}). By d\'evissage it yields the result for
   $\mathbb{B}^+_{\dr}/t^i$. 
  
 If $\sff=\underline{\Q}_p$,   we compute (set $Z:=Y\times T$):
 \begin{align*}
 \R\Gamma_{\proeet}(Z,\underline{\Q}_p) & \stackrel{\sim}{\leftarrow} \R\Gamma_{\proeet}(Z,\underline{\Z}_p)[1/p],\quad 
  \R\Gamma_{\proeet}(Z,\underline{\Z}_p)\stackrel{\sim}{\to} \R\lim_n\R\Gamma_{\proeet}(Z,{\Z}/p^n),\\
 \R\Gamma_{\proeet}(Z,{\Z}/p^n) & \stackrel{\sim}{\leftarrow} \R\Gamma_{\eet}(Z,{\Z}/p^n).
 \end{align*}
 The first quasi-isomorphism holds because $Z$ is quasi-compact; the third follows from \cite[Prop. 14.7, Prop. 14.8]{SchD}. Since $Z$ is strictly totally disconnected, we have $H^i_{\eet}(Z,{\Z}/p^n)=0$ for $i>0$. The transition maps on $H^0_{\eet}(Z,{\Z}/p^n)$ are surjective, so the derived inverse limit is concentrated in degree zero. This yields \eqref{vanish1}. 
 
  Combining the above with the fundamental exact sequence \eqref{warsz11}, we get the claim of the lemma for $\mathbb{B}^{+,\phi=p^i}_{\crr}$. 
 \end{proof}
\subsection{Naive Topological Vector Spaces}  \label{naive52} We will discuss here the topological sheaves on $S_{\proeet}$, for $S\in {\rm Perf}_C$,  in the condensed and solid settings. 

\subsubsection{Condensed and solid modules} \label{naive1a} For the convenience of the reader we will recall  briefly  the basic properties of condensed and solid modules. 

\begin{proposition}{\rm (Clausen-Scholze \cite[Prop. 7.5]{Sch19})} \label{nyear1} Let $A$ be a commutative analytic ring with an underlying condensed ring $\underline{A}$. 
\begin{enumerate}
 \item  The full subcategory of $A$-complete modules (inside the category of condensed $\underline{A}$-modules)
\begin{equation}
 \label{form1}{\rm Mod}^{\rm comp}_{A}\subset  {\rm Mod}^{\rm cond}_{\underline{A }}
 \end{equation}
is a Grothendieck   abelian subcategory,   stable under all limits, colimits, and extensions.  The
inclusion \eqref{form1} admits a left adjoint
\begin{equation}
 \label{form11}
 {\rm Mod}^{\rm cond}_{\underline{A}}\to  {\rm Mod}^{\rm comp}_{A} : \quad M \mapsto M \otimes_{\underline{A}} A, 
 \end{equation}
which   preserves all colimits and is  symmetric monoidal. 
\item  The canonical  functor 
\begin{equation}
\label{form2}
\sd({\rm Mod}^{\rm comp}_{A})
\to  \sd({\rm Mod}^{\rm cond}_{\underline{A}}) 
\end{equation}
is  fully faithful. It preserves all limits and colimits. 
A complex  $M \in \sd({\rm Mod}^{\rm cond}_{\underline{A}})$ is in $\sd({\rm Mod}^{\rm comp}_{A})$ if and only if $H^i(M)$
is in ${\rm Mod}^{\rm comp}_{A}$, for all $i$. The functor \eqref{form2}  admits a left adjoint
\begin{equation}
\sd({\rm Mod}^{\rm cond}_{\underline{A}}) \to \sd({\rm Mod}^{\rm comp}_{A}) : M\mapsto M\otimes^{\rm L}_{\underline{A}}A, 
\end{equation}
which  is the left derived functor of \eqref{form11}. It is symmetric
monoidal.
\item For  $M, N \in \sd({\rm Mod}^{\rm comp}_{A})$, we have  the derived internal Hom $$\R{\uHom}_{A}(M,N)\in\sd({\rm Mod}^{\rm comp}_{A}).$$
The natural map
$\R{\uHom}_{A}(M,N) \to  \R{\uHom}_{\underline{A}}(M,N)$
is a quasi-isomorphism.
\end{enumerate}
\end{proposition}
The tensor products in the proposition are obtained by completing the corresponding condensed tensor products. Thus, for $A$-complete modules $M,N$, and analogously for complexes, they are given by
\begin{align*}
 M\otimes_A N&=(M\otimes_{\underline{A}}N)\otimes_{\underline{A}}A,\\
 M\otimes_A^{\LL}N&=(M\otimes_{\underline{A}}^{\LL}N)\otimes_{\underline{A}}^{\LL}A.
\end{align*}
\begin{remark}
For the solid analytic rings used below, we write ${\rm Mod}^{\rm solid}_A$ for ${\rm Mod}^{\rm comp}_A$ and call completion solidification. In this setting, we set $M^{\Box}:=M\otimes_{\underline{A}}A$, for $M\in{\rm Mod}^{\rm cond}_{\underline{A}}$, and $M^{\LL_{\Box}}:=M\otimes^{\LL}_{\underline{A}}A$, for 
$M\in\sd({\rm Mod}^{\rm cond}_{\underline{A}})$.
\end{remark}
 \subsubsection{Condensed sheaves} \label{cond-kol1} We list here properties of topological sheaves with values in condensed abelian groups. 
 
 \vskip2mm
  ($\bullet$) {\em Sheaves of condensed abelian groups.} Consider first the category ${\rm Sh}(S_{\proeet},{\rm CondAb})\simeq {\rm Sh}(S^{\rm top}_{\proeet},{\rm Ab})$ of sheaves with values in condensed abelian groups.  It is a Grothendieck abelian category, which inherits
 a closed symmetric monoidal structure from that of ${\rm CondAb}$:  For  $\sff,\sg\in {\rm Sh}(S_{\proeet},{\rm CondAb})$, their tensor product  $\sff\otimes \sg$ is defined by sheafifying the presheaf tensor product
$$
\{Y\mapsto \sff(Y)\otimes \sg(Y)\in {\rm CondAb}\},\quad Y\in S_{\proeet}. 
$$
The internal $\Hom$, $\Hhom_{S^{\rm top}}(\sff,\sg)$, can be defined as a presheaf by
$$
\{Y\mapsto  \uHom_{Y^{\rm top}}(\sff_Y,\sg_Y)\in {\rm CondAb}\},\quad Y\in S_{\proeet},
$$
where we set 
\begin{equation}\label{coend1}
\uHom_{Y^{\rm top}}(\sff_Y,\sg_Y)
:=\ker\Big(\prod_{Y_1\to Y}\uHom(\sff(Y_1),\sg(Y_1))\to \prod_{Y_2\to Y_1}\uHom(\sff(Y_1),\sg(Y_2))\Big).
\end{equation}
For an arrow $f:Y_2\to Y_1$ over $Y$, the corresponding component of the arrow in \eqref{coend1} sends a family $(u_{Y_1})$ to
$$
\sg(f)\circ u_{Y_1}-u_{Y_2}\circ\sff(f).
$$
Thus its kernel imposes compatibility with all restriction maps. It is actually a sheaf (see \cite[Prop. 2.2.14]{Schn}). The internal tensor--$\Hom$ adjunction \cite[Prop. 2.2.16]{Schn} is the natural isomorphism
$$
\Hhom_{S^{\rm top}}(\sff_1\otimes\sff_2,\sg)
\simeq \Hhom_{S^{\rm top}}(\sff_1,\Hhom_{S^{\rm top}}(\sff_2,\sg)),
$$
for $\sff_1,\sff_2,\sg\in {\rm Sh}(S_{\proeet},{\rm CondAb})$.

  By Remark \ref{koczkodan1} and  Lemma \ref{trump1} below,  the category ${\rm Sh}(S_{\proeet},{\rm CondAb})$  is generated by the set 
\begin{equation}\label{pierre10}
\{\Z[h^{\delta}_Y]\otimes\Z[T]\}, \qquad Y\in {\rm Perf}_S,\quad T\text{ profinite}.
\end{equation}
Here the sheaf $h^{\delta}_Y$ is defined by $Y_1\mapsto \Hom_S(Y_1,Y)$, where the $\Hom$ is given discrete topology. 
\begin{lemma}\label{koczkodan2}
The set  \eqref{pierre10} for  $w$-contractible $Y$'s and  extremally disconnected $T$'s  generates ${\rm Sh}(S_{\proeet},{\rm CondAb})$.
Moreover, these generators are   compact and projective. 
\end{lemma}
\begin{proof}
Take  $\Z[h^{\delta}_Y]\otimes \Z[T]$ as in \eqref{pierre10}. For compactness we need to show that
 the functor  $\Hom_{S^{\rm top}}(\Z[h^{\delta}_Y]\otimes \Z[T], -)$ commutes with filtered colimits. But
  \begin{align}\label{niedziela2}
\Hom_{S^{\rm top}}(\Z[h^{\delta}_Y]\otimes \Z[T], \sff) & \simeq \Hom_{S^{\rm top}}(\Z[h^{\delta}_Y], [\Z[T],\sff])\simeq [\Z[T],\sff](Y)(*)\simeq [\Z[T],\sff(Y)](*)\\
 & \simeq \Hom(\Z[T],\sff(Y)),\notag
  \end{align}
  where we wrote $[\Z[T],\sff]$ for the sheaf $Y_1\mapsto \uHom(\Z[T],\sff(Y_1))$ and  the second isomorphism follows from the enriched Yoneda Lemma (see Section \ref{enrichedyoneda}). But
   $\Gamma(Y,-)$ commutes with colimits since  both $Y$  and $\Z[T]$ are  quasi-compact. 
  
    For projectivity, take a surjection $\sff_1\to \sff_2\to 0$. We need to show that the induced map 
    $$
    \Hom_{S^{\rm top}}(\Z[h^{\delta}_Y]\otimes \Z[T], \sff_1)\to \Hom_{S^{\rm top}}(\Z[h^{\delta}_Y]\otimes \Z[T], \sff_2)
    $$
    is surjective as well. But, using \eqref{niedziela2}, this map can be written as 
    $$
    \Hom(\Z[T],\sff_1(Y))\to \Hom(\Z[T],\sff_2(Y))
    $$
    and this is surjective by projectivity of $\Z[T]$ since the map $\sff_1(Y)\to \sff_2(Y)$ is surjective
as $Y$ is $w$-contractible. 
\end{proof}
 \vskip2mm
   ($\bullet$) {\em Sheaves of condensed $\underline{\Q}_p$-modules.} We use the ring sheaf $\underline{\Q}_p$ from Example \ref{ex1}(1), whose value at $(Y,T)$ is $\scc(|Y|\times T,\Q_p)$. 
We denote by ${\rm Sh}(S_{\proeet},{\rm Mod}^{\rm cond}_{\underline{\Q}_p})$ the category of modules over this ring object in ${\rm Sh}(S_{\proeet},{\rm CondAb})$.  It is again a Grothendieck abelian category (as the category of modules on a ringed site). Since we have ${\rm Sh}(S_{\proeet},{\rm Mod}^{\rm cond}_{\underline{\Q}_p})\simeq {\rm Sh}(S^{\rm top}_{\proeet},\underline{\Q}_p)$,  the classical theory of tensor products and internal $\Hom$s of sheaves on ringed sites (see \cite[Tag 03A4]{Stck}) yields
the condensed tensor product and internal $\Hom$: For  $\sff,\sg\in {\rm Sh}(S_{\proeet},{\rm Mod}^{\rm cond}_{\underline{\Q}_p})$ their tensor product  $\sff\otimes_{\underline{\Q}_p}\sg$ is defined by sheafifying the presheaf tensor product
$$
\{Y\mapsto \sff(Y)\otimes_{\underline{\Q}_p(Y)}\sg(Y)\in {\rm Mod}^{\rm cond}_{\Q_p(Y)}\},\quad Y\in S_{\proeet}. 
$$
The internal $\Hom$, $\Hhom_{S^{\rm top},\underline{\Q}_p}(\sff,\sg)$, can be defined  as a presheaf by
$$
\{Y\mapsto  \uHom_{Y^{\rm top},\underline{\Q}_p}(\sff_Y,\sg_Y)\in {\rm Mod}^{\rm cond}_{\Q_p(Y)}\},\quad Y\in S_{\proeet},
$$
where we set 
\begin{equation}
\uHom_{Y^{\rm top},\underline{\Q}_{p}}(\sff_Y,\sg_Y)
:=\ker\Big(\prod_{Y_1\to Y}\uHom_{\Q_p(Y_1)}(\sff(Y_1),\sg(Y_1))\to \prod_{Y_2\to Y_1}\uHom_{\Q_p(Y_1)}(\sff(Y_1),\sg(Y_2))\Big).
\end{equation}
It is actually a sheaf \cite[Tag 03EM]{Stck} and it is the (internal)  right adjoint of the tensor product \cite[Tag 03EO]{Stck}.

  The category ${\rm Sh}(S_{\proeet},{\rm Mod}^{\rm cond}_{\underline{\Q}_p})$  is generated by the set 
$$
\{\underline{\Q}_p[h^{\delta}_Y]\otimes_{\underline{\Q}_p}\underline{\Q}_p[T]\}, \quad Y\in {\rm Perf}_S, 
\ T - \mbox{ profinite set}.
$$
And we have an analog of Lemma \ref{koczkodan2} in this setting.

 \vskip2mm
 ($\bullet$) {\em Derived picture.} We have an analogous derived picture, which, again, follows from the classical story (see \cite[01FQ]{Stck}). Namely, the existence of $K$-flat resolutions\footnote{Recall that $K$-flatness, and the related notion of $K$-injectivity, for unbounded complexes (originally of modules over a ring $R$) were introduced by Spaltenstein  in \cite{Spal} as an analogue of the classical notion of flatness (and injectivity) of modules. }  yields the derived tensor product 
  $\sff\otimes^{\LL}_{\underline{\Q}_p}\sg$, for $\sff, \sg\in\sd(S_{\proeet},{\rm Mod}^{\rm cond}_{\underline{\Q}_p})$ (see \cite[Tag 06YU]{Stck}). The internal $\Hom$, $\R\Hhom_{S^{\rm top},\underline{\Q}_p}(\sff,\sg)$, is defined using the existence of $K$-injective resolutions. It is the right (internal) adjoint of the derived tensor product  \cite[Tag 08J9]{Stck}.

 We have analogous  definitions and properties for sheaves with values in ${\rm CondAb}$. 
 
\begin{remark}   
For $\sff\in  \sd(S_{\proeet},{\rm Ab})$,   we have in $\sd(\rm Ab)$
$$
\R\Gamma(T, \R\Gamma(Y^{\rm top}_{\proeet},\R\pi_*\sff))\simeq \R\Gamma_{\proeet}(Y\times T,\sff),
$$
where, for a profinite set $T$ and $K\in\sd({\rm CondAb})$, we set $\R\Gamma(T,K):=\R\Hom_{{\rm CondAb}}(\Z[T],K)$.
This can be easily seen for $Y$, which are w-contractible and by pro-\'etale descent in $Y$-variable in general. 
In particular, $\R\Gamma(Y^{\rm top}_{\proeet},\R\pi_*\sff)(*)\simeq \R\Gamma_{\proeet}(Y,\sff)$. Similarly for $\underline{\Q}_p$-sheaves.
\end{remark}
\subsubsection{Solid sheaves} \label{solid1} We pass now to topological sheaves with values in solid abelian groups. 

  \vskip2mm
   ($\bullet$) {\em Sheaves of solid  abelian groups.} Recall (see Proposition \ref{nyear1}) that the category ${\rm Solid}$ of solid abelian groups is a Grothendieck abelian category, which is stable under limits, colimits, and extensions in ${\rm CondAb}$. It is also compactly generated.  It follows that  the category
${\rm Sh}(S_{\proeet},{\rm Solid})$ behaves as expected: it is a Grothendieck abelian category.  Moreover, it is generated by the set 
\begin{equation}
\{\Z[h^{\delta}_Y]\otimes^{\Box}\Z[T]\}, \quad Y\in {\rm Perf}_S,\  T - \mbox{ profinite set}.
\end{equation}
And we have an analog of Lemma \ref{koczkodan2} in this setting. 

  For $\sff,\sg\in  {\rm Sh}(S_{\proeet},{\rm Solid})$, we define $\uHom^{\Box}_{S}(\sff,\sg)\in {\rm Solid}$ via  the  end definition as in \eqref{coend1}. The set  $\uHom^{\Box}_{S}(\sff,\sg)(*)$ is the usual $\Hom$-set
  in the category $ {\rm Sh}(S_{\proeet},{\rm Solid})$.  The two categories of topological sheaves ${\rm Sh}(S_{\proeet},{\rm CondAb})$ and $ {\rm Sh}(S_{\proeet},{\rm Solid})$  are related: 

\begin{lemma}\label{solidification1}
 The canonical forgetful functor ${\rm Sh}(S_{\proeet},{\rm Solid})\to {\rm Sh}(S_{\proeet},{\rm CondAb})$ is fully faithful as a ${\rm Cond}$-enriched functor and has a ${\rm Cond}$-enriched left adjoint $\sff\mapsto \sff^{\Box}$ given by the sheafification of the presheaf\footnote{The sheafification is necessary because the solidification functor $\Box$ is not left exact.}:
 $$
{\sff}^{\Box,{\rm psh}}:\quad  \{Y\mapsto \sff(Y)^{\Box}\},\quad Y\in S_{\proeet}.
 $$
 It commutes with all colimits. 
 \end{lemma}
 \begin{proof}For the first claim, we need to show that, for $\sff, \sg \in {\rm Sh}(S_{\proeet},{\rm Solid})$,  the canonical map in ${\rm CondAb}$
 $$
 \uHom^{\Box}_{S}(\sff,\sg)\to  \uHom_{S^{\rm top}}(\sff,\sg)
 $$
 is an isomorphism. Using the end definition of both sides, this follows from full faithfulness on condensed internal $\Hom$'s and from the stability of ${\rm Solid}$ under all limits in ${\rm CondAb}$, since the end is a kernel of a map between products.
 
    For the second claim, note that, from the definitions, we get a natural map (of presheaves on $S_{\proeet}$ with values in ${\rm CondAb}$):
 $$
 \sff\to \sff^{\Box}. 
 $$
 We need to show that, for $\sff\in {\rm Sh}(S_{\proeet},{\rm CondAb})$,
$\sg \in {\rm Sh}(S_{\proeet},{\rm Solid})$, the induced morphism in ${\rm CondAb}$
 $$
\uHom^{\Box}_{S}(\sff^{\Box},\sg)\stackrel{\sim}{\to}  \uHom_{S^{\rm top}}(\sff^{\Box},\sg) \to \uHom_{S^{\rm top}}(\sff,\sg) 
 $$
 is an isomorphism. 
 By adjointness of the sheafification functor, we can pass from sheaves to presheaves and  use the presheaf ${\sff}^{\Box,{\rm psh}}$ in place of $\sff^{\Box}$. Using the end definition, this reduces to showing that
 $$
  \uHom^{\Box}({\sff}(Y_1)^{\Box},\sg(Y_2))\simeq \uHom(\sff(Y_1),\sg(Y_2)),\quad Y_1, Y_2\in S_{\proeet},
 $$
 which is clear. 
  \end{proof}

    The category $ {\rm Sh}(S_{\proeet},{\rm Solid})$ inherits
 a closed symmetric monoidal structure from that of ${\rm Solid}$:  For  presheaves $\sff,\sg\in {\rm PSh}(S_{\proeet},{\rm Solid})$,  their tensor product  $\sff\otimes^{\Box, {\rm psh}} \sg$ is defined as  the presheaf 
$$
\{Y\mapsto \sff(Y)\otimes^{\Box} \sg(Y)\in {\rm Solid}\},\quad Y\in S_{\proeet}. 
$$
For  sheaves $\sff,\sg\in {\rm Sh}(S_{\proeet},{\rm Solid})$,  their tensor product  $\sff\otimes^{\Box} \sg$ is defined as the sheafification of the presheaf tensor product $\sff\otimes^{\Box, {\rm psh}} \sg$: 
$$\sff\otimes^{\Box} \sg:=(\sff\otimes^{\Box, {\rm psh}} \sg)^{\sharp}.
$$
We have the expected relations: 
\begin{lemma}\label{tensor} 
\begin{enumerate}
\item For $\sff,\sg \in  {\rm PSh}(S_{\proeet},{\rm Solid})$,  we have a canonical isomorphism in ${\rm Sh}(S_{\proeet},{\rm Solid})$
$$\sff^{\sharp} \otimes^{\Box} \sg^{\sharp}\simeq  (\sff \otimes^{\Box,{\rm psh}}\sg)^{\sharp}.$$
\item  For $\sff,\sg \in  {\rm PSh}(S_{\proeet},{\rm CondAb})$,  we have a canonical isomorphism in ${\rm Sh}(S_{\proeet},{\rm Solid})$
$$(\sff^{\sharp} \otimes\sg^{\sharp})^{\Box}\simeq  ((\sff \otimes^{\rm psh}\sg)^{\Box,{\rm psh}})^{\sharp}.$$
\end{enumerate}
\end{lemma}
\begin{proof} We will reduce to universal properties of the terms of the isomorphisms. 
For $\sh \in {\rm Sh}(S_{\proeet},{\rm Solid})$, we have 
\begin{align*}
 & \Hom^{\Box}_{S}((\sff\otimes^{\Box,{\rm psh}}\sg)^{\sharp},\sh)\simeq \Hom_{\rm psh}(\sff \otimes^{\Box,{\rm psh}}\sg,\sh)\simeq 
 \Hom_{\rm psh}(\sff,\Hhom_{S^{\rm top}}(\sg,\sh))\\
  & \simeq  \Hom_{\rm psh}(\sff^\sharp,\Hhom_{S^{\rm top}}(\sg,\sh))\simeq  \Hom_{\rm psh}(\sff^\sharp\otimes^{\Box,{\rm psh}}\sg,\sh).
\end{align*}
Here we wrote  $\Hom_{\rm psh}(-,-) $ for the sets of morphisms in  ${\rm PSh}(S_{\proeet},{\rm Solid})$. 
Repeating this computation with $\sg$ in place of $\sff$ yields claim (1) of the lemma.

 The proof of the second   claim of the lemma is similar. 
\end{proof}
The solid tensor product of sheaves  inherits many properties from the solid tensor product of modules. In particular, it commutes with all colimits (since the sheafification functor does).
The internal $\Hom$, $\Hhom^{\Box}_{S}(\sff,\sg)$, can be defined as a presheaf by 
$$
\{Y\mapsto  \uHom^{\Box}_{Y}(\sff_Y,\sg_Y)\in {\rm Solid}\},\quad Y\in S_{\proeet}.
$$
It is actually a sheaf by \cite[Prop. 2.2.14]{Schn}.  Moreover, the image of the sheaf $\Hhom^{\Box}_{S}(\sff,\sg)$ by the embedding functor 
${\rm Sh}(S_{\proeet},{\rm Solid})\hookrightarrow  {\rm Sh}(S_{\proeet},{\rm CondAb})$ is the sheaf $\Hhom_{S^{\rm top}}(\sff,\sg)$.
We have the usual adjointness property:
\begin{lemma}\label{adj1}If $\sff_1,\sff_2,\sff_3\in {\rm Sh}(S_{\proeet},{\rm Solid})$, there is a canonical functorial isomorphism in ${\rm Sh}(S_{\proeet},{\rm Solid})$
$$
\Hhom^{\Box}_{S}(\sff_1\otimes^{\Box}\sff_2,\sff_3)\stackrel{\sim}{\to} \Hhom^{\Box}_{S}(\sff_1,\Hhom^{\Box}_{S}(\sff_2,\sff_3)).
$$
In particular, we have a canonical functorial isomorphism in ${\rm Solid}$
$$
\Hom^{\Box}_{S}(\sff_1\otimes^{\Box}\sff_2,\sff_3)\simeq \Hom^{\Box}_{S}(\sff_1,\Hhom^{\Box}_{S}(\sff_2,\sff_3)).
$$
\end{lemma}
\begin{proof}
This can be checked in $ {\rm Sh}(S_{\proeet},{\rm CondAb})$, where it is known. 
 \end{proof}

    ($\bullet$) {\em Sheaves of solid  $\underline{\Q}_p$-modules.} Let    ${\rm Sh}(S_{\proeet},\underline{\Q}_{p,\Box})$ denote the subcategory of ${\rm Sh}(S_{\proeet},{\rm Solid})$ of $\underline{\Q}_{p}$-modules. 
    \begin{lemma}
    The category ${\rm Sh}(S_{\proeet},\underline{\Q}_{p,\Box})$ is a Grothendieck abelian category.
    \end{lemma}
    \begin{proof} We need to check that the category ${\rm Sh}(S_{\proeet},\underline{\Q}_{p,\Box})$ admits a generator and small colimits such that small filtered colimits are exact. 
    But the colimits can be inherited from ${\rm Sh}(S_{\proeet},{\rm Solid})$ and they have the required properties (use the fact that the solid tensor product of solid abelian sheaves commutes with direct sums). Moreover,
 if $X$ is a generator of ${\rm Sh}(S_{\proeet},{\rm Solid})$ then $\underline{\Q}_p\otimes^{\Box}X$ is a generator of ${\rm Sh}(S_{\proeet},\underline{\Q}_{p,\Box})$ as easily follows, again,  from the fact that the solid tensor product of solid abelian sheaves commutes with direct sums.   This finishes our argument. 
  \end{proof}
     The categories of $\Q_p(Y)_{\Box}$-modules, $ Y\in S_{\proeet}$, being closed symmetric monoidal, the category ${\rm Sh}(S_{\proeet},\underline{\Q}_{p,\Box})$ is equipped with tensor product and internal $\Hom$:
For  $\sff,\sg\in {\rm Sh}(S_{\proeet},{\underline{\Q}_{p,\Box}})$, their tensor product $\sff\otimes^{\Box}_{\underline{\Q}_p}\sg$  is defined by sheafifying the presheaf tensor product
$$
\{Y\mapsto \sff(Y)\otimes^{\Box}_{\underline{\Q}_{p}(Y)}\sg(Y)\},\quad Y\in S_{\proeet}.
$$
We have an analog of Lemma \ref{tensor} in this setting. 
The internal $\Hom$,  $\Hhom^{\Box}_{S,\underline{\Q}_{p}}(\sff,\sg)$, can be defined as a presheaf by 
$$
\{Y\mapsto  \uHom^{\Box}_{Y, \underline{\Q}_{p}}(\sff_Y,\sg_Y)\in {\rm Mod}_{\Q_p(Y)}^{\rm solid}\},\quad Y\in S_{\proeet},
$$
where we set
\begin{equation}
\uHom^{\Box}_{Y,\underline{\Q}_{p}}(\sff_Y,\sg_Y)
:=\ker\Big(\prod_{Y_1\to Y}\uHom^{\Box}_{\Q_p(Y_1)}(\sff(Y_1),\sg(Y_1))\to \prod_{Y_2\to Y_1}\uHom^{\Box}_{\Q_p(Y_1)}(\sff(Y_1),\sg(Y_2))\Big).
\end{equation}
 It is actually a sheaf (compare with the condensed definition). Moreover, the image of the sheaf $\Hhom^{\Box}_{S,\underline{\Q}_{p}}(\sff,\sg)$ by the forgetful functor ${\rm Sh}(S_{\proeet},\underline{\Q}_{p,\Box})\to 
 {\rm Sh}(S^{\rm top}_{\proeet},\underline{\Q}_{p})$ is the sheaf $\Hhom_{S^{\rm top},\underline{\Q}_{p}}(\sff,\sg)$. 
In particular,  the  forgetful functor ${\rm Sh}(S_{\proeet},\underline{\Q}_{p,\Box})\to {\rm Sh}(S^{\rm top}_{\proeet},\underline{\Q}_{p})$ is (topologically) fully faithful. 
We have the usual adjointness property: 
\begin{lemma}
If $\sff_1,\sff_2,\sff_3\in {\rm Sh}(S_{\proeet},\underline{\Q}_{p,\Box})$, there is a canonical functorial isomorphism in ${\rm Sh}(S_{\proeet},\underline{\Q}_{p,\Box})$
$$
\Hhom^{\Box}_{S,\underline{\Q}_{p}}(\sff_1\otimes^{\Box}_{\underline{\Q}_p}\sff_2, \sff_3)\stackrel{\sim}{\to} \Hhom^{\Box}_{S,\underline{\Q}_{p}}(\sff_1,\Hhom^{\Box}_{S,\underline{\Q}_{p}}(\sff_2, \sff_3)).
$$
In particular, we have a canonical functorial isomorphism in ${\rm Mod}^{\rm solid}_{\Q_p(S)}$
$$
\uHom^{\Box}_{S,\underline{\Q}_{p}}(\sff_1\otimes^{\Box}_{\underline{\Q}_p}\sff_2, \sff_3)\simeq \uHom^{\Box}_{S,\underline{\Q}_{p}}(\sff_1,\Hhom^{\Box}_{S,\underline{\Q}_{p}}(\sff_2, \sff_3)).
$$
 \end{lemma}
 \begin{proof}Reduce to the condensed adjointness as in the proof of Lemma \ref{adj1}.
 \end{proof}

 The category $ {\rm Sh}(S_{\proeet},\underline{\Q}_{p,\Box})$  is generated by the set 
\begin{equation}
\{\underline{\Q_p}[h^{\delta}_Y]\otimes^{\Box}_{\Z}\Z[T]\}, \quad Y\in {\rm Perf}_S, T - \mbox{ profinite set}.
\end{equation}
And we have an analog of Lemma \ref{koczkodan2} in this setting. 

  \vskip3mm
    ($\bullet$) {\em Derived picture.} 
The categories ${\rm Sh}(S_{\proeet},{\rm Solid})$ and
${\rm Sh}(S_{\proeet},\underline{\Q}_{p,\Box})$ are
Grothendieck abelian and therefore admit $K$-injective
resolutions (see \cite{Ser}, \cite[Tag~079I]{Stck}).

We equip their unbounded derived categories with the
closed symmetric monoidal structures obtained by derived
solidification of the corresponding condensed structures,
as explained in Proposition~\ref{nyear1a} below.
This is the sheaf analogue of the construction in
\cite[Th.~6.2]{Sch19}.
Explicitly, for
$\sff,\sg\in\sd(S_{\proeet},{\rm Solid})$, we have
\[
\sff\otimes^{\LL_{\Box}}\sg
 \simeq
\bigl(\sff\otimes^{\LL}\sg\bigr)^{\LL_{\Box}},
\]
where the tensor product inside parentheses is formed in
$\sd(S_{\proeet},{\rm CondAb})$ and
$(-)^{\LL_{\Box}}$ denotes derived solidification of
sheaves.

Similarly, for
$\sff,\sg\in\sd(S_{\proeet},\underline{\Q}_{p,\Box})$,
we have
\[
\sff\otimes^{\LL_{\Box}}_{\underline{\Q}_p}\sg
 \simeq
\bigl(
 \sff\otimes^{\LL}_{\underline{\Q}_p}\sg
\bigr)^{\LL_{\Box}},
\]
with the tensor product inside parentheses formed in
the derived category of condensed
$\underline{\Q}_p$-module sheaves.

We denote the corresponding derived internal Hom's by
$\R\Hhom^{\Box}_S$ and
$\R\Hhom^{\Box}_{S,\underline{\Q}_p}$, respectively.
They are the internal right adjoints to the derived
tensor products above. By Proposition~\ref{nyear1a},
after forgetting solidity they identify with the
corresponding condensed derived internal Hom's.
In particular, for
$\sff_1,\sff_2,\sg\in\sd(S_{\proeet},{\rm Solid})$,
we have the internal tensor--Hom adjunction
\[
\R\Hhom^{\Box}_S
 \bigl(\sff_1\otimes^{\LL_{\Box}}\sff_2,\sg\bigr)
 \simeq
\R\Hhom^{\Box}_S
 \bigl(\sff_1,\R\Hhom^{\Box}_S(\sff_2,\sg)\bigr),
\]
and analogously with $\underline{\Q}_p$-coefficients.

\subsubsection{Constant sheaves and monoidal structures}
Let $W\in K_{\Box}$ be a flat module and let\footnote{The definition and properties of the category ${\rm Sh}(S_{\proeet}, K_{\Box})$ is analogous to those of the category 
$ {\rm Sh}(S_{\proeet}, \Q_{p,\Box})$.}  $\sff\in {\rm Sh}(S_{\proeet}, K_{\Box})$.  We define  sheaves $  W\otimes^{\LL_{\Box}}_{K}\sff,   W\otimes^{\Box}_{K}\sff$  as presheaves by setting
  \begin{align*}
  W\otimes^{\LL_{\Box}}_{K}\sff:=\{Y\mapsto W\otimes^{\LL_{\Box}}_{K}\sff(Y)\},\\
    W\otimes^{\Box}_{K}\sff:=\{Y\mapsto W\otimes^{\Box}_{K}\sff(Y)\}.
  \end{align*}
  Flatness of $W$ gives a canonical quasi-isomorphism
$$
W\otimes^{\LL_{\Box}}_K\sff\stackrel{\sim}{\to}W\otimes^{\Box}_K\sff.
$$
The common degree-zero presheaf is a sheaf because tensoring with $W$ preserves the finite equalizers expressing the sheaf condition for finite covering families. We will need the following computations:
\begin{lemma}\label{form23} 
Let $W\in \Q_{p,\Box}$ be a Fr\'echet space or a space of compact type. Let $\sff\in {\rm  Sh}(S_{\proeet}, \underline{\Q}_{p,\Box} )$.
\begin{enumerate}

\item    There is a canonical quasi-isomorphism in $\sd(S_{\proeet},\Q_{p,\Box})$
 \begin{equation}\label{form11-1}
  W\otimes^{\LL_{\Box}}_{\Q_{p}}\sff\stackrel{\sim}{\to}  \underline{W}\otimes^{\LL_{\Box}}_{\underline{\Q}_{p}}\sff.
 \end{equation}
 
 \item   There is a natural isomorphism in $\Q_{p,\Box}$
$$\uHom^{\Box}_{S,\Q_p}(\underline{W} ,\sff) \stackrel{\sim}{\to}  \uHom^{\Box}_{ \Q_p} (W,\sff(S)).$$

\item   There is a natural quasi-isomorphism in $\sd(\Q_{p,\Box})$
$$\R\uHom^{\Box}_{S,\Q_p}(\underline{W},\sff) \simeq  \R\uHom_{\Q_{p,\Box}} (W,\R\Gamma(S^{\rm top}_{\proeet},\sff)).
$$
\item   Assume that \(W\) is of compact type and that
\(\mathrm R\Gamma(S_{\mathrm{pro\acute et}}^{\mathrm{top}},\mathcal F)\) is quasi-isomorphic, as a solid \(\mathbf Q_p\)-complex, to a complex of Fréchet spaces. Then  there is a natural quasi-isomorphism in $\sd(\Q_{p,\Box})$
$$\R\uHom^{\Box}_{S,\Q_p}(\underline{W},\sff) \simeq  W^*\otimes^{\LL_{\Box}}_{\Q_p}\R\Gamma(S^{\rm top}_{\proeet},\sff).
$$
\end{enumerate}
\end{lemma}
\begin{proof}  We start with the first claim. As presheaves both sides of \eqref{form11-1} are given by 
  \begin{align*}
  \{Y\mapsto W\otimes^{\LL_{\Box}}_{\Q_{p}}\sff(Y)\},\quad 
    \{Y\mapsto \underline{W}(Y)\otimes^{\LL_{\Box}}_{\Q_{p}(Y)}\sff(Y)\},
  \end{align*}
  respectively. 
  For $Y$ strictly totally disconnected,  we compute 
  \begin{align*}
  \underline{W}(Y)\otimes^{\LL_{\Box}}_{\Q_{p}(Y)}\sff(Y) & =\scc(\pi_0(Y),W)\otimes^{\LL_{\Box}}_{\Q_{p}(Y)}\sff(Y)\simeq W{\otimes}^{{\Box}}_{\Q_{p}}\scc(\pi_0(Y),\Q_p)\otimes^{\LL_{\Box}}_{\Q_{p}(Y)}\sff(Y)\\
  &  \simeq W\otimes^{\LL_{\Box}}_{\Q_{p}}\sff(Y)
  \end{align*}
   Here, the second quasi-isomorphism holds by \cite[Cor. 10.5.4]{PGS}. Our claim follows.
  
  For the second claim, we use the  end description
  \begin{align*}
  \uHom^{\Box}_{S,\Q_p}(\underline{W} ,\sff)  & \stackrel{\sim}{\to}  \uHom^{\Box}_{S,\Q_p}({W} \otimes^{\Box}_{\Q_{p}}\underline{\Q}_p,\sff)  = \int_{Y}  \uHom_{\Q_p(Y)}(W\otimes^{\Box}_{\Q_{p}}\Q_p(Y) ,\sff(Y))\\
   & \simeq \int_{Y} \uHom_{\Q_p}(W ,\sff(Y))\simeq  \uHom_{\Q_p}(W ,\int_{Y} \sff(Y))\simeq  \uHom_{\Q_p} (W,\sff(S)).
  \end{align*}
  Here, the first isomorphism follows from the first claim of the lemma.

  The third claim  of the lemma follows from the second one by applying  the latter to an injective resolution of $\sff$. For the fourth claim, recall that, 
for a Fr\'echet space \(V\), we have (see \cite[Th. 3.40]{RR}) 
\[
\underline{\operatorname{Hom}}_{\mathbf Q_p}(W,V)
\simeq W^*\otimes_{\mathbf Q_p}^{\Box}V.
\]
For the derived version, write \(W=\colim_n W_n\) using Smith spaces, with the dual Banach transition maps having dense image. Projectivity, topological Mittag–Leffler, and compatibility of Fr\'echet tensor products with these limits give
\[
\mathrm \R\underline{\operatorname{Hom}}_{\Q_{p,\Box}}(W,V)
\simeq
\mathrm \R\!\lim_n(W_n^*\otimes^\Box V)
\simeq W^*\otimes^\Box V.
\]
 Applying it to the complex of Fr\'echet spaces representing $\mathrm R\Gamma(S_{\mathrm{pro\acute et}}^{\mathrm{top}},\mathcal F)$ termwise and using the third claim we get what we wanted (use  the usual two-term telescope resolution of \(W\) to deal  with  convergence issue).
\end{proof}

  \subsubsection{Condensed versus solid sheaves}
The following result is a sheaf version of Proposition \ref{nyear1}. 
  \begin{proposition}\label{nyear1a}
  \begin{enumerate}
 \item  The forgetful functor
 \begin{equation}\label{form1a}
 {\rm Sh}(S_{\proeet},{\rm Solid})\to  {\rm Sh}(S_{\proeet},{\rm CondAb})
 \end{equation}
 is  (topologically) fully faithful. The essential image is
  stable under all limits, colimits, and extensions.  Moreover, the
inclusion \eqref{form1a} admits a (topological)  left adjoint, the solidification functor, 
\begin{equation}
{\rm Sh}(S_{\proeet},{\rm CondAb})\to  {\rm Sh}(S_{\proeet},{\rm Solid}): \quad \sff \mapsto \sff^{\Box}, 
 \end{equation}
which   preserves all colimits and is  symmetric monoidal. We have analogous claims for the forgetful functor  ${\rm Sh}(S_{\proeet},\Q_{p,\Box}) \to {\rm Sh}(S^{\rm top}_{\proeet},\Q_p)$.
\item 
  The  forgetful functor
  \begin{equation}\label{form2a}
  \sd(S_{\proeet},{\rm Solid})\to \sd(S_{\proeet},{\rm CondAb})
  \end{equation}
   is  (topologically) fully faithful.   It preserves  all limits and colimits. 
\item The functor \eqref{form2a}  admits a  left adjoint
$$\sff  \mapsto  \sff^{\LL_{\Box}}: \sd(S_{\proeet},{\rm CondAb}) \to \sd(S_{\proeet},{\rm Solid}),
$$
which is the (topological)  left derived functor of the solidification functor $(-)^{\Box}$.   It is symmetric monoidal. We have analogous claims for the forgetful functor $\sd(S_{\proeet},\underline{\Q}_{p,\Box})\to \sd(S^{\rm top}_{\proeet},\underline{\Q}_{p})$.
\item For  $\sff_1, \sff_2 \in \sd(S_{\proeet},{\rm Solid})$, the natural map
$$\R{\Hhom}^{\Box}_{S}(\sff_1,\sff_2) \to  \R{\Hhom}_{S^{\rm top}}(\sff_1,\sff_2)$$
is a quasi-isomorphism. We have an analogous claim for the $\underline{\Q}_p$-sheaves.
\end{enumerate}
  \end{proposition}
  \begin{proof} We start with the first claim. The fully-faithfulness was shown above, in Section \ref{solid1} (see Lemma \ref{solidification1}). Limits of solid sheaves in the condensed setting are solid because this is true for presheaves; for colimits -- this holds on the level of presheaves and condensed sheafification is compatible with solid sheafification. For extensions: If $0\to \sff_1\to \sff_2\to \sff_3\to 0 $ is an extension of condensed sheaves such that $\sff_1,\sff_3$ are solid then, on every $Y\in S_{\proeet}$,  we have
  an exact sequence in ${\rm CondAb}$
  $$
  0\to \sff_1(Y)\to \sff_2(Y)\to \sff_3(Y)\stackrel{f}{\to} H^1_{\proeet}(Y,\sff_1). 
  $$
 The cohomology group $H^1_{\proeet}(Y,\sff_1)$ is solid (because so is $\sff_1$), hence so is 
the kernel of the map~$f$.  Thus, by Proposition \ref{nyear1}, $\sff_2(Y)$ is solid, as wanted.  The claim about solidification follows from Lemma \ref{solidification1}, Lemma \ref{tensor}, and Proposition  \ref{nyear1}.

   We turn now to the second claim of the proposition. For fully-faithfulness, 
for $\sff,\sg\in \sd(S_{\proeet},{\rm Solid})$, we need to show that the canonical map 
\begin{equation}\label{lost1}
\R\uHom^{\Box}_{S}(\sff,\sg  ) \to \R\uHom_{S^{\rm top}}(\sff,\sg  )
\end{equation}
is a quasi-isomorphism.  We will need the following fact:
\begin{lemma}\label{paris20}
If $W$ is a prodiscrete condensed set (i.e. a cofiltered limit of discrete condensed
sets) then $\Z[W]^{\LL_{\Box}}\stackrel{\sim}{\to} \Z[W]^{\Box}$. Similarly with $\underline{\Q}_p(S)$-coefficients, for $S\in {\rm Perf}_C$.
\end{lemma}
\begin{proof} A prodiscrete condensed set is associated to a topological inverse limit of discrete sets, which is Hausdorff and has a basis of clopen subsets. Its compact subsets are therefore profinite. They form a filtered system under inclusion, since a finite union of compact subsets is compact. Every continuous map from a profinite test set to $W$ factors through its compact image; consequently, $W=\colim_iT_i$ as a condensed set, where the $T_i$ run through its compact profinite subsets. For each $T_i$, we have $\Z[T_i]^{\LL_{\Box}}\simeq\Z[T_i]^{\Box}$ by \cite[Cor. 6.1.]{Sch19}. We have
$$
\Z[W]^{\LL_{\Box}}=\Z[\colim_iT_i]^{\LL_{\Box}}\simeq(\colim_i\Z[T_i])^{\LL_{\Box}}\simeq \colim_i\Z[T_i]^{\LL_{\Box}}\stackrel{\sim}{\to} \colim_i\Z[T_i]^{\Box}\simeq \Z[W]^{\Box},
$$
which is what we wanted. 

     For the case of $\underline{\Q}_p(S)$-coefficients, put
$A:=\underline{\Q}_p(S)$ and denote its solid analytic ring by
$A_{\Box}$. Since $A_{\Box}$ is an analytic ring over
$\Z_{\Box}$, \cite[Prop.~2.11]{And21} gives, for every
profinite set $T$, a natural quasi-isomorphism
\[
A[T]^{\LL_{\Box}}\stackrel{\sim}{\longrightarrow} A[T]^{\Box}.
\]
Writing $W=\colim_i T_i$ as above and using the compatibility
of derived and underived solidification with filtered
colimits, we obtain
\[
A[W]^{\LL_{\Box}}
\simeq \colim_i A[T_i]^{\LL_{\Box}}
\simeq \colim_i A[T_i]^{\Box}
\simeq A[W]^{\Box},
\]
as wanted.
\end{proof}
We also record the compatibility of solidification with evaluation that will be used below.
\begin{lemma}\label{solid-wcontractible}
Let $\sh\in {\rm Sh}(S_{\proeet},{\rm CondAb})$ and let $X$ be w-contractible. Then the natural map
$$
\sh(X)^{\Box}\longrightarrow \sh^{\Box}(X)
$$
is an isomorphism. Here $\sh^{\Box}$ is obtained by sheafifying the objectwise solidification of $\sh$.
\end{lemma}
\begin{proof}
Since $X$ is   w-contractible, every covering of $X$  is refined by a finite disjoint clopen decomposition: a section of the disjoint union of a finite covering family determines the required clopen pieces. Thus a presheaf on w-contractible basis is a sheaf precisely when it takes finite disjoint unions to products. Since solidification preserves finite products, the objectwise solidification of the restriction of $\sh$ is already a sheaf on this basis. This gives the wanted isomorphism.
\end{proof}

  Coming back to map  \eqref{lost1}, for the moment, let $\sff\in{\rm Sh}(S_{\proeet},{\rm Solid})$
be concentrated in degree zero. Choose a projective resolution
$\spp^\bullet\to\sff$ in
${\rm Sh}(S_{\proeet},{\rm CondAb})$, concentrated in
nonpositive degrees, whose terms are direct sums of the
generators $\Z[h_Y^\delta]\otimes\Z[T]$, with $Y$
w-contractible and $T$ extremally disconnected.

We claim that $(\spp^\bullet)^\Box\to\sff$ is still a
quasi-isomorphism. This can be checked on w-contractible
$X$, where evaluation is exact and
Lemma~\ref{solid-wcontractible} identifies
$(\spp^j)^\Box(X)$ with $\spp^j(X)^\Box$.
For a generator, sheafification on the w-contractible basis
gives
\[
\bigl(\Z[h_Y^\delta]\otimes\Z[T]\bigr)(X)
 \simeq
 \colim_{X=\coprod_a X_a}
 \bigoplus_a \Z[\Hom_S(X_a,Y)^\delta\times T],
\]
where the colimit runs over finite clopen partitions of $X$
and $\Hom_S(X_a,Y)^\delta$ denotes the discrete condensed
set of morphisms. Each
$\Hom_S(X_a,Y)^\delta\times T$ is prodiscrete.
Lemma~\ref{paris20}, together with compatibility of derived
solidification with direct sums and filtered colimits,
therefore gives
\[
\spp^j(X)^{\LL_{\Box}}\simeq\spp^j(X)^\Box.
\]
Thus $\spp^\bullet(X)\to\sff(X)$ is a bounded-above
resolution by objects acyclic for solidification. Applying
solidification termwise computes $\sff(X)^{\LL_{\Box}}$;
see \cite[Tag~05TA]{Stck}. Hence,
\[
(\spp^\bullet)^\Box(X)
 \simeq \sff(X)^{\LL_{\Box}}
 \simeq \sff(X),
\]
the last quasi-isomorphism following from
Proposition~\ref{nyear1}, since $\sff(X)$ is solid.
This proves the claim.

  Using  that the $\spp^j$'s above are projective in   ${\rm Sh}(S_{\proeet}, {\rm CondAb})$ and their solidifications are projective in ${\rm Sh}(S_{\proeet}, {\rm Solid})$, showing that the map \eqref{lost1} is a quasi-isomorphism  is   reduced to showing  that the map
\begin{equation}
\Hom^{\Box}_{S}(\Z[h^{\delta}_Y]\otimes^{\Box}\Z[T],\sg  ) \to \Hom_{S^{\rm top}}(\Z[h^{\delta}_Y]\otimes\Z[T],\sg  )
\end{equation}
is an isomorphism for a w-contractible $Y$ over $S$, an extremally disconnected $T$, and a solid sheaf $\sg$. But this follows from the first claim of the proposition. 

  To pass to arbitrary complexes $\sff$, denote by
\[
j:\quad \sd(S_{\proeet},{\rm Solid})
 \longrightarrow \sd(S_{\proeet},{\rm CondAb})
\]
the functor \eqref{form2a}.
The underived inclusion is exact by claim~(1).
Using $K$-projective resolutions, which exist by the
compact projective generators above and
\cite[Cor.~3.5]{Spal}, we derive the solidification
adjunction to obtain
\[
L:=(-)^{\LL_{\Box}}\dashv j;
\]
see \cite[Tag~0FNC]{Stck}.
The preceding calculation shows that the counit $Lj\sff\longrightarrow\sff$
is a quasi-isomorphism for every solid sheaf $\sff$
concentrated in degree zero.

Both $L$ and $j$ preserve colimits: this holds for $L$
because it is a left adjoint, and for $j$ because it is
exact and preserves coproducts by claim~(1).
The compact projective generators of
${\rm Sh}(S_{\proeet},{\rm Solid})$, placed in degree
zero, generate its unbounded derived category under
colimits and shifts. Consequently, the counit is a
quasi-isomorphism for every
$\sff\in\sd(S_{\proeet},{\rm Solid})$.
Thus $j$ is fully faithful on the unbounded derived
categories.
The topological version can be derived by the same argument applied to the tensor products $\sff\otimes^{\LL_{\Box}}\Z[T]$'s. Moreover, $j$ preserves all limits as a right adjoint,
and its preservation of colimits was established above. This proves claim (2).

   In fact, this also proves the fourth claim of the proposition.

   For claim~(3), the existence of derived solidification
and its adjunction to \eqref{form2a} were established
above. The corresponding internal adjunction yields
compatibility with monoidal structures via the
following computation:
  We want to show that, for $\sff_1, \sff_2\in \sd(S_{\proeet},{\rm CondAb})$, we have 
   \begin{equation}
  ( \sff_1\otimes^{\LL}\sff_2)^{\LL_{\Box}}\simeq \sff_1^{\LL_{\Box}}\otimes^{\LL_{\Box}}\sff_2^{\LL_{\Box}}
   \end{equation}
   in $\sd(S_{\proeet}, {\rm Solid})$. 
   But, for $\sg\in\sd(S_{\proeet}, {\rm Solid})$, by the  already proven adjointness properties and claim (4) of our proposition, we have the following quasi-isomorphisms 
   \begin{align*}
   \R\Hom^{\Box}_{S}((\sff_1\otimes^{\LL}\sff_2)^{\LL_{\Box}},\sg) & \simeq  \R\Hom_{S^{\rm top}}(\sff_1\otimes^{\LL}\sff_2,\sg)\simeq \R\Hom_{S^{\rm top}}(\sff_1,\R\Hhom_{S^{\rm top}}(\sff_2,\sg))\\
    & \stackrel{\sim}{\leftarrow} \R\Hom_{S^{\rm top}}(\sff_1,\R\Hhom^{\Box}_{S}(\sff^{\LL_{\Box}}_2,\sg))
     \simeq  \R\Hom^{\Box}_{S}(\sff^{\LL_{\Box}}_1, \R\Hhom^{\Box}_{S}(\sff^{\LL_{\Box}}_2,\sg))   \\
      & \simeq  \R\Hom^{\Box}_{S}(\sff^{\LL_{\Box}}_1\otimes^{\LL_{\Box}}\sff^{\LL_{\Box}}_2,\sg),
   \end{align*}
   as wanted. 
  
    Claim (3)  for $\underline{\Q}_p$-coefficients follows by the same argument as in the abelian case. 
  \end{proof}
 
\begin{remark}  Recall that sheaves of $\underline{\Q}_p$-modules on ${\rm Spa}(C)_{\proeet}$ are called Vector Spaces (${\rm VS}$ for short). Following this, we will call sheaves from ${\rm Sh}({\rm Spa}(C)_{\proeet},\underline{\Q}_{p,\Box})$ {\em Naive Topological Vector Spaces} (${\rm NTVS}$ for short).
\end{remark}
\subsubsection{\'Etale Naive Topological Vector Spaces}  The constructions from Section \ref{defnaive} and Section \ref{naive52} have \'etale versions though one needs to be  a bit careful so that all of them go through. 

  Let $S$ be  a strictly totally disconnected affinoid perfectoid  over $C$. Denote by ${\rm sPerf}_S$ the category of strictly totally disconnected affinoid perfectoid spaces over $S$.  Proceeding as in Section \ref{naive-521} we can define the topological \'etale site $S^{\rm top}_{\eet}$ based on the category ${\rm sPerf}_S$ equipped with the \'etale topology.  The constructions from Section \ref{defnaive} and Section \ref{naive52} go through with two exceptions: we do not have the map $\eta$ from Section \ref{naive-521} and 
  in Lemma \ref{koczkodan2} we can take all $Y$'s. 
  
  In fact, later on, we will go even further and often drop the topology altogether (but still work  with the category ${\rm sPerf}_S$). 
  
\section{Topological Vector Spaces} 
We pass now to our main object of study: the categories of topologically enriched presheaves.
We start with the definition and properties of the topology on mapping spaces between perfectoid affinoids over $C$. Then we define topologically enriched presheaves with  values in condensed and solid abelian groups, discuss examples and the Enriched Yoneda Lemma. After that we prove that topologically enriched presheaves  form Grothendieck abelian categories, study their $\infty$-derived categories, the solidification functors, and monoidal structures. Finally, we extend all of this to the categories of $\underline{\Q}_p$-modules.
\subsection{Topologically enriched  presheaves}\label{ias1}
Many of the topological presheaves    we will be working with will be canonically
 enriched.  
 \begin{remark} The main technical tool that we need that convinced us to use topologically enriched presheaves is the enriched Yoneda Lemma that is needed in the proof of Theorem \ref{duck1} and Corollary \ref{lebrasII}  (used to  relate topological $\R\Hom$'s to the algebraic ones and the ones for perfect complexes on the Fargues-Fontaine curve). The consideration of topologically enriched presheaves  is not new:
 \begin{enumerate}
 \item It appears implicitly in the original definition of Banach-Colmez spaces in  \cite{CF} (Banach-Colmez spaces are presheaves valued in Banach spaces and are topologically enriched) and also explicitly\footnote{Though at the time of writing \cite{CN5} the authors were not aware that the extra continuity property they have imposed amounts to the notion of enrichment in category theory.}  in the definition of ${\rm qBC}$'s in \cite{CN5}. 
 
 \item Topologically enriched functors in the context of the theory of Banach spaces were studied for a while: see, for example, \cite{HP}, \cite{Pel}. They are called "strong functors" there. More generally, the papers \cite{Cas1}, \cite{Cas2} survey the categorical approach to the theory of Banach spaces. 
 \end{enumerate}
\end{remark}

 Our main references for enriched categories and enriched functors  are   \cite{Kelly}, \cite{GM}, \cite[App. C]{HM}.
 
  \subsubsection{Topologized mapping space} Let $S\in {\rm Perf}_C$. We will enrich the category ${\rm Perf}_S$ in ${\rm Cond}$.  For $Y_1, Y_2\in{\rm Perf}_S$, we set: 
$$
\uHom_S(Y_1,Y_2): \{T\mapsto \Hom_S(Y_1\times T,Y_2)\}.
$$ For every $Y\in {\rm Perf}_S$, the {\em identity map}  $i_Y:\{*\}\to \uHom_S(Y,Y)$ sends $\{*\}$ to ${\rm Id}$. For every triple $Y_1,Y_2,Y_3\in {\rm Perf}_S$, the {\em composition} map
$$
\uHom_S(Y_2,Y_3)\otimes \uHom_S(Y_1,Y_2)\to \uHom_S(Y_1,Y_3)
$$
is defined by compatible compositions, for  profinite sets $T$, 
$$
\Hom_S(Y_2\times T,Y_3)\otimes \Hom_S(Y_1\times T,Y_2)\to \Hom_S(Y_1\times T,Y_3),
$$
which send the pair $(f_2,f_1)$ to the composition 
$$
f_2f_1:\quad Y_1\times T\lomapr{{\rm Id}\times \Delta} Y_1\times T\times T\lomapr{f_1\times{\rm Id} }Y_2\times T\lomapr{f_2}Y_3.
$$
The associativity and unit axioms (see \cite[(1.3), (1.4)]{Kelly}) are easily checked to hold.

\begin{lemma}\label{prodiscrete} 
The condensed set $ \uHom_S(Y_1,Y_2)$  is prodiscrete.
\end{lemma}
 \begin{proof}
 Set $R^+_1:= R^+_{Y_1}, R^+_2:= R^+_{Y_2}, R^+_S:= R^+_{S}$. These are prodiscrete sets. 
We claim that there is an isomorphism of condensed sets
 $$
  \uHom_S(Y_1,Y_2)\simeq   \underline{\Hom}_{\underline{R}^+_S}^{\rm Alg}(\underline{R}^+_2, \underline{R}^+_1),
 $$
Here $\underline{\Hom}_{\underline{R}^+_S}^{\rm Alg}$ denotes the condensed set of morphisms of condensed $\underline{R}^+_S$-algebras: its $T$-points are the morphisms
$\underline{R}^+_2\to\uHom(T,\underline{R}^+_1)$ of such algebras, where the target has pointwise ring operations and its constant $\underline{R}^+_S$-algebra structure. The notation $\Hom_{R^+_S}^{\rm Alg}$ below denotes continuous homomorphisms of topological $R^+_S$-algebras.

 To see this we start with showing that  there is an isomorphism of condensed sets
 $$
   \uHom_S(Y_1,Y_2)\simeq  \underline{\Hom_{R^+_S}^{\rm Alg}(R^+_2, R^+_1)},
 $$
  where $ \Hom_{R^+_S}^{\rm Alg}(R^+_2, R^+_1)$ is equipped with the compact-open topology. For that, we compute:
 \begin{align*}
 \Hom_S(Y_1\times T,Y_2) & =\Hom_{R^+_S}^{\rm Alg}(R^+_2, R^+_1\wh{\otimes}\scc(T,\so_C))\simeq \Hom_{R^+_S}^{\rm Alg}(R^+_2, \scc(T,R^+_1))\\
  & \simeq \scc(T, \Hom_{R^+_S}^{\rm Alg}(R^+_2, R^+_1)),
 \end{align*}
  where  the last  isomorphism follows from the exponential law since $T, R^+_1, R^+_2$ are compactly generated  (since they are countable colimits of discrete sets, hence metrizable)  and Hausdorff. 
By \cite[Prop. 4.2]{Sch19},  we have  the natural isomorphism of condensed sets $$
 \underline{\Hom}_{\underline{R}^+_S}^{\rm Alg}(\underline{R}^+_2, \underline{R}^+_1)\stackrel{\sim}{\to} 
 \underline{\Hom_{R^+_S}^{\rm Alg}(R^+_2, R^+_1)}.
 $$
 
  It suffices now to show that $\underline{\Hom}_{\underline{R}^+_S}^{\rm Alg}(\underline{R}^+_2, \underline{R}^+_1)$ is prodiscrete.   But this  is a closed condensed subset of 
  $ \uHom_{\underline{\so}_C}(\underline{R}^+_2,\underline{R}^+_1)$  hence, by \cite{MO},  it suffices to show that  $ \uHom_{\underline{\so}_C}(\underline{R}^+_2,\underline{R}^+_1)$ is prodiscrete. For that, we compute
 \begin{align*}
 \uHom_{\underline{\so}_C}(\underline{R}^+_2,\underline{R}^+_1) & \simeq \lim_n \uHom_{\underline{\so}_{C,n}}(\underline{R}^+_{2,n},\underline{R}^+_{1,n}) \simeq \lim_n \uHom_{\underline{\so}_{C,n}}(\colim_i\underline{R}^+_{2,n,i},\underline{R}^+_{1,n})\\
  & \simeq \lim_n \lim_i\uHom_{\underline{\so}_{C,n}}(\underline{R}^+_{2,n,i},\underline{R}^+_{1,n}),
 \end{align*}
Here $\so_{C,n}=\so_C/p^n$, $R^+_{j,n}=R^+_j/p^n$, and the $\underline{R}^+_{2,n,i}\subset \underline{R}^+_{2,n}$ are finitely generated $\underline{\so}_{C,n}$-submodules. Each of the final internal $\Hom$'s is discrete, since its source is finitely generated and its target is discrete. The displayed limit is therefore prodiscrete, which proves the claim.

 \end{proof}
 \subsubsection{Topologically enriched  presheaves} \label{warsz1} 
Let $\scc$ be a bicomplete, locally finitely presentable, closed symmetric monoidal  category (for example, 
${\rm Cond}$, ${\rm CondAb}$, ${\rm Solid}$). 
Such a category $\scc$ is canonically $\scc$-enriched ($\scc$-category, for short). We assume that a faithful right adjoint $U:\scc\to{\rm Cond}$ is fixed and that its left adjoint $L$ is symmetric monoidal. For the three examples above, $U$ is the forgetful functor and $L$ is, respectively, the identity, $T\mapsto\Z[T]$, and $T\mapsto\Z[T]^{\Box}$. The underlying condensed sets of the internal $\Hom$'s give the ${\rm Cond}$-enrichment of $\scc$. Whenever a condensed set $T$ occurs in a tensor or cotensor in $\scc$, we mean $L(T)$; in particular, $c\otimes T$ means $c\otimes_{\scc}L(T)$.

  We will denote by
  $\Hom_{\scc}(-,-)$ and $\uHom_{\scc}(-,-)$ its $\Hom$-set and $\scc$-$\Hom$, respectively. 
  We have $\uHom_{\scc}(-,-)(*)\simeq \Hom_{\scc}(-,-)$.   If $\scc={\rm Cond}$ we will skip the subscripts. 
 
   Let $S\in {\rm Perf}_C$.  A $\scc$-{\em enriched presheaf} (or $\scc$-presheaf for short)
 $\sff: {\rm Perf}^{\rm op}_S \to \scc$ consists of a function
 $$\sff: {\rm Perf}_S \mapsto {\rm Ob} \scc$$
 together with,  for $Y_1, Y_2\in{\rm Perf}_S$,
 a map of condensed sets 
\begin{equation}\label{str2}
\underline{\sff}_{Y_1, Y_2}:\quad\uHom_S(Y_1,Y_2)\to \uHom_{\scc}(\sff(Y_2),\sff(Y_1))
\end{equation}
in a   manner  compatible with  compositions and  identities (see \cite[(1.5), (1.6)]{Kelly}).  For $\scc$-presheaves $\sff, \sg$, a $\scc$-natural transformation $f: \sff\to \sg$
is a ${\rm Perf}_S$-indexed family of maps in $\scc$:
$$
f_Y: \sff(Y)\to \sg(Y)
$$
satisfying a $\scc$-naturality condition (see \cite[(1.7)]{Kelly}, \cite[(4.30)]{GM}): the following diagram commutes in ${\rm Cond}$
$$
\xymatrix{
\uHom_S(Y,Y_1)\ar[r]^-{\underline{\sff}_{Y,Y_1}} \ar[d]^-{\underline{\sg}_{Y,Y_1}} & \uHom_{\scc}(\sff(Y_1),\sff(Y))\ar[d]^{f_{*,Y}} \\
 \uHom_{\scc}(\sg(Y_1),\sg(Y)) \ar[r]^{f^*_{Y_1}} &  \uHom_{\scc}(\sff(Y_1),\sg(Y)).
}
$$
 We will write $\Hom_{\scc}(\sff,\sg)$ for the set of $\scc$-natural transformations. 
 
 We will denote the category of such presheaves by $\underline{\rm PSh}(S,\scc)_0$ and, in the case $\scc={\rm Cond}$, we will simply write $\underline{\rm PSh}(S)_0$. 
The topological sheaves that come from $S_{\proeet}$ are canonically enriched:
\begin{lemma}\label{enriched1}Let $\sff\in {\rm Sh}(S_{\proeet})$. The sheaf $\pi_*\sff\in {\rm Sh}(S_{\proeet},{\rm Cond})$ is canonically enriched. 
\end{lemma}
\begin{proof}
We need to define,  compatible in $T$, maps
\begin{equation}
\Hom_S(Y_1\times T,Y_2)\longrightarrow
\Hom(\pi_*\sff(Y_2)\times T,\pi_*\sff(Y_1)).
\end{equation}
The $\Hom$ on the right is in ${\rm Cond}$: its elements are families of maps
$$
\sff(Y_2\times T_1)\times\underline{T}(T_1)\longrightarrow\sff(Y_1\times T_1)
$$
natural in the profinite set $T_1$. For $f:Y_1\times T\to Y_2$ and $g:T_1\to T$, the required map on sections is pullback along
$$
Y_1\times T_1\longrightarrow Y_2\times T_1,\qquad
(f\circ({\rm Id}_{Y_1}\times g),{\rm pr}_{T_1}).
$$
Evaluating at $T=*$ recovers the ordinary pullback maps of the presheaf. Commutativity of the composition and the identities diagrams \cite[(1.5), (1.6)]{Kelly} is easy (if tedious) to check. 
\end{proof}

\begin{example}
Lemma \ref{enriched1} combined with Example \ref{ex1} shows that the sheaves $\Q_p, {\mathbb G}_a\in {\rm Sh}(S_{\proeet},{\rm Cond})$ are canonically enriched and so are the period sheaves ${\mathbb B}^+_{\dr}/t^i$ and ${\mathbb B}_{\crr}^{+,\phi=p^i}$, $i\geq 0$. 
\end{example}
\subsubsection{The ${\rm Cond}$-category of  topologically enriched  presheaves} Assume moreover that  $\scc$ is a bicomplete  $\scc$-category:  $\scc$ is bicomplete in the usual set-based sense and the tensors and cotensors:
$$
c\otimes W,\quad [W,c], \quad c, W\in\scc,
$$
 satisfy the $\scc$-adjunctions
$$
\uHom_{\scc}(c_1\otimes W,c_2)\simeq [W,\uHom_{\scc}(c_1,c_2)]\simeq \uHom_{\scc}(c_1,[W,c_2]).
$$
(Here $[-,-]=\uHom_{\scc}(-,-)$ but below we will need to separate these two operations.)  
{\it This will be our standard assumption on $\scc$ from now on}.

   The category  $\underline{{\rm PSh}}(S,\scc)_0$ is the underlying category of the $\scc$-category $\underline{{\rm PSh}}(S,\scc)$, where  the $\Hom$-object in $\underline{{\rm PSh}}(S,\scc)$  is defined as the enriched end:
\begin{equation}\label{kol12}
\uHom_{\scc}(\sff,\sg)=\int_{Y\in {\rm Perf}_S}\uHom_{\scc}( \sff(Y),\sg(Y))
\end{equation}
computed  as an ordinary limit in $\scc$. More precisely, it is the equalizer in $\scc$ given by the diagram
$$
\xymatrix{
\prod_{Y\in {\rm Perf}_S}\uHom_{\scc}(\sff(Y),\sg(Y))\ar@<-1mm>[r] \ar@<1mm>[r]  & \prod_{Y,Y_1\in {\rm Perf}_S}\uHom_{\scc}(\uHom_S(Y_1,Y)\otimes\sff(Y),\sg(Y_1)),
}
$$
where the parallel arrows are defined via the evaluation maps
$$
\uHom_S(Y_1,Y)\otimes\sff(Y)  \to \sff(Y_1),\quad 
 \uHom_S(Y_1,Y)\otimes\sg(Y)\to \sg(Y_1)
 $$
of the $\scc$-presheaves $\sff,\sg$. These are the adjoints, under the tensor-hom adjunction, of the enrichment maps $\underline{\sff}_{Y_1,Y}$ and $\underline{\sg}_{Y_1,Y}$ from \eqref{str2}. We have $\uHom_{\scc}(\sff,\sg)(*)=\Hom_{\scc}(\sff,\sg)$. 

   The  category $\underline{{\rm PSh}}(S,\scc)$ is a  bicomplete  $\scc$-category \cite[Sec. 3.3]{Kelly} with limits and colimits computed objectwise. Hence it is  tensored and cotensored over $\scc$:   we have the tensor functor
 \begin{align*}
  (-)\otimes (-):\quad  & \scc\times \underline{\rm PSh}(S,\scc)\to \underline{\rm PSh}(S,\scc),\\
   & W\otimes\sff=\{Y\mapsto W\otimes\sff(Y)\},
 \end{align*}
and  the cotensor functor
 \begin{align*}
 [-,-]:\quad  & \scc^{\rm op}\times \underline{\rm PSh}(S,\scc)\to \underline{\rm PSh}(S,\scc),\\
  & [W,\sff]=\{Y\mapsto \uHom_{\scc}(W,\sff(Y))\},
 \end{align*}
such that, for each $W\in \scc$, there is a natural isomorphism
 \begin{equation}
 \uHom_{\scc}(W,\uHom_{\scc}(\sff,\sg))\simeq \uHom_{\scc}(\sff,[W,\sg]).
 \end{equation}
 Moreover,  for each $W\in \scc$, there is a natural isomorphism
 $$
 \uHom_{\scc}(W\otimes\sff,\sg)\simeq\uHom_{\scc}(W,\uHom_{\scc}(\sff,\sg)). 
 $$
  \begin{remark}{\rm (Topological presheaves)} \label{koczkodan1} The presheaves studied in Section \ref{naive1},  which we called "{topological presheaves}", also form a $\scc$-category. We will call  such $\scc$-categories $ {\rm PSh}(S,\scc)$. Their underlying categories\footnote{For an enriched category $\sa$, we denote by $\sa_0$ the underlying category.} ${\rm PSh}(S,\scc)_0$ are   defined as in Section \ref{warsz1}   but with the trivial enrichment on the mapping spaces, i.e., in the structure maps \eqref{str2}, instead of
$\uHom_S(Y_1,Y_2)$ we take $\underline{\Hom_S(Y_1,Y_2)}$, where $\Hom_S(Y_1,Y_2)$ is given the discrete topology. Hence they are canonically  induced by structure maps of the form
\begin{equation}\label{str1}
\sff_{Y_1, Y_2}:\quad \Hom_S(Y_1,Y_2)\to \Hom_{\scc}(\sff(Y_2),\sff(Y_1)).
\end{equation}
The $\Hom$-object is defined as in \eqref{kol12}. 
 
 We have the forgetful functor
$$
(-)^{\rm cl}: \quad  \underline{\rm PSh}(S,\scc)\to {\rm PSh}(S,\scc)
 $$
 given by evaluating the structure map \eqref{str2} on $*$ to yield the structure map \eqref{str1}.
\end{remark}
  \subsubsection{Alternative description of topologically enriched presheaves} 
   We keep assuming that the ${\rm Cond}$-category $\scc$ is  bicomplete. In that case the category $\underline{{\rm PSh}}(S,\scc)_0$ has an alternative description, which will allow us to understand it better.  We want to  think of a ${\rm Cond}$-presheaf $\sff: {\rm Perf}^{\rm op}_S\to \scc$ as a presheaf
  $$
  \sff:{\rm Perf}_S^{\rm op}\mapsto {\rm Ob}\scc
  $$
  plus a compatible collection of (structure) maps in $\scc$
  \begin{equation}\label{alt1}
 \underline{\sff}_{Y,T}:\quad  \sff(Y\times T)\otimes {T}\to \sff(Y),\quad Y\in{\rm Perf}_S,\ T\in {\rm ProFin}. 
  \end{equation}
Equivalently, we    can use  the adjoints of the structure maps \eqref{alt1} 
$$
\underline{\sff}^{\prime}_{Y,T}:\quad  \sff(Y\times T)\to  [T, \sff(Y)]
$$
and the compatibility condition corresponds to  the maps
$$
\underline{\sff}^{\prime}_{Y,*}:\quad  \sff(Y)\to  [*, \sff(Y)]\simeq \sff(Y)
$$
being the identities, the maps $\underline{\sff}^{\prime}_{Y,T}$ being functorial in $Y$ and $T$, i.e., the following
 diagrams
\begin{equation}\label{alt1b}
\xymatrix{
  \sff(Y\times T)\ar[r]^{\underline{\sff}^{\prime}_{Y,T}} \ar[d]^{(f\times g)^*}& [T, \sff(Y)]\ar[d]^{(g^*,f^*)} \\
  \sff(Y_1\times T_1)\ar[r]^{\underline{\sff}^{\prime}_{Y_1,T_1} }&   [T_1, \sff(Y_1)],
}
\end{equation}
for all maps $f:Y_1\to Y$, $g:T_1\to T$, being commutative, together with the composition rule
$$
\underline{\sff}^{\prime}_{Y,T\times T_1}
=[T_1,\underline{\sff}^{\prime}_{Y,T}]
\circ\underline{\sff}^{\prime}_{Y\times T,T_1},
$$
where $[T_1,[T,\sff(Y)]]$ is identified with $[T\times T_1,\sff(Y)]$.

  A ${\rm Cond}$-natural transformation $f:\sff\to\sg$ between two such presheaves    $\sff$, $\sg$ is given by  functorial maps in $\scc$
  $$
  f_Y: \sff(Y)\to \sg(Y),\quad Y\in {\rm Perf}_S,
  $$
  which are  compatible with the structure maps \eqref{alt1}. We will write $\Hom_{\scc}(\sff,\sg)_0$ for the set of ${\rm Cond}$-natural transformations. 
  
  We will call the category of such data $\underline{{\rm PSh}}^{\prime}(S,\scc)_0$.
  We obtain the ${\rm Cond}$-category $\underline{{\rm PSh}}^{\prime}(S,\scc)$ (whose underlying category is $\underline{{\rm PSh}}^{\prime}(S,\scc)_0$) by enriching in ${\rm Cond}$ the $\Hom$-sets $\Hom_{\scc}(\sff,\sg)_0$:
  $$\uHom_{\scc}(\sff, \sg):=\{T\mapsto \Hom_{\scc}(\sff\otimes {T},\sg)_0\},\quad \sff, \sg\in \underline{{\rm PSh}}^{\prime}(S,\scc)_0.
 $$
The presheaf $\sff\otimes T$ belongs to $\underline{{\rm PSh}}^{\prime}(S,\scc)_0$: for a profinite set $T_1$, its structure map is the composite
$$
\sff(Y\times T_1)\otimes T\otimes T_1
\stackrel{\sim}{\longrightarrow}
\sff(Y\times T_1)\otimes T_1\otimes T
\lomapr{\underline{\sff}_{Y,T_1}\otimes{\rm Id}_T}
\sff(Y)\otimes T.
$$
Its unit, functoriality, and composition compatibilities follow from those for $\sff$ and the symmetry of the tensor product. We clearly have $\uHom_{\scc}(\sff,\sg)(*)=\Hom_{\scc}(\sff,\sg)_0$.

  \begin{lemma}\label{goraco1}
The ${\rm Cond}$-categories $\underline{{\rm PSh}}(S,\scc)$ and $\underline{{\rm PSh}}^{\prime}(S,\scc)$  are equivalent. 
\end{lemma}
\begin{proof}  We start with defining  a ${\rm Cond}$-functor $$F: \underline{{\rm PSh}}(S,\scc)\to \underline{{\rm PSh}}^{\prime}(S,\scc),$$ which we will show later to be an equivalence.  Let $\sff\in  \underline{{\rm PSh}}(S,\scc)$. We define $F(\sff)$ as being given by the same presheaf on ${\rm Perf}^{\rm op}_S$ as $\sff$.  To define  the structure maps $\underline{\sff}_{Y, T}$ for $F(\sff)$, we start with the structure maps for $\sff$
$$
\underline{\sff}_{Y, Y_1}: \uHom_S(Y,Y_1)\to \uHom_{\scc}(\sff(Y_1),\sff(Y)).
$$
Evaluated at $T\in {\rm ProFin}$, they yield a compatible collection (indexed by $T$) of maps of sets
$$
\underline{\sff}_{Y,Y_1,T}: \Hom_S(Y\times T, Y_1)\to \Hom_{\scc}(\sff(Y_1)\otimes {T},\sff(Y)).
$$
Setting $Y_1=Y\times T$ in the above and taking the image of the identity by the map $\underline{\sff}_{Y,Y_1,T}$,  we get a compatible, in $Y$ and $T$, family of maps in $\scc$
$$
\sff(Y\times T)\otimes {T}\to \sff(Y),
$$
as wanted. This defines a function
$$
F: {\rm Ob}\, \underline{{\rm PSh}}(S,\scc)\to {\rm Ob}\, \underline{{\rm PSh}}^{\prime}(S,\scc).  
$$

  It remains to define, for each pair $\sff,\sg\in \underline{{\rm PSh}}(S,\scc)$, a map 
  $$
  F_{\sff,\sg}: \uHom_{\scc}(\sff,\sg)\to \uHom_{\scc}(F(\sff),F(\sg))
  $$
compatible with  the composition and the identity (see \cite[(1.5), (1.6)]{Kelly}). This amounts to defining, indexed by $T\in {\rm ProFin}$, a compatible family of maps
$$
  F_{\sff,\sg}(T): \Hom_{\scc}(\sff\otimes{T},\sg)_0\to \Hom_{\scc}(F(\sff)\otimes{T},F(\sg))_0.
$$
But, by the definition of $F(\sff), F(\sg)$, these two sets can be canonically identified. 

 The above proves that the ${\rm Cond}$-functor $F$  is in fact fully faithful. It remains thus to show that it is essentially surjective on objects. 
 For that we will  define the quasi-inverse $G$ to $F$ on objects:    $G$ is  given by the same function on ${\rm Perf}^{\rm op}_S$ as $F$ and  to define  the structure maps $\underline{\sff}_{Y_1, Y_2}$ for $G(\sff)$:
$$
\underline{\sff}_{Y_1, Y_2}: \uHom_S(Y_1,Y_2)\to \uHom_{\scc}(G(\sff)(Y_2),G(\sff)(Y_1))
$$
we need to define, a compatible  in $T$, family of maps of sets
$$
\underline{\sff}_{Y_1, Y_2,T}: \Hom_S(Y_1\times T,Y_2)\to \Hom_{\scc}(\sff(Y_2)\otimes {T},\sff(Y_1)).
$$
To do that, we map a map  $f:Y_1\times T\to Y_2$ to the composition 
$$
\sff(Y_2)\otimes{T}\lomapr{f^*\otimes{\rm Id}}\sff(Y_1\times T)\otimes{T}\lomapr{\underline{\sff}_{Y_1,T}}\sff(Y_1).
$$
The unit condition for $\underline{\sff}^{\prime}_{Y,*}$ gives the identity axiom. Functoriality in $Y$ and $T$ gives naturality in the profinite parameter, and the displayed composition rule identifies successive pullbacks with pullback along the composite. Thus these maps satisfy the enrichment axioms. The same descriptions show that transformations compatible with \eqref{alt1} are exactly the enriched natural transformations. Finally, $FG={\rm Id}$ and $GF={\rm Id}$, so $F$ is essentially surjective, as wanted.
\end{proof}

\begin{example}{\rm (Rigid topologically enriched presheaves)} We will distinguish a subcategory of topologically enriched presheaves. 
\begin{definition}
A topologically enriched presheaf $\sff\in  \underline{{\rm PSh}}(S,\scc)$ is called {\em rigid} if  the structure maps from \eqref{alt1b} 
$$
\underline{\sff}^{\prime}_{Y,T}:\quad  \sff(Y\times T)\to  [T, \sff(Y)]
$$
are isomorphisms in $\scc$, for all $Y\in {\rm Perf}_S$, $ T\in {\rm ProFin}$.  Equivalently, if the maps
$$
\underline{\sff}^{\prime}_{Y,T}(*):\quad \sff(Y\times T, *)\to  \sff(Y,T)
$$
are isomorphisms of sets, for all $Y\in  {\rm Perf}_S$ and $T\in {\rm ProFin}$. The full subcategory of rigid presheaves in $ \underline{{\rm PSh}}(S,\scc)$ 
will be denoted by $ \underline{{\rm PSh}}^{\rm rig}(S,\scc)$.  We will denote by $\underline{{\rm Sh}}^{\rm rig}(S,\scc)$ its full subcategory of {\em rigid sheaves}: a rigid presheaf $\sff$ is a sheaf if the pro-\'etale presheaf  $\eta_*\sff$ is a sheaf. 
\end{definition}
We note that the equivalence in the above definition follows from the commutative diagram
$$
\xymatrix{
\sff(Y\times T\times T_1,*)\ar[rrd]^{\underline{\sff}^{\prime}_{Y ,T\times T_1}(*)} \ar[d]_{\underline{\sff}^{\prime}_{Y\times T,T_1}(*)}  \\
\sff(Y\times T,T_1)\ar@{-->}[rr]^-{\underline{\sff}^{\prime}_{Y,T}(T_1)}   & & \sff(Y,T\times T_1),
}
$$
which is just a rewriting of the composition commutative diagram
$$
\xymatrix{
\sff(Y\times T\times T_1)(*)\ar[rrd]^{\underline{\sff}^{\prime}_{Y ,T\times T_1}(*)} \ar[d]_{\underline{\sff}^{\prime}_{Y\times T,T_1}(*)}  \\
[T_1,\sff(Y\times T)](*)\ar@{-->}[rr]^-{[T_1,\underline{\sff}^{\prime}_{Y,T}]}   & & [T_1,[T,\sff(Y)]](*)=[T_1\times T,\sff(Y)](*).
}
$$

\begin{lemma}\label{enriched1a}
Let $\sff\in {\rm Sh}(S_{\proeet})$. The topologically enriched presheaf $\pi_*\sff\in \underline{{\rm PSh}}(S)$ is a rigid sheaf. Moreover, the functor
$$
\pi^{\rm rig}_*:\quad {\rm Sh}(S_{\proeet})\to \underline{{\rm Sh}}^{\rm rig}(S)_0
$$
is an equivalence of categories. 
\end{lemma}
\begin{remark}
The above lemma is also valid for the pair ${\rm Ab}, {\rm CondAb}$  (in place of the pair ${\rm Set}, {\rm Cond}$). With the same proof.
\end{remark}
\begin{proof}
By Lemma \ref{enriched1}, the presheaf $\pi_*\sff$ is topologically enriched. To see that it is rigid we need to show that 
the structure maps
\begin{equation}
(\underline{\pi_*\sff})^{\prime}_{Y,T}(*):\quad  \pi_*\sff(Y\times T,*)\to    \pi_*\sff(Y, T)
\end{equation}
are isomorphisms of sets. 
But this is clear since they are  isomorphic to  the identity maps
$$
{\rm Id}: \sff(Y\times T)\to   \sff(Y\times T).
$$
 Moreover, we have $\eta_*\pi_*\sff\simeq \sff$ hence $\pi_*\sff$ is a sheaf. 

   Concerning the second claim of the lemma, let us start with essential surjectivity of the functor $\pi_*$. Take a rigid sheaf $\sff$. We have the structure maps
\begin{equation}\label{down2}
\underline{\sff}^{\prime}_{Y,T}(*):\quad  \sff(Y\times T,*)\to   \sff(Y,T),
\end{equation}
of sets
satisfying certain compatibilities. In the language of Section \ref{ias2}, they give a natural transformation of topologically enriched presheaves\footnote{The fact that this is  a natural transformation follows from the coherence of the structure maps \eqref{down2}.}
$$
\pi_*\eta_*\sff\to \sff. 
$$
Since these maps are isomorphisms by assumption, we get the essential surjectivity.

 It remains to show that the functor $\pi_*$ is fully faithful. Faithfulness is clear.  Let $\sff, \sg\in {\rm Sh}(S_{\proeet})$ and take a map of topologically enriched presheaves
 $f: \pi_*\sff\to \pi_*\sg$. Define a map $\tilde{f}: \sff\to \sg$ by setting $\tilde{f}_Y=f_Y(*)$. We need to show that $f=\pi_*\tilde{f}$. This follows, after unwinding the definitions, from the commutative diagram, for $Y\in {\rm Perf}_S$ and $T\in {\rm ProFin}$,
 $$
 \xymatrix{
 \pi_*\sff(Y,T)\ar[r]^{f_Y(T)} & \pi_*\sg(Y,T)\\
 \sff(Y\times T)\ar[r]^{f_{Y\times T}(*)} \ar[u]^{\sim}& \sg(Y\times T)\ar[u]^{\sim}
 }
 $$
 Here, the vertical maps, after identification of the source  and the target, are just identities.
 \end{proof}
 \begin{remark} 
(1) It follows from Lemma \ref{enriched1a} that a rigid presheaf $\sff$ is a sheaf if and only if the topological presheaf $\sff^{\rm cl}$ is a sheaf.

(2) More generally, the same proof identifies $\underline{{\rm PSh}}^{\rm rig}(S)_0$ with the category of ordinary presheaves $\sff$ for which $T\mapsto\sff(Y\times T)$ is a condensed set for every $Y$. This condition ensures that the formula defining $\pi_*\sff$ has condensed values.

(3)  By contrast, the free abelian enriched Yoneda presheaves
$\Z[h_Y^{\rm top}]$ need not be rigid.  For example, take
$\sff:=\Z[h_S^{\rm top}]$.
Since $S$ is terminal in ${\rm Perf}_S$, we have
$\sff(X)=\Z$ for every $X\in{\rm sPerf}_S$.
For $T=\{0,1\}$, the rigidity map is the diagonal
\[
\sff(S\times T)=\Z
 \longrightarrow
\uHom_{{\rm CondAb}}(\Z[T],\sff(S))
 \simeq \Z\oplus\Z,
\]
which is not an isomorphism. Thus $\sff$ is not rigid.  
 \end{remark}
\end{example}
\subsubsection{Enriched Yoneda Lemma}\label{enrichedyoneda}
Recall the following classical construction from \cite[Sec. 2.4]{Kelly}.  
 Let $S\in {\rm Perf}_C$, $Y\in {\rm Perf}_S$. We define the functor
\begin{align*}
{h}^{\rm top}_Y: {\rm Perf}_S^{\rm op}\to {\rm Cond}, \quad X\mapsto \uHom_S(X,Y).
\end{align*}
It is easy to check that 
\begin{equation}\label{Chicago10}
h^{\rm top}_Y\simeq \pi_*h_Y,
\end{equation} where $  \pi: S_{\proeet}\to S_{\proeet}^{\rm top}$ 
 is the canonical projection. In particular, since the pro-\'etale site is subcanonical  and thus the presheaf $h_Y$ is a sheaf, same is true of $h^{\rm top}_Y$. 
Moreover, by Lemma \ref{enriched1}, this sheaf is canonically ${\rm Cond}$-enriched. 

    Let $\sff\in \underline{\rm PSh}(S)$. The ${\rm Cond}$-{\em enriched Yoneda Lemma} yields  an existence of a $\rm Cond$-natural isomorphism
\begin{align}\label{lato24}
\uHom(h^{\rm top}_Y,\sff)=\int_{Y_1\in{\rm Perf}_S}\uHom(\uHom_S(Y_1,Y),\sff(Y_1))\stackrel{\sim}{\leftarrow}\sff(Y)
\end{align}
induced by the maps
$$
\phi_Y: \sff(Y)\to \uHom(\uHom_S(Y_1,Y),\sff(Y_1))
$$
obtained, via the adjunction  $\Hom(x,\uHom(y,z))\simeq \Hom(y,\uHom(x,z))$, $x,y,z\in {\rm Cond}$, from the structure maps
$$
\underline{\sff}_{Y_1,Y}: \uHom_S(Y_1,Y)\to \uHom(\sff(Y),\sff(Y_1)).
$$

 The special case of \eqref{lato24} with $\sff=h^{\rm top}_{Y_1}, Y_1\in{\rm Perf}_S,$ proves that the functor $Y\mapsto h^{\rm top}_Y$ is fully faithful. 
 Moreover, we have $h^{\rm top}_{Y_1\times_S Y_2}\simeq h^{\rm top}_{Y_1}\times h^{\rm top}_{Y_2}$.

   Equivalently, we can formulate an enriched  {\em co-Yoneda Lemma} as a ${\rm Cond}$-natural isomorphism 
$$
\int^{Y_1\in{\rm Perf}_S}\uHom_S(Y,Y_1)\otimes\sff(Y_1)\stackrel{\sim}{\to} \sff(Y)
$$
induced by the morphisms 
$$
\uHom_S(Y,Y_1)\otimes\sff(Y_1)\to \sff(Y),
$$
which are adjoints to the structure maps $\sff_{Y,Y_1}$. Here $\int^{Y_1\in{\rm Perf}_S}\uHom_S(Y,Y_1)\otimes\sff(Y_1)$ is the enriched coend: the  coequalizer given by the diagram 
$$
\xymatrix{
\coprod_{Y_1\in{\rm Perf}_S}\uHom_S(Y,Y_1)\otimes\sff(Y_1)& \coprod_{Y_1,Y_2\in{\rm Perf}_S} \uHom_S(Y,Y_1)\otimes\uHom_S(Y_1,Y_2)\otimes\sff(Y_2), \ar@<-1mm>[l]\ar@<1mm>[l]
}
$$
where the parallel arrows are induced  by the composition and the evaluation maps
$$
\uHom_S(Y,Y_1)\otimes \uHom_S(Y_1,Y_2)\to  \uHom_S(Y,Y_2),\quad \uHom_S(Y_1,Y_2)\otimes \sff(Y_2)\to \sff(Y_1).
$$
\subsubsection{Monoidal structures}  Let $S\in {\rm Perf}_C$ and let $\scc={\rm Cond}$. 
We will use the same notation for  tensor products and internal $\Hom$'s of ${\rm Cond}$-presheaves on the category ${\rm Perf}_S$ as in  Section \ref{cond-kol1} and Section \ref{solid1}. 

  \vskip2mm
 ($\bullet$) {\em Tensor product.}  The tensor products of presheaves  are defined objectwise and the structure maps are modified in a canonical way.  

\vskip2mm
 ($\bullet$) {\em Internal $\Hom$.} For two ${\rm Cond}$-presheaves $\sff, \sg$, we set 
\begin{align*}
\Hhom_{S^{\rm top}}(\sff,\sg) & :=\{Y\mapsto \uHom_{S^{\rm top}}(h^{\rm top}_Y\otimes\sff,\sg)\}\\
 &=\{Y\mapsto \int_{Y_1\in {\rm Perf}_S}\uHom(\uHom_S(Y_1,Y)\otimes\sff(Y_1),\sg(Y_1))\}.
\end{align*}
This is a ${\rm Cond}$-presheaf (with values in ${\rm Cond}$).

\vskip2mm
 ($\bullet$) {\em Adjunction.}   We have the usual tensor-hom adjunction: 
\begin{lemma}  Let $\sff_1,\sff_2,\sff_3\in \underline{\rm PSh}(S)$. We have  functorial isomorphisms in ${\rm Cond}$ and $\underline{\rm PSh}(S)$, respectively:
\begin{align*}
\uHom_{S^{\rm top}} (\sff_1\otimes\sff_2,\sff_3)\simeq \uHom_{S^{\rm top}}(\sff_1,\Hhom_{S^{\rm top}}(\sff_2,\sff_3)),\\
\Hhom_{S^{\rm top}} (\sff_1\otimes\sff_2,\sff_3)\simeq \Hhom_{S^{\rm top}}(\sff_1,\Hhom_{S^{\rm top}}(\sff_2,\sff_3)).
\end{align*}
\end{lemma}
\begin{proof} For the first isomorphism, the computation is standard: We have the following sequence of functorial isomorphisms in ${\rm Cond}$
\begin{align*}
\uHom_{S^{\rm top}}(\sff_1, & \Hhom_{S^{\rm top}}(\sff_2,\sff_3))\\
 & \simeq \int_{Y\in {\rm Perf}_S}\uHom(\sff_1(Y), \int_{Y_1\in {\rm Perf}_S} \uHom(\uHom_S(Y_1,Y)\otimes \sff_2(Y_1),\sff_3(Y_1)))\\
  & \simeq \int_{Y\in {\rm Perf}_S}\int_{Y_1\in {\rm Perf}_S}\uHom(\sff_1(Y),  \uHom(\uHom_S(Y_1,Y)\otimes \sff_2(Y_1),\sff_3(Y_1)))\\
  & \simeq \int_{Y\in {\rm Perf}_S}\int_{Y_1\in {\rm Perf}_S}\uHom(\sff_1(Y)\otimes \uHom_S(Y_1,Y)\otimes \sff_2(Y_1),\sff_3(Y_1))\\
    & \simeq \int_{Y_1\in {\rm Perf}_S}\int_{Y\in {\rm Perf}_S}\uHom(\sff_1(Y)\otimes \uHom_S(Y_1,Y)\otimes \sff_2(Y_1),\sff_3(Y_1))\\
      & \simeq \int_{Y_1\in {\rm Perf}_S}\uHom((\int^{Y\in {\rm Perf}_S}\sff_1(Y)\otimes \uHom_S(Y_1,Y))\otimes \sff_2(Y_1),\sff_3(Y_1))\\
       & \simeq \int_{Y_1\in {\rm Perf}_S}\uHom(\sff_1(Y_1)\otimes \sff_2(Y_1),\sff_3(Y_1))\\
        & \simeq \uHom_{S^{\rm top}}(\sff_1\otimes \sff_2,\sff_3).
\end{align*}
Here, the second and the  fifth isomorphisms hold because internal $\Hom$ in ${\rm Cond}$ commutes with limits and colimits in the second and the first variable, respectively; the  third one -- by adjunction in ${\rm Cond}$; the fourth one --  by  the Fubini theorem for ends; the sixth one  -- by co-Yoneda Lemma.

 For the second isomorphism of the lemma, evaluating on $Y\in {\rm Perf}_S$, it suffices to show that we have a natural isomorphism in ${\rm Cond}$ 
$$
\uHom_{S^{\rm top}} (h^{\rm top}_Y\otimes \sff_1\otimes\sff_2,\sff_3)\simeq \uHom_{S^{\rm top}}(h^{\rm top}_Y\otimes\sff_1,\Hhom_{S^{\rm top}}(\sff_2,\sff_3)).
$$
But this follows from the first isomorphism of the lemma.
\end{proof}

\vskip2mm
 ($\bullet$) {\em Generators.} The above implies the following: 
 \begin{lemma}\label{trump1} The  category of ${\rm Cond}$-presheaves 
$\underline{\rm PSh}(S)$ is generated by the family $\{h^{\rm top}_Y\otimes W_i\}, i\in I,$
 $Y\in {\rm Perf}_S$, where $\{W_i\}, i\in I$ is a family  of generators of ${\rm Cond}$. 
 \end{lemma}
\begin{proof} 
 We follow the argument in the proof of \cite[Th. 4.2]{Gr}. 
Let $\alpha_1, \alpha_2: \sff \to \sg$  be two  maps in $\underline{\rm PSh}(S)$ such that $\alpha_1\neq \alpha_2$. We want to show that there is $i \in  I$, $Y\in {\rm Perf}_S$, and a map
$\beta:  h^{\rm top}_Y\otimes W_i\to  \sff$  such that $\alpha_1\beta \neq   \alpha_2 \beta$.

 Since $\alpha_1\neq \alpha_2$, there exists $Y \in {\rm Perf}_S$ such that   $\alpha_{1,Y}\neq  \alpha_{2,Y}: \sff(Y)\to \sg(Y)$ in  ${\rm Cond}$. We fix such a $Y$. Since 
$ \{W_i\}, i\in I $, are generators of ${\rm Cond}$,  there exists a map $\overline{\beta} : W_i \to \sff(Y) $ such that $\alpha_{1,Y} \overline{\beta}\neq \alpha_{2,Y} \overline{\beta}$.
But we have natural isomorphisms
\begin{align*}\Hom_{S^{\rm top}}(h^{\rm top}_Y\otimes W_i,\sff)\simeq \Hom(W_i,\uHom_{S^{\rm top}}(h^{\rm top}_Y,\sff))\simeq\Hom(W_i,\sff(Y));
\end{align*}
the second one by the enriched Yoneda Lemma. 
Hence  we can find a unique map $\beta : h^{\rm top}_Y\otimes W_i\to \sff$ corresponding to $\overline{\beta}$.
Now $\alpha_{1,Y} \overline{\beta} \neq  \alpha_{2,Y} \overline{\beta}$ implies that $\alpha_1 \beta \neq  \alpha_2 \beta$, as wanted.
\end{proof}

  \subsection{Derived version} \label{plane1} We study here the $\infty$-derived category of topologically enriched presheaves, show that it admits a canonical ${\rm Cond}$-enrichment, and compare it with the category of ${\rm Cond}$-enriched presheaves with values in corresponding $\infty$-derived ${\rm Cond}$-categories. 
\subsubsection{Abelian and derived categories of topologically enriched presheaves} We start with the abelian categories of topologically enriched presheaves:
\begin{align}
 & \underline{\rm PSh}(S,{\rm CondAb})_0, \quad \underline{\rm PSh}(S,{\rm Solid})_0,\\
 & \underline{\rm PSh}(S,{\underline{\Q}_p})_0, \quad \underline{\rm PSh}(S,\underline{\Q}_{p,\Box})_0\notag
\end{align}
are Grothendieck abelian with compact projective generators.
Hence they  admit $K$-injective and $K$-projective
resolutions (see \cite[Tag~079I]{Stck} and
\cite[Cor.~3.5]{Spal}, respectively).
We denote their unbounded derived $\infty$-categories by
\begin{align}\label{orange1}
& \underline{\sd}(S,{\rm CondAb}),\quad
  \underline{\sd}(S,{\rm Solid}),\\
& \underline{\sd}(S,\underline{\Q}_p),\quad
  \underline{\sd}(S,\underline{\Q}_{p,\Box}).\notag
\end{align}
In the condensed cases, derived tensor products are
constructed using $K$-flat resolutions. In the solid cases,
we use the symmetric monoidal derived solidification
of Proposition~\ref{nyear1b}. In particular, for
$\sff,\sg\in\underline{\sd}(S,{\rm Solid})$, we have
\[
\sff\otimes^{\LL_{\Box}}\sg
 \simeq
\bigl(\sff\otimes^{\LL}\sg\bigr)^{\LL_{\Box}},
\]
where the tensor product on the right is formed after
viewing $\sff$ and $\sg$ as condensed enriched presheaves.
The analogous formula holds with
$\underline{\Q}_p$-coefficients, using tensor products
over $\underline{\Q}_p$.
The derived internal Hom's are the right adjoints to
these derived tensor products, giving the usual
internal tensor--Hom adjunctions.

   The first two categories in \eqref{orange1} are enriched in $\sd({\rm CondAb})$ and $\sd({\rm Solid})$, respectively; the coefficient variants are enriched in $\sd({\rm Mod}^{\rm cond}_{\Q_p(S)})$ and $\sd({\rm Mod}^{\rm solid}_{\Q_p(S)})$. For example, the enriched mapping object in the first category is
$$
\R\uHom_{S^{\rm top}}(\sff,\sg)
:=\R\Gamma(S,\R\Hhom_{S^{\rm top}}(\sff,\sg))
\quad\text{in }\sd({\rm CondAb}),
$$
with the analogous formula in the other cases. Here global sections of a presheaf mean evaluation at the terminal object $S$. The derived tensor action and its right adjoint provide this enrichment; see also the construction in Lemma \ref{denver1} below. The enriched tensor-hom adjunctions are inherited from the internal ones.

        We have the following result concerning the relationship between constant presheaves and monoidal structures:
  \begin{lemma}   Let \(S\in\mathrm{sPerf}_C\). The  claims  of Lemma \ref{form23}  hold for $\mathcal F\in\underline{\mathrm{PSh}}(S,\mathbf Q_{p,\Box}),$
with the corresponding enriched presheaf categories in place of the sheaf categories, and with
\(\mathrm R\Gamma(S_{\mathrm{pro\acute et}}^{\mathrm{top}},\mathcal F)\) replaced by \(\mathcal F(S)\). 
  \end{lemma}
  \begin{proof}
  Claim (1) of Lemma \ref{form23} has the same proof in this setting. 
  For the second claim, we use the  enriched end description 
  \begin{align*}
  \uHom^{\Box}_{S,\Q_p}(\underline{W} ,\sff)  & \stackrel{\sim}{\to}  \uHom^{\Box}_{S,\Q_p}({W} \otimes_{\Q_{p,\Box}}\underline{\Q}_p,\sff)  = \int_{Y\in {\rm sPerf}_S}  \uHom_{\Q_p(Y)}(W\otimes_{\Q_{p,\Box}}\Q_p(Y) ,\sff(Y))\\
   & \simeq \int_{Y\in {\rm sPerf}_S} \uHom_{\Q_p}(W ,\sff(Y))\simeq  \uHom_{\Q_p}(W ,\int_{Y\in {\rm sPerf}_S} \sff(Y))\simeq  \uHom_{\Q_p} (W,\sff(S)).
  \end{align*}
  Here, we wrote $\int_{Y\in {\rm sPerf}_S} \sff(Y)$ for the equalizer
$$
\xymatrix{
\prod_{Y\in {\rm sPerf}_S}\sff(Y)\ar@<-1mm>[r] \ar@<1mm>[r]  & \prod_{Y,Y_1\in {\rm sPerf}_S}\uHom(\uHom_S(Y_1,Y),\sff(Y_1))
}
$$
and  the first isomorphism follows from the first claim of the lemma. 
  
    For the remaining two claims, note that 
  evaluation at \(S\) is exact and has the exact enriched left adjoint $M\mapsto M\otimes_{\mathbf Q_p}^{\Box}\mathbf Q_p.$
Deriving this adjunction and using
\(\underline W\simeq W\otimes_{\mathbf Q_p}^{\Box}\mathbf Q_p\), we obtain
\[
\mathrm R\underline{\operatorname{Hom}}^{\Box}_{S,\mathbf Q_p}
      (\underline W,\mathcal F)
\simeq
\mathrm R\underline{\operatorname{Hom}}_{\mathbf Q_p,\Box}
      (W,\mathcal F(S)),
\]
which proves claim (3). Under the hypotheses of claim (4), the Fr\'echet-complex comparison established in the proof of Lemma \ref{form23}, applied to \(\mathcal F(S)\), gives
\[
\mathrm R\underline{\operatorname{Hom}}_{\mathbf Q_p,\Box}
      (W,\mathcal F(S))
\simeq
W^*\otimes_{\mathbf Q_p}^{\LL_{\Box}}\mathcal F(S).
\]
This proves claim (4).
  \end{proof}

\subsubsection{Condensed vs solid topologically enriched presheaves} 
 The following result is an enriched  version of Proposition \ref{nyear1a}. 
  \begin{proposition}\label{nyear1b}
  \begin{enumerate}
 \item  The forgetful functor
 \begin{equation}\label{form1b}
 \underline{{\rm PSh}}(S, {\rm Solid})\to  \underline{{\rm PSh}}(S,{\rm CondAb})
 \end{equation}
 is  (topologically) fully faithful. The essential image is
  stable under all limits, colimits, and extensions.  Moreover, the
inclusion \eqref{form1b} admits a (topological)  left adjoint, the solidification functor, 
\begin{equation}
\underline{{\rm PSh}}(S,{\rm CondAb})\to  \underline{{\rm PSh}}(S,{\rm Solid}): \quad \sff \mapsto \sff^{\Box}, 
 \end{equation}
which   preserves all colimits and is  symmetric monoidal. We have analogous claims for the forgetful functor  
${\underline{\rm PSh}}(S,\underline{\Q}_{p,\Box}) \to \underline{{\rm PSh}}(S,\underline{\Q}_p)$.
\item 
  The  forgetful functor
  \begin{equation}\label{form2b}
  \underline{\sd}(S,{\rm Solid})\to \underline{\sd}(S,{\rm CondAb})
  \end{equation}
   is  (topologically) fully faithful.   It preserves  all limits and colimits. 
\item The functor \eqref{form2b}  admits a  left adjoint
$$\sff  \mapsto  \sff^{\LL_{\Box}}: \underline{\sd}(S,{\rm CondAb}) \to \underline{\sd}(S,{\rm Solid}),
$$
which is the (topological)  left derived functor of the solidification functor $(-)^{\Box}$. It is symmetric monoidal. We have analogous claims for the forgetful functor
 $\underline{\sd}(S,\underline{\Q}_{p,\Box})\to \underline{\sd}(S,\underline{\Q}_{p})$.
\item  For  $\sff_1, \sff_2 \in \underline{\sd}(S,{\rm Solid})$, the natural map
$$\R{\Hhom}^{\Box}_{S}(\sff_1,\sff_2) \to  \R{\Hhom}_{S^{\rm top}}(\sff_1,\sff_2)$$
is a quasi-isomorphism. We have an analogous claim for the $\underline{\Q}_p$-sheaves.
\end{enumerate}
  \end{proposition}
  \begin{proof} The proof of  Proposition \ref{nyear1a} goes through almost verbatim (we just need to replace $h^{\delta}_Y$ with $h^{\rm top}_Y$) since we have Lemma \ref{prodiscrete} (for the second claim). 
      \end{proof}
\begin{remark} (1) The constructions in Section \ref{ias1} and above in Section \ref{plane1} apply as well with ${\rm Perf}_C$ replaced by the category ${\rm sPerf}_C$ of strictly totally disconnected affinoids.  {\bf In what follows we will use the category ${\rm sPerf}_C$ as test objects.}

 (2) It is possible to define a notion of a  topologically enriched \'etale sheaf as a  finite-products preserving enriched functor. Such objects form a  full reflective subcategory of topologically enriched presheaves in the sense of \cite{BAC}, \cite{Ros} (the key point being that the naive sheafification of a topologically enriched  presheaf is canonically enriched). In fact in the first draft of this paper we have worked in this setting only later realizing that this does not add anything (besides extra layer of complexity) to the applications we had in mind. It might be however useful in the future. 
 \end{remark}
\begin{definition}
We will call presheaves from $\underline{{\rm PSh}}({\rm Spa}(C),\underline{\Q}_{p,\Box})$ {\em Topological Vector Spaces} (${\rm TVS}$ for short). Explicitly, these are modules over the enriched ring presheaf $\underline{\Q}_p$ in $\underline{{\rm PSh}}({\rm Spa}(C),{\rm Solid})$, evaluated on strictly totally disconnected affinoids over $C$. Thus the values are solid $\Q_p(Y)$-modules, and the restrictions carry the compatible enrichment\footnote{The underline on ${\rm PSh}$ records this enrichment, which is the extra condition relative to the ordinary presheaf categories of Section \ref{naive1}.} of \eqref{str2}. 
 \end{definition}
 \begin{remark}We will also use the  term ${\rm TVS}$  for the derived category $\underline{\mathcal{D}}({\rm Spa}(C),\underline{\Q}_{p,\Box})$ of ${\rm TVS}$'s  and sometimes, abusively, for the analogous categories with values in condensed Abelian groups,    etc (if this does not cause confusion).
\end{remark}

  \subsubsection{Notation} In the rest of this section, unless otherwise stated, "category" etc. means "$\infty$-category" etc. 
   We call  groupoids  "spaces" and write $\sss$ for their
category.  For presentable $\infty$-categories $\scc_1,\scc_2$, we denote by ${\rm Fun}^{\LL}(\scc_1,\scc_2)$  the full $\infty$-subcategory of the
functor category ${\rm Fun}(\scc_1,\scc_2) $ of the functors which preserve colimits.
  We call algebra objects
$\sv\in  {\rm Alg}({\rm Pr}) $  {\em presentably monoidal categories}.  In plain language, these are monoidal categories $\sv$
whose underlying categories are presentable and, moreover,   the tensor product functors $ \sv\times\sv\to \sv$ 
preserve colimits separately in each  variable. For $\sv\in  {\rm Alg}({\rm Pr})$, we set ${\rm Pr}_{\sv}:= {\rm RMod}_{\sv}({\rm Pr})$.  We call  module objects $\sm \in  {\rm RMod}_{\sv}({\rm Pr})$  {\em presentable
right $\sv$-module categories} (and similarly for left modules).
   \subsubsection{Enriched categories} In this section, we recall a particularly friendly definition of enriched $\infty$-categories\footnote{In the terminology of \cite{RZ},
these are valent enriched categories. We will simply call them
enriched \(\infty\)-categories. By  \cite[Prop. 4.5.3]{Hin20}, \cite{Mac21}  this definition agrees with the definition in
\cite{GH15}.}
due to Hinich (see  \cite{Hin20}) following the exposition of  Reutter--Zetto from \cite[Sec. 2.4]{RZ}. For  $\sv \in {\rm  Alg}({\rm Pr})$,  a $\sv$-enriched category is defined 
as a space $\sx$ (of objects) together with an algebra in the category ${\rm Fun}(\sx^{\rm op} \times  \sx, \sv)$; the latter being 
equipped with a certain  monoidal structure.

   More specifically, we have the following: 
 \begin{definition}(\cite[Prop. 4.5.3]{Hin20}) \label{def2} Let $\sv\in {\rm  Alg}({\rm Pr})$ and $\sx \in  \sss$. The category of
valent  $\sv$-enriched categories with space of objects $\sx$ is the category
$${\rm Cat}_{\sx}(\sv) := {\rm Alg}({\rm End}^{\LL}_{\sv}(\spp(\sx) \otimes \sv)) .$$
It is a presentable category. 
\end{definition}
Here ${\rm End}^{\LL}_{\sv}(\spp(\sx) \otimes \sv)$ comes equipped with the composition monoidal structure defined via the equivalences
\begin{equation}\label{kicius1}{\rm End}^L_{\sv}(\spp(\sx) \otimes \sv) \simeq {\rm Fun}^L_{\sv}(\spp(\sx) \otimes \sv,{\rm Fun}(\sx^{\rm op}, \sv))\simeq  
{\rm Fun}(\sx^{\rm op} \times  \sx, \sv) \in  {\rm Alg}({\rm Pr}),
\end{equation}
and the  monoidal structure on ${\rm Fun}(\sx^{\rm op} \times  \sx, \sv)$ given by composition of functors (see \cite[Cor. 2.29]{RZ} for details).

\begin{remark} (1) The first equivalence in \eqref{kicius1} follows from the equivalence
$$
\spp(\sx) \otimes \sv\simeq {\rm Fun}(\sx^{\rm op}, \sv).
$$
For the second one,  recall that 
  we can think of 
$\spp(\sx) \otimes \sv$ as the free presentable $\sv$-module category on $\sx$: we have  an equivalence
$$ {\rm Fun}^{\LL}_{\sv}(\spp(\sx) \otimes \sv, \sm) \simeq {\rm Fun}(\sx, \sm),$$
for any   $\sm\in {\rm Pr}_{\sv}$.
This equivalence is obtained  via 
 precomposing with
the composite functor
$$\sx \to \spp(\sx) \simeq \spp(\sx) \otimes \sss \to \spp(\sx) \otimes \sv,$$
where the third map comes from the unit   $\sss\to  \sv$,

   (2) Via the equivalences
$${\rm End}^{\LL}_{\sv}(\spp(\sx) \otimes  \sv)\simeq  {\rm Fun}(\sx^{\rm op} \times  \sx, \sv) \simeq  \spp(\sx \times \sx^{\rm op}) \otimes  \sv\simeq  {\rm End}^{\LL}(\spp(\sx)) \otimes  \sv,
$$
the composition monoidal structure from Definition \eqref{def2}  agrees with the induced monoidal
structure as a product of algebras ${\rm End}^{\LL}(\spp(\sx)) $ and $\sv$ in ${\rm Alg}({\rm Pr}) $ by \cite[Prop. 3.10]{BMS24}.

 (3)  The {\em graph}  of a  $\sv$-enriched category $\se \in {\rm Cat}_{\sx}(\sv) $ is its underlying
object in  ${\rm Fun}(\sx^{\rm op}\times\sx, \sv) \simeq {\rm End}^{\LL}_{\sv}(\spp(\sx) \otimes  \sv).$ We will denote it by $\uHom_{\se}(-,-)$. 

\end{remark}
\subsubsection{Enriched presheaves and enriched  Yoneda Lemma} Here again we follow the exposition from \cite[Sec. 2.6]{RZ}.  By definition, a $\sv$-enriched
category $\se$ with space of objects $\sx$ is an algebra in the category
${\rm End}^{\LL}_{\sv}(\spp(\sx)\otimes \sv). $ Since this is the endomorphism algebra of  $\spp(\sx)\otimes \sv \in  {\rm Pr}_{\sv}$ and
hence acts on it  from the left, we may take the category of left $\se$-modules
in $\spp(\sx) \otimes \sv$:
\begin{definition}  Let $\sx \in \sss, \sv \in {\rm Alg}({\rm Pr})$. For $\se \in {\rm Cat}_{\sx}(\sv)= {\rm Alg}({\rm End}^{\LL}_\sv(\spp(\sx) \otimes  \sv))$,
its {\em enriched presheaf category} 
$$\spp_{\sv}(\se) := {\rm LMod}_{\se}(\spp(\sx) \otimes  \sv).$$  It is a presentable right $\sv$-module category, i.e.,  an object of ${\rm Pr}_{\sv}$.
\end{definition}
Hence, an enriched presheaf of a $\sv$-category $\se$ with
space of objects $\sx$ is  an object of ${\rm LMod}_{\se}(\spp(\sx) \otimes \sv) \simeq  {\rm LMod}_{\se}({\rm Fun}(\sx^{\rm op}, \sv)),$
i.e., a functor $\sff : \sx^{\rm op}\to \sv$  equipped  with a left action of the algebra 
 $\uHom_{\se} \in {\rm Fun}(\sx^{\rm op} \times  \sx, \sv)$  via the canonical action of $ {\rm Fun}(\sx^{\rm op} \times  \sx, \sv) = {\rm End}^{\LL}_{\sv}({\rm Fun}(\sx^{\rm op}, \sv)) $ on
${\rm Fun}(\sx^{\rm op}, \sv)$. This amounts to morphisms in
 $ \sv$
$$\uHom_{\se}(X, X^{\prime}) \otimes \sff(X^{\prime}) \to \sff(X),$$
which are functorial  in $ X, X^{\prime} \in  \sx$ and coherently compatible with the composition in $\se$.
 
    The {\em enriched Yoneda functor} of $\se$ is the composition
$$h^{\sv}_{\se} : \sx\lomapr{h_{\sx}} \spp(\sx)
\lomapr{{\rm Id}\otimes 1_{\sv}} \spp(\sx) \otimes \sv\lomapr{\rm Free} {\rm  LMod}_{\se}(\spp(\sx)\otimes \sv) = \spp_{\sv}(\se),$$
where $h_{\sx}$ denotes the ordinary Yoneda functor of $\sx$ and ${\rm Free}$  the free $\se$-module functor.
 For $X \in  \sx$, 
 the evaluation functor ${\rm ev}_X : \spp_{\sv}(\se) \to \sv $ is  the composition
$$\spp_{\sv}(\se) = {\rm LMod}_{\se}(\spp(\sx) \otimes \sv)\lomapr{\rm Forget}
 \spp(\sx) \otimes  \sv \simeq  {\rm Fun}(\sx^{\rm op}, \sv) \lomapr{{\rm ev}_X} \sv.$$
In plain language, the evaluation functor ${\rm ev}_X : \spp_{\sv}(\se) \to  \sv$ sends $\sff$ to $\sff(X)$
and the  presheaf $h^{\sv}_{\se}(X) $ unpacks to the functor $\uHom_{\se}(-, X) : \sx^{\rm op} \to  \sv.$

\begin{lemma}{\rm({\em Enriched Yoneda Lemma})} \label{enr-kol} The functor
${\rm ev}_X : \spp_{\sv}(\se) \to \sv $ is right adjoint to the functor $h^{\sv}_{\se}(X) \otimes - : \sv\to   \spp_{\sv}(\se)$  and  thus it is
equivalent to $ \uHom_{\spp_{\sv}(\se)}(h^{\sv}_{\se}(X),-):$
we have
$$\uHom_{\spp_{\sv}(\se)}(h^{\sv}_{\se}(X),\sff)\simeq   {\rm ev}_X(\sff), \quad \sff\in \spp_{\sv}(\se).$$
\end{lemma}
In particular,  for $X,X^{\prime}\in\sx$, we have 
$$\uHom_{\spp_{\sv}(\se)}(h^{\sv}_{\se}(X),h^{\sv}_{\se}(X^{\prime}))\simeq  \uHom_{\se}(X, X^{\prime}).$$
\subsubsection{Generators} We will use the following terminology. 
 Let $\scc \in  {\rm Pr}$.  A full subcategory $\scc_0\subset \scc$ {\em generates} $\scc$ under
colimits if $\scc$ is the smallest full subcategory of $\scc$ containing  $\scc_0$ and  closed under small
colimits. Let  $\sv\in  {\rm Alg}({\rm Pr}) $ and $\scc \in  {\rm Pr}_{\sv}$. A full subcategory $\scc_0 \subset  \scc$ {\em generates
$\scc$  under colimits and tensoring} if $ \scc$ is the smallest full subcategory of $ \scc $ containing 
$\scc_0 $ and is closed under small colimits and tensoring with objects of $\sv$. 
 
   By \cite[Th. 5.5]{RZ}, for $ \se \in {\rm  Cat}_{\sx}(\sv)$
the presheaf category $\spp_{\sv}(\se)$ is generated under colimits and tensoring by the image of the Yoneda functor
$h^{\sv}_{\se}:  \sx\to  \spp_{\sv}(\se).$ 
\subsubsection{Example: Derived topologically enriched presheaves} The setting we will be mostly  interested in is the following: We set
$
\sv:=\sd({\rm CondAb}), S\in {\rm sPerf}_C. $ We take for $\sx$ the discrete groupoid ${\rm sPerf}^{\delta}_S$ of objects from ${\rm sPerf}_S$ with only identity morphisms. 
We obtain $\se\in{\rm Cat}_{\sx}(\sv)$ from the
${\rm Cond}$-enriched category ${\rm sPerf}_S$
by change of enrichment along the symmetric monoidal
functor
\[
\Z[-]:({\rm Cond},\times)\longrightarrow
(\sd({\rm CondAb}),\otimes^{\LL}),
\qquad T\longmapsto\Z[T][0].
\]
Thus its underlying graph is
$\uHom_{\se}(X,X')=\Z[\uHom_S(X,X')]$,
and its composition and unit maps are induced by
those of ${\rm sPerf}_S$.
 We have $h^{\sv}_{\se}(X)=\Z[h^{\rm top}_X]$, for $X\in\sx$. We set
$$
\underline{{\rm PSh}}(S,\sd({\rm CondAb}))_0:=\spp_{\sv}(\se). 
$$
Hence, a topologically  enriched presheaf is  a functor $\sff : {\rm sPerf}_S^{\delta, \rm op}\to \sd({\rm CondAb})$  together with  morphisms in $\sd({\rm CondAb})$
\begin{equation}\label{berlin1}
\Z[\uHom_{S}(X, X^{\prime})] \otimes^{\LL} \sff(X^{\prime}) \to \sff(X),
\end{equation}
which are  coherently compatible with the composition in $\se$. The category $\underline{{\rm PSh}}(S,\sd({\rm CondAb}))_0$ has  a canonical presentable right $\sd({\rm CondAb})$-module structure. 
It  is generated under colimits and tensoring by the presheaves $\Z[h^{\rm top}_X]$, $X\in \sx$; these generators are compact. 
We equip this category with the pointwise symmetric
monoidal structure
\[
(\sff\otimes^{\LL}\sg)(Y)
 :=\sff(Y)\otimes^{\LL}\sg(Y).
\]
To define the enrichment on this tensor product, put
$H_{X,X'}:=\Z[\uHom_S(X,X')]$.
The diagonal of the condensed mapping object induces
a comultiplication
\[
H_{X,X'}\longrightarrow
H_{X,X'}\otimes^{\LL}H_{X,X'}.
\]
The action on $\sff\otimes^{\LL}\sg$ is the composite
\[
H_{X,X'}\otimes^{\LL}
  \bigl(\sff(X')\otimes^{\LL}\sg(X')\bigr)
\longrightarrow
\bigl(H_{X,X'}\otimes^{\LL}\sff(X')\bigr)
 \otimes^{\LL}
\bigl(H_{X,X'}\otimes^{\LL}\sg(X')\bigr)
\longrightarrow
\sff(X)\otimes^{\LL}\sg(X).
\]
The diagonals and augmentations are compatible with
composition and units, and give the required coherent
symmetric monoidal structure. Its unit is the constant
enriched presheaf with value $\Z$, and its tensor
product preserves colimits separately in each variable.
The analogous construction with discrete mapping
objects gives the pointwise symmetric monoidal structure
for ordinary topological presheaves.
\begin{lemma}\label{denver1}
We can upgrade $\underline{{\rm PSh}}(S,\sd({\rm CondAb}))_0$ to a $\sd({\rm CondAb})$-enriched category, which  we will denote by $\underline{{\rm PSh}}(S,\sd({\rm CondAb}))$.
\end{lemma}
\begin{proof}
For every
$\sff\in\underline{{\rm PSh}}(S,\sd({\rm CondAb}))_0$,
the functor
\[
(-)\otimes^{\LL}\sff:
\sd({\rm CondAb})\longrightarrow
\underline{{\rm PSh}}(S,\sd({\rm CondAb}))_0
\]
preserves colimits, by the presentable module structure
constructed above. Since both categories are presentable,
it admits a right adjoint (see \cite[Th.~A.5.7]{HM}).
We denote this right adjoint by
$\R\uHom_{S^{\rm top}}(\sff,-)$.
By \cite[Ex.~C.1.11]{HM}, these right adjoints give the
required $\sd({\rm CondAb})$-enrichment.
\end{proof}
\begin{remark}
By construction, for $W\in\sd({\rm CondAb})$, we have
a natural quasi-isomorphism
\[
\R\Hom_{S^{\rm top}}
 \bigl(\sff\otimes^{\LL}W,\sg\bigr)
\simeq
\R\Hom_{{\rm CondAb}}
 \bigl(W,\R\uHom_{S^{\rm top}}(\sff,\sg)\bigr).
\]
In particular, for an extremally disconnected profinite
set $T$, exactness of evaluation at $T$ gives
\[
\bigl(\R\uHom_{S^{\rm top}}(\sff,\sg)\bigr)(T)
\simeq
\R\Hom_{S^{\rm top}}
 \bigl(\sff\otimes^{\LL}\Z[T],\sg\bigr).
\]
For an arbitrary profinite set $T$, the same formula
holds with derived sections
$\R\Gamma(T,\R\uHom_{S^{\rm top}}(\sff,\sg))$
on the left, where
\[
\R\Gamma(T,K):=
\R\Hom_{{\rm CondAb}}(\Z[T],K),
\qquad K\in\sd({\rm CondAb}).
\]
\end{remark}
For internal $\Hom$, we set
  $$
  \R\Hhom_{S^{\rm top}}(\sff,\sg):=\{Y\mapsto\R\uHom_{S^{\rm top}}(\sff\otimes^{\LL}\Z[h^{\rm top}_Y],\sg)\}.
  $$
It is a right adjoint to the tensor product $\sff\otimes^{\LL}\sg$ (this is easy to check by looking at the colimit presentation of presheaves via the standard generators).

If we replace the graph $\Z[\uHom_S(-,-)]$ with the graph $\Z[\Hom_S(-,-)]$, we get the  category of topological presheaves ${\rm PSh}(S,\sd({\rm CondAb}))_0$. In plain language, a topological  presheaf is  a functor $\sff : {\rm sPerf}_S^{\delta, \rm op}\to \sd({\rm CondAb})$  together with  morphisms in $\sd({\rm Ab})$

\begin{equation}\label{berlin2}
\Z[\Hom_{S}(X, X^{\prime})] \to \R\Hom(\sff(X^{\prime}), \sff(X)),
\end{equation}
which are  coherently compatible with the composition in $\se$.
We have the forgetful functor
\begin{equation}\label{zmecz1}
(-)^{\rm cl} :\quad \underline{ {\rm PSh}}(S,\sd({\rm CondAb}))_0 \to {\rm PSh}(S,\sd({\rm CondAb}))_0
\end{equation}
given by evaluating the tensor--Hom adjoints of the
structure maps \eqref{berlin1} at $*$ to obtain
the structure maps \eqref{berlin2}. The category ${\rm PSh}(S,\sd({\rm CondAb}))_0$  is generated under colimits and tensoring by the presheaves $\Z[h^{\delta}_X]$, $X\in \sx$; these generators are compact. Also, as above, we see that the category 
${\rm PSh}(S,\sd({\rm CondAb}))_0$ is naturally  $\sd({\rm CondAb})$-enriched and equipped with internal $\Hom$'s. We will denote the enriched version by ${\rm PSh}(S,\sd({\rm CondAb}))$; the morphism \eqref{zmecz1} lifts to the enriched categories.

\vskip2mm
   ($\bullet$) {\em $\underline{\Q}_p$-modules.}  We denote by\footnote{We do apologize for the unusual notation.}
\begin{equation}
\underline{\rm PSh}(S,\sd({\rm Mod}^{\rm cond}_{\underline{\Q}_p}))_0:={\rm Mod}_{\underline{\Q}_p}(\underline{\rm PSh}(S,\sd({\rm CondAb}))_0)
\end{equation}
 the category of  $\underline{\Q}_p$-modules in the category $\underline{\rm PSh}(S,\sd({\rm CondAb}))_0$.
 The underlying categories are built from presheaves $\sff$ equipped with  actions
 $
 \underline{\Q}_p\otimes \sff\to \sff
 $
 that are coherently compatible with the structure maps.   We will denote by  $\R\uHom_{S^{\rm top},\Q_p}(-,-)$ the $\Hom$-object.  The category $\underline{{\rm PSh}}(S, \sd({\rm Mod}^{\rm cond}_{\underline{\Q}_p}))_0$  is generated under colimits and tensors (by elements from $\sd({\rm Mod}^{\rm cond}_{\Q_p(S)})$) by presheaves 
$
\{\underline{\Q}_p[h^{\rm top}_X]\},
$
for  $ X\in \sx$.

 The category
$\underline{{\rm PSh}}
 (S,\sd({\rm Mod}^{\rm cond}_{\underline{\Q}_p}))_0$
is a presentable module category over
$\sd({\rm Mod}^{\rm cond}_{\underline{\Q}_p(S)})$.
The same right-adjoint construction as in the proof of Lemma \ref{denver1}  gives
its canonical enrichment over this coefficient category.
We denote the enriched category by
$\underline{{\rm PSh}}
 (S,\sd({\rm Mod}^{\rm cond}_{\underline{\Q}_p}))$
and its mapping objects by
$\R\uHom_{S^{\rm top},\Q_p}(\sff,\sg)$.

For every extremally disconnected profinite set $T$,
adjunction gives
\[
\bigl(\R\uHom_{S^{\rm top},\Q_p}
       (\sff,\sg)\bigr)(T)
\simeq
\R\Hom_{S^{\rm top},\Q_p}
 \bigl(
 \sff\otimes^{\LL}_{\underline{\Q}_p}
       \underline{\Q}_p[T],\sg
 \bigr),
\]
where
$\underline{\Q}_p[T]
 :=\underline{\Q}_p\otimes_{\Z}\Z[T]$.
For an arbitrary profinite set $T$, the same formula
holds with
$\R\Gamma(T,\R\uHom_{S^{\rm top},\Q_p}(\sff,\sg))$
on the left, with derived sections taken on the
underlying condensed complex.

  \begin{remark}
(1) All of the above has a solid version: we get the $\sd({\rm Solid})$-category $\underline{\rm PSh}(S,\sd({\rm Solid}))$ and the $\sd(\underline{\Q}_{p}(S)_{\Box})$-category $\underline{\rm PSh}(S,\sd(\underline{\Q}_{p,{\Box}}))$.

 (2) An argument analogous to the one used in the proof of Lemma \ref{denver1} shows that $\underline{\sd}(S,{\rm CondAb})$ has a canonical structure of a $\sd({\rm CondAb})$-enriched category. Similarly, for the related categories.
\end{remark}
\subsubsection{Example: Derived rigid topologically enriched  presheaves}  \label{plan10} Let $S\in {\rm sPerf}_C$ and let $\sff\in \underline{{\rm PSh}}(S,{\sd}({\rm CondAb}))$. Proceeding as in the proof of Lemma \ref{goraco1} we can functorially associate to $\sff$ 
structure maps in ${\sd}({\rm CondAb})$:
\begin{equation}\label{goraco2}
\sff^{\prime}_{Y,T}:\quad \sff(Y\times T)\to \R[\Z[T],\sff(Y)],
\end{equation}
for all $Y\in {\rm sPerf}_S$ and $T\in {\rm ProFin}$. Here we set $\R[\Z[T],\sff(Y)]:=\R\uHom(\Z[T],\sff(Y))$. 
\begin{definition}
A derived topologically enriched presheaf $\sff\in  \underline{{\rm PSh}}(S,{\sd}({\rm CondAb}))$ is called {\em rigid} if  the structure maps from \eqref{goraco2} 
$$
\underline{\sff}^{\prime}_{Y,T}:\quad  \sff(Y\times T)\to  \R[\Z[T], \sff(Y)]
$$
are quasi-isomorphisms in ${\sd}({\rm CondAb})$, for all $Y\in {\rm sPerf}_S$, $ T\in {\rm ProFin}$. This cuts out a full sub-$\infty$-category $\underline{{\rm PSh}}^{\rm rig}(S,{\sd}({\rm CondAb}))\subset \underline{{\rm PSh}}(S,{\sd}({\rm CondAb}))$.

   Similarly  with  ${\rm CondAb}$ replaced with ${\rm Solid}$.
\end{definition}
The derived realizations of complexes of sheaves on
$S_{\proeet}$ are canonically derived rigid:

\begin{lemma}\label{enriched1c}
Let $S\in{\rm sPerf}_C$ and let
$\sff\in\sd(S_{\proeet},{\rm Ab})$.
Then the derived topologically enriched presheaf
\[
\R\pi_*\sff\in
\underline{{\rm PSh}}(S,\sd({\rm CondAb}))
\]
is derived rigid. 
\end{lemma}
\begin{proof}
Set $\sg:=\R\pi_*\sff$, viewed as a derived enriched
presheaf by applying $\pi_*$ to a $K$-injective
representative of $\sff$.
Fix  $Y\in{\rm sPerf}_S$ and  a profinite set $T$. For every extremally
disconnected profinite set $U$,  the derived
tensor--Hom adjunction yields
\begin{align*}
\bigl(\R[\Z[T],\sg(Y)]\bigr)(U)
&\simeq
\R\Hom_{{\rm CondAb}}
 \bigl(\Z[U]\otimes^{\LL}\Z[T],\sg(Y)\bigr)\\
&\simeq
\R\Hom_{{\rm CondAb}}
 \bigl(\Z[U\times T],\sg(Y)\bigr)\\
&\simeq
\R\Gamma_{\proeet}(Y\times T\times U,\sff)\simeq \sg(Y\times T)(U).
\end{align*}
Here we used the flatness of free condensed abelian
groups,  the exactness of evaluation at $U$, and the formula 
\[
\R\Hom_{{\rm CondAb}}(\Z[V],\sg(Y))
 \simeq \R\Gamma_{\proeet}(Y\times V,\sff),
\]
for  any profinite set $V$.
Under these identifications, the structure morphism
\[
\sg^{\prime}_{Y,T}:
\sg(Y\times T)\longrightarrow \R[\Z[T],\sg(Y)]
\]
from \eqref{goraco2} evaluates to the identity on
$\R\Gamma_{\proeet}(Y\times T\times U,\sff)$.
Since evaluation on extremally disconnected profinite
sets detects quasi-isomorphisms in $\sd({\rm CondAb})$,
this proves the claim.
\end{proof}
 Lemma~\ref{enriched1c}, together with
Example~\ref{ex1} and the acyclicity statement of
Lemma~\ref{acyclic1}, shows that the topologically
enriched presheaves $\underline{\Q}_p$ and
${\mathbb G}_a$, as well as the period presheaves
${\mathbb B}^+_{\dr}/t^i$ and
${\mathbb B}_{\crr}^{+,\phi=p^i}$, $i\geq 0$,
are derived rigid.

\subsubsection{$t$-structure} We start the discussion of $t$-structures with the case of topological presheaves. 
We induce a $t$-structure on ${\rm PSh}(S,\sd({\rm CondAb}))_0$ from the one on $\sd({\rm CondAb})$.  We have the equivalence of abelian categories
$$\iota: \quad {\rm PSh}(S,{\rm CondAb})_0\stackrel{\sim}{\to}{{\rm PSh}}(S,\sd({\rm CondAb}))^{\heartsuit}_0.
$$
In this subsection, the notation
${\rm PSh}(S,\sd^{-}({\rm CondAb}))_0$
denotes the full subcategory of
${\rm PSh}(S,\sd({\rm CondAb}))_0$
consisting of uniformly bounded-above objects:
there exists $n\in\Z$ such that
$H^i(\sff(Y))=0$ for every $Y\in{\rm sPerf}_S$
and every $i>n$.
We use the same convention for enriched presheaves
and for the coefficient variants below.
\begin{lemma}\label{heart1}
The canonical extension of the map $\iota$,
\[
\iota_{\sd}:\quad
\sd^{-}(S,{\rm CondAb})_0
\longrightarrow
{\rm PSh}(S,\sd^{-}({\rm CondAb}))_0,
\]
is a symmetric monoidal equivalence.
It is compatible with the
$\sd({\rm CondAb})$-valued enriched mapping objects
induced from the ambient unbounded categories.
\end{lemma}
\begin{proof}
For the first claim, note that the objectwise $t$-structure
is left complete (since  $\sd({\rm CondAb})$ is left
complete). Hence, by \cite[Prop.~1.3.3.7]{HA}, it  suffices to show
that
\[
H^i\R\Hom_{S^{\rm top}}(\sff_1,\sff_2)=0,
\qquad i>0,
\]
for $\sff_1,\sff_2\in{\rm PSh}(S,{\rm CondAb})_0$
with $\sff_1$ projective.

 For that, it suffices to consider the projective generators
$\sff_1=\Z[h_X^\delta]\otimes\Z[T]$, where $X\in\sx$
and $T$ is extremally disconnected: every projective
object is a retract of a direct sum of these generators,
and mapping complexes take such sums to products.
We compute in $\sd({\rm Ab})$:
\begin{align*}
\R\Hom_{S^{\rm top}}
 \bigl(\Z[h_X^\delta]\otimes\Z[T],\sff_2\bigr)
&\simeq
\R\Hom_{S^{\rm top}}
 \bigl(\Z[h_X^\delta]\otimes^{\LL}\Z[T],\sff_2\bigr)\\
&\simeq
\R\Hom_{{\rm CondAb}}
 \bigl(\Z[T],\sff_2(X)\bigr)
\simeq \sff_2(X)(T)[0].
\end{align*}
The first quasi-isomorphism uses the flatness of
$\Z[T]$, the second is the mapping-complex form of
the enriched Yoneda adjunction
(Lemma~\ref{enr-kol}), and the last follows from the
projectivity of $\Z[T]$ in ${\rm CondAb}$.
This proves the required vanishing.

  Compatibility with tensor products and with the
$\sd^{-}({\rm CondAb})$-action follows by computing
with bounded-above resolutions by direct sums of the
flat generators $\Z[h_X^\delta]\otimes\Z[T]$:
the functor $\iota_{\sd}$ carries these generators
to their counterparts in the target category.

For compatibility with enriched mapping objects,
let $\sff,\sg\in\sd^{-}(S,{\rm CondAb})_0$.
For every extremally disconnected profinite set $T$,
the evaluation formula, the preceding compatibility,
and full faithfulness give
\begin{align*}
\bigl(
\R\uHom_{S^{\rm top}}
 (\iota_{\sd}\sff,\iota_{\sd}\sg)
\bigr)(T)
&\simeq
\R\Hom_{S^{\rm top}}
 \bigl(
 \iota_{\sd}(\sff\otimes^{\LL}\Z[T]),
 \iota_{\sd}\sg
 \bigr)\\
&\simeq
\R\Hom_{S^{\rm top}}
 \bigl(\sff\otimes^{\LL}\Z[T],\sg\bigr)\simeq
\bigl(\R\uHom_{S^{\rm top}}(\sff,\sg)\bigr)(T).
\end{align*}
Here the enriched mapping objects are formed in the
ambient unbounded categories.
Since evaluation on extremally disconnected profinite
sets detects quasi-isomorphisms in
$\sd({\rm CondAb})$, this proves the claimed
compatibility.
  \end{proof}

  We pass now to topologically enriched presheaves.

We induce a $t$-structure on $\underline{\rm PSh}(S,\sd({\rm CondAb}))_0$ from the one on ${\rm PSh}(S,\sd({\rm CondAb}))_0$ (hence from $\sd({\rm CondAb})$).  The fact that the classical truncation functors yield actually a $t$-structure could be checked easily using the fact that the graph $\Z[\uHom(-,-)]$ is flat.   We have the equivalence of abelian categories
$$\iota: \quad \underline{\rm PSh}(S,{\rm CondAb})_0\stackrel{\sim}{\to}\underline{\rm PSh}(S,\sd({\rm CondAb}))^{\heartsuit}_0.
$$
\begin{lemma}\label{heart2}
The canonical extension of the map $\iota$,
\[
\iota_{\sd}:\quad
\underline{\sd}^{-}(S,{\rm CondAb})_0
\longrightarrow
\underline{{\rm PSh}}(S,\sd^{-}({\rm CondAb}))_0,
\]
is a symmetric monoidal equivalence.
It is compatible with the
$\sd({\rm CondAb})$-valued enriched mapping objects
induced from the ambient unbounded categories.
\end{lemma}
\begin{proof}
The objectwise $t$-structure is again left complete,
since limits and truncations are computed objectwise.
We apply the proof of Lemma~\ref{heart1}, replacing
$\Z[h_X^\delta]\otimes\Z[T]$ by
$\Z[h_X^{\rm top}]\otimes\Z[T]$.
The enriched Yoneda adjunction gives the same
ordinary mapping-complex calculation, and hence the
required vanishing.
The compatibility claims follow by the same
argument as in Lemma~\ref{heart1}.
\end{proof}
\begin{remark}\label{heart-unbounded}
(1) The analogues of Lemmas~\ref{heart1} and \ref{heart2}
hold with ${\rm Solid}$ in place of ${\rm CondAb}$,
and with $\underline{\Q}_p$-coefficients.

For the solid versions, let $h_X$ denote $h_X^\delta$
in the first case and $h_X^{\rm top}$ in the second,
and set
\[
P_{X,T}:=\Z[h_X]\otimes\Z[T],
\]
where $T$ is extremally disconnected.
The values of $h_X$ are prodiscrete
(use Lemma~\ref{prodiscrete}).
Hence Lemma~\ref{paris20}, applied to
$h_X(Y)\times T$, yields
$P_{X,T}^{\LL_{\Box}}\simeq P_{X,T}^{\Box}.$
The objects $P_{X,T}^{\Box}$ form a family of compact
projective generators of the corresponding abelian
category of solid presheaves. Derived solidification
adjunction and the Yoneda Lemma give, for $\sff$ in
its heart,
\[
\R\Hom(P_{X,T}^{\Box},\sff)
 \simeq
\R\Hom_{{\rm CondAb}}(\Z[T],\sff(X))
 \simeq \sff(X)(T).
\]
Here the first derived Hom is the ordinary mapping
complex in the corresponding category of presheaves
with values in $\sd({\rm Solid})$.
Compatibility with tensor products follows from
symmetric monoidal derived solidification, as in
Proposition~\ref{nyear1b}.

For $\underline{\Q}_p$-coefficients, the same argument
uses the free $\underline{\Q}_p$-module presheaves and,
in the solid case, the coefficient version of
Lemma~\ref{paris20}.

 (2) The equivalences of Lemmas~\ref{heart1}
and \ref{heart2} extend to $t$-exact symmetric monoidal
equivalences of the unbounded categories:
\[
\begin{aligned}
\sd(S,{\rm CondAb})_0
&\xrightarrow{\sim}
{\rm PSh}(S,\sd({\rm CondAb}))_0,\\
\underline{\sd}(S,{\rm CondAb})_0
&\xrightarrow{\sim}
\underline{{\rm PSh}}(S,\sd({\rm CondAb}))_0.
\end{aligned}
\]
These equivalences are compatible with the
$\sd({\rm CondAb})$-enrichments and with internal
Hom's.

Indeed, let $\mathcal A$ denote either abelian presheaf
category and $\mathcal C$ the corresponding category
of presheaves with values in $\sd({\rm CondAb})$.
Writing $\iota_{\sd}^{-}$ for the bounded-above
equivalence, define
\[
\iota_{\sd}(K):=
\operatorname*{colim}_{n\to+\infty}
\iota_{\sd}^{-}(\tau^{\leq n}K),
\qquad K\in\sd(\mathcal A),
\]
where the truncations use cohomological indexing.
Its value at $X$ identifies with $K(X)$.
Since evaluations are conservative and preserve
colimits, this functor is exact and preserves colimits.

 Set
\[
P_{X,T}:=\Z[h_X]\otimes\Z[T],
\]
where $h_X=h_X^\delta$ or $h_X=h_X^{\rm top}$,
respectively, and $T$ is extremally disconnected.
These objects are compact projective generators of
$\mathcal A$, hence compact generators of
$\sd(\mathcal A)$.
For every $M\in\mathcal C$, the Yoneda adjunction gives
\[
\R\Hom_{\mathcal C}(P_{X,T},M)
 \simeq
\R\Hom_{{\rm CondAb}}(\Z[T],M(X))
 \simeq M(X)(T)
\]
in $\sd({\rm Ab})$.
Hence,  the $P_{X,T}$ are also compact generators
of $\mathcal C$: the displayed functors preserve
coproducts and jointly detect zero objects.
The same formula computes maps from $P_{X,T}$ in
$\sd(\mathcal A)$, and the realization functor
respects these identifications.

Thus the realization is fully faithful on maps from
the generators into arbitrary objects. Since the
generators generate $\sd(\mathcal A)$ under colimits
and shifts, full faithfulness follows for all objects.
Its essential image contains the generators of
$\mathcal C$ and is closed under colimits and shifts,
so the realization is an equivalence.

Compatibility with tensor products and with the
coefficient action follows from the bounded-above
compatibilities by expressing both factors as filtered
colimits of their bounded-above truncations.
Compatibility with enriched mapping objects and
internal Hom's then follows from their defining
adjunctions.

The same argument proves the unbounded solid and
coefficient variants. In the solid case, use
$P_{X,T}^{\LL_{\Box}}\simeq P_{X,T}^{\Box}$
and the derived solidification adjunction to obtain
the same mapping-complex calculation. For
$\underline{\Q}_p$-coefficients, use the corresponding
free module generators.
We will use these unbounded equivalences in what follows.

\end{remark}
 \section{Fully-faithfulness results}\label{lato25S}
Topological Vector Spaces  are closely related to Vector Spaces as well as perfect complexes on the Fargues-Fontaine curve. In this section we prove two relevant fully-faithfulness results. 
\subsection{Vector Spaces  and Topological Vector Spaces} The algebraic  pro-\'etale and the  topological  presheaves are  closely related assuming that we enrich the latter. Let $S\in {\rm sPerf}_C$. 
\begin{theorem}{\rm(Enriched fully-faithfulness)}\label{duck1}
Let $S\in{\rm sPerf}_C$ and let
$\sff,\sg\in\sd(S_{\proeet},{\rm Ab})$.
The canonical morphism in $\underline{\sd}(S,{\rm CondAb})$
\begin{equation}\label{faith10-1}
\R\pi_*\R\Hhom_S(\sff,\sg)
\longrightarrow
\R\Hhom_{S^{\rm top}}(\R\pi_*\sff,\R\pi_*\sg)
\end{equation}
is a quasi-isomorphism.
The analogous statement holds for
$\sff,\sg\in\sd(S_{\proeet},\underline{\Q}_p)$,
with the internal $\Hom$'s taken over $\underline{\Q}_p$.
\end{theorem}
\begin{remark} The map \eqref{faith10-1}  is    constructed in the usual way: 
By adjointness, it suffices to construct a map
$$
\R\pi_*\R\Hhom_{S}(\sff,\sg)\otimes^{\LL}\R\pi_*\sff\to \R\pi_*\sg.
$$
For this, we will use the composition 
$$
\R\pi_*\R\Hhom_{S}(\sff,\sg)\otimes^{\LL}\R\pi_*\sff\to \R\pi_*(\R\Hhom_{S}(\sff,\sg)\otimes^{\LL}\sff)\to\R\pi_*\sg,
$$
where  the second arrow is $\R\pi_*$ applied to the canonical map
$$
\R\Hhom_{S}(\sff,\sg)\otimes^{\LL}\sff\to \sg
$$
and the first arrow is the relative cup product defined in the following way. For  $\sg_1,\sg_2\in \sd(S_{\proeet}, {\rm Ab})$, 
 the relative cup product map in $\underline{\sd}(S, {\rm CondAb})$
$$
\R\pi_*\sg_1\otimes^{\LL}\R\pi_*\sg_2\to \R\pi_*(\sg_1\otimes^{\LL}\sg_2)
$$
is induced by the functorial maps  in $\sd({\rm Ab})$, for $Y\in {\rm sPerf}_S$, profinite set $T$: 
\begin{equation}
\R\Gamma_{\proeet}(Y\times T,\sg_1)\otimes^{\LL}\R\Gamma_{\proeet}(Y\times T,\sg_2)\to \R\Gamma_{\proeet}(Y\times T,\sg_1\otimes^{\LL} \sg_2).
\end{equation}
These maps are compatible with the enriched structure maps as the latter are just induced by the identity maps (see the proof of Lemma \ref{enriched1a}).
\end{remark}

\subsubsection{Proof of Theorem \ref{duck1}}
We first pass to presheaves and then use an exact adjunction to compare
the two derived internal $\Hom$'s.

\vskip2mm
{\bf Step 1.} {\em Passage to presheaves.}
We factor the functor $\pi_*$ as
$$
\pi_*: {\rm Sh}(S_{\proeet},{\rm Ab})
\lomapr{\epsilon_*}{\rm Sh}(S^{\rm std}_{\proeet},{\rm Ab})
\lomapr{\iota_*}{\rm PSh}({\rm sPerf}_S,{\rm Ab})
\lomapr{\pi_*^{\rm psh}}
\underline{\rm PSh}(S,{\rm CondAb})_0.
$$
Here $S^{\rm std}_{\proeet}$ is the category ${\rm sPerf}_S$
equipped with the pro-\'etale topology, $\epsilon_*$ is an equivalence,
and $\iota_*$ is the inclusion of sheaves into presheaves.
We use $\epsilon_*$ to regard $\sff$ and $\sg$ as complexes of sheaves
on $S^{\rm std}_{\proeet}$. Set
$$
\sa:={\rm Sh}(S^{\rm std}_{\proeet},{\rm Ab}),\qquad
\spp:={\rm PSh}({\rm sPerf}_S,{\rm Ab}),\qquad
\se:=\underline{\rm PSh}(S,{\rm CondAb})_0.
$$

  For $A\in\spp$, we set
$$
(\pi_*^{\rm psh}A)(Y)
 :=\bigl(T\mapsto A(Y\times T)\bigr)^{\sharp},
$$
where $(-)^{\sharp}$ denotes sheafification on $*_{\proeet}$.
This is well-defined since $Y\times T\in{\rm sPerf}_S$ for
$Y\in{\rm sPerf}_S$ and $T$ profinite. The enrichment is induced
by pullback along parametrized morphisms: for
$f:Y_1\times U\to Y_2$ and $g:V\to U$, use the map
$$
Y_1\times V\longrightarrow Y_2\times V,\qquad
(f\circ({\rm Id}_{Y_1}\times g),{\rm pr}_V).
$$
These pullback maps satisfy the unit and composition
compatibilities before sheafification. Since abelian
sheafification is symmetric monoidal, they induce action
maps satisfying the same compatibilities after
sheafification, hence define the required enrichment.
Precomposition and abelian sheafification are exact and preserve
colimits, so $\pi_*^{\rm psh}$ has these properties. It is also
symmetric monoidal, since tensor products of condensed abelian groups
are obtained by sheafifying the objectwise tensor products.

   For $Z\in{\rm sPerf}_S$, write
$$
P_Z:=\Z[h_Z]\in\spp,\qquad E_Z:=\Z[h_Z^{\rm top}]\in\se.
$$
Then $\pi_*^{\rm psh}P_Z\simeq E_Z:$
at $Y$, both sides are obtained by sheafifying
$T\mapsto\Z[\Hom_S(Y\times T,Z)]$; compare
\eqref{Chicago10}. The objects $P_Z$ are flat, being objectwise free.
The objects $E_Z$ are flat as well, since free condensed abelian groups
are flat. Hence, for a complex $\sm$ of
objects of $\spp$,
$$
\pi_*^{\rm psh}(\sm\otimes^{\LL}P_Y)
 \simeq (\pi_*^{\rm psh}\sm)\otimes^{\LL}E_Y.
$$
Indeed, both tensor products here can be computed termwise.

The left adjoint $\iota^*$ of $\iota_*$ is sheafification. It is
exact and symmetric monoidal. We therefore have the derived
adjunction $\iota^*\dashv\R\iota_*$ and
$$
\iota^*\R\iota_*\sff\simeq\sff,\qquad
\R\pi_*\sff\simeq\pi_*^{\rm psh}\R\iota_*\sff.
$$
The second identification follows by applying the exact functor
$\pi_*^{\rm psh}$ to a $K$-injective resolution.

\vskip2mm 
{\bf Step 2.} {\em Enriched presheaves fully-faithfulness.} 
By Step~1, it suffices to show that the natural morphism
\begin{equation}\label{cicho1}
\pi_*^{\rm psh}\R\Hhom_{\spp}(\sff,\wt{\sg})
 \longrightarrow
\R\Hhom_{\se}
   (\pi_*^{\rm psh}\sff,\pi_*^{\rm psh}\wt{\sg})
\end{equation}
is a quasi-isomorphism for $\sff,\wt{\sg}\in\sd(\spp)$,
provided that $\wt{\sg}\simeq\R\iota_*\sg$ for 
$\sg\in\sd(\sa)$.
By the derived adjunction for sheafification, we have
\[
\R\Hhom_{\spp}(\sff,\R\iota_*\sg)
 \simeq
\R\iota_*\R\Hhom_{\sa}(\iota^*\sff,\sg).
\]
Evaluating the comparison \eqref{cicho1} at $Y\in{\rm sPerf}_S$ and an
extremally disconnected profinite set $T$, and setting
$\sm=\sff\otimes^{\LL}P_Y$, we are therefore reduced to showing
that the natural morphism
\begin{equation}
\R\Hom_{\spp}
 \bigl(\sm\otimes^{\LL}P_{S\times T},\R\iota_*\sg\bigr)
 \longrightarrow
\R\Hom_{\se}
 \bigl(\pi_*^{\rm psh}\sm\otimes^{\LL}\Z[T],\R\pi_*\sg\bigr)
\end{equation}
is a quasi-isomorphism in $\sd({\rm Ab})$.

  To prove this, fix such a $T$ and consider the functors
$$
L_T:\spp\longrightarrow\se,\quad
 \sff\longmapsto\pi_*^{\rm psh}\sff\otimes\Z[T],
\qquad
U_T:\se\longrightarrow\spp,\quad
 \sn\longmapsto\{Z\mapsto \sn(Z)(T)\}.
$$
The classical and enriched Yoneda Lemmas give
$$
\Hom_{\se}(L_TP_Z,\sn)
 \simeq \sn(Z)(T)
 \simeq\Hom_{\spp}(P_Z,U_T\sn).
$$
Since every abelian presheaf admits a presentation by direct sums of
the $P_Z$'s, and $L_T$ preserves colimits, these identifications extend
to an adjunction $L_T\dashv U_T$.
Both functors are exact: for $L_T$, use exactness of
$\pi_*^{\rm psh}$ and flatness of $\Z[T]$ in ${\rm CondAb}$;
for $U_T$, use objectwise exactness in $\se$ and projectivity of
$\Z[T]$ for extremally disconnected $T$ (see
\cite[Th.~2.2,~p.13]{Sch19}). Deriving this adjunction gives
\begin{equation}\label{zimno11}
\R\Hom_{\se}
 \bigl(\pi_*^{\rm psh}\sm\otimes^{\LL}\Z[T],\sn\bigr)
 \simeq\R\Hom_{\spp}(\sm,U_T\sn),
\end{equation}
for complexes $\sm\in\sd(\spp)$ and $\sn\in\sd(\se)$
(see \cite[Tag~0FNC]{Stck}).

  We apply this with $\sn=\R\pi_*\sg$.
Choose a $K$-injective resolution $I^\bullet$ of $\sg$
in $\sa$. Since $\iota_*$ is right adjoint to the exact
functor $\iota^*$, the complex $\iota_*I^\bullet$ is
$K$-injective in $\spp$ and represents $\R\iota_*\sg$
(see \cite[Tag~08BJ]{Stck}).
Moreover,
$$
U_T(\R\pi_*\sg)(Z)=I^\bullet(Z\times T),
$$
since each $I^n$ is a sheaf. Put $Q_T:=P_{S\times T}$.  
For $\sbb\in\spp$, the ordinary Yoneda Lemma identifies
$$
\Hhom_{\spp}(Q_T,\sbb)(Z)
 =\Hom_{\spp}(P_Z\otimes Q_T,\sbb)
 \simeq \sbb(Z\times T),
$$
using $P_Z\otimes Q_T=P_{Z\times T}$. This internal Hom
functor is exact, being precomposition with $Z\mapsto Z\times T$.
It follows that
$$
U_T(\R\pi_*\sg)
 \simeq\R\Hhom_{\spp}(Q_T,\R\iota_*\sg).
$$
Combining this with \eqref{zimno11} and the tensor--Hom adjunction,
we obtain
\begin{equation}
\R\Hom_{\se}
 \bigl(\pi_*^{\rm psh}\sm\otimes^{\LL}\Z[T],\R\pi_*\sg\bigr)
 \simeq
\R\Hom_{\spp}
 \bigl(\sm\otimes^{\LL}Q_T,\R\iota_*\sg\bigr),
\end{equation}
as wanted.

     For $\underline{\Q}_p$-coefficients, put
$R:=\iota_*\underline{\Q}_p$ and
$R^{\rm top}:=\pi_*^{\rm psh}R$.
We use the corresponding module categories and replace $P_Z,E_Z$ by
$R[h_Z]$ and $R^{\rm top}[h_Z^{\rm top}]$, respectively.
The same exact adjunction holds with
$$
L_T(\sff)=\pi_*^{\rm psh}\sff\otimes_{\Z}\Z[T],\qquad
U_T(\sn)(Z)=\sn(Z)(T),
$$
where the $R(Z)$-action on the right is induced by
$R(Z)\to R^{\rm top}(Z)(T)=R(Z\times T)$.
With $Q_T=R[h_{S\times T}]$, the ordinary internal Hom is again
$\sbb\mapsto\{Z\mapsto \sbb(Z\times T)\}$, with this restriction of
scalars. Taking the remaining tensor products and internal Hom's over
$R$ or $R^{\rm top}$, the preceding computation proves the
$\underline{\Q}_p$-version as well.
\begin{remark}
The assumption that the base $S$ is strictly totally
disconnected is not essential in Theorem~\ref{duck1}.
For an arbitrary affinoid perfectoid $S$ over $C$, the
same internal-Hom comparison holds when the target is
defined using enriched presheaves on ${\rm sPerf}_S$,
with strictly totally disconnected test objects.

Indeed, these test objects form a basis for the
pro-\'etale topology and are stable under products
with profinite sets. In the proof, interpret
$Q_T=P_{S\times T}$ as the restricted presheaf
\[
Z\longmapsto \Z[\Hom_S(Z,S\times T)],
\qquad Z\in{\rm sPerf}_S.
\]
The identity
$P_Z\otimes Q_T\simeq P_{Z\times T}$
still holds, so the same argument applies.
\end{remark}

\begin{remark}
Let $S\in{\rm sPerf}_C$, and write
\[
\mathcal E:=
\underline{{\rm PSh}}(S,{\rm CondAb})_0,
\qquad
\mathcal B:=
\underline{{\rm PSh}}^{\rm rig}(S,{\rm CondAb})_0.
\]
 It would be interesting to determine whether the
inclusion $j:\mathcal B\hookrightarrow\mathcal E$
gives rise to a fully faithful derived realization
$\sd(\mathcal B)\longrightarrow\sd(\mathcal E), $
and, if so, to describe its essential image.
In particular, one may ask how this image relates
to the subcategory of derived-rigid objects introduced
in Section~\ref{plan10}

 This comparison does not follow formally from
the proof of Theorem~\ref{duck1}.
The enriched Yoneda Lemma describes the relations
expressing rigidity, but the free abelian enriched
Yoneda presheaves used in that proof need not
themselves be rigid. An analogous argument would
therefore require a comparison of resolutions in
the rigid and ambient categories, together with
the corresponding acyclicity statement. But  it is easy to see that the rigidification functor can not be exact. 
\end{remark}

\subsubsection{Applications of Theorem \ref{duck1}}
  We list now two applications of Theorem \ref{duck1}. 
\begin{corollary}\label{duck2}
Let $\sff,\sg\in \{{\mathbb G}_a, \underline{\Q}_p\}$.  Then we have a natural quasi-isomorphism
$$
\R\pi_*\R\Hhom_{S,\Q_p}(\sff,\sg)\stackrel{\sim}{\to}\R\Hhom_{S^{\rm top},\Q_p}(\sff,\sg).
$$
\end{corollary}
\begin{proof}
This follows immediately from Theorem \ref{duck1} because both $\sff$ and $\sg$ are $\pi_*$-acyclic. 
\end{proof}
Let $K$ be a complete discrete valuation field of mixed characteristic $(0,p)$ and a perfect residue field. Let $X$ be a smooth partially proper rigid analytic variety over $K$. Let ${\mathbb R}^{\rm alg}_{\proeet,*}(X_C,\Q_p)$, $*\in\{\emptyset, c\}$,  be the  Vector Space representing pro-\'etale cohomology: 
$$
S\mapsto \R\Gamma_{\proeet,*}(X_S,\Q_p). 
$$
\begin{corollary} We have a natural quasi-isomorphism
$$
\R\pi_*\R\Hhom_{S,\Q_p}({\mathbb R}^{\rm alg}_{\proeet,*}(X_C,\Q_p),\Q_p)\stackrel{\sim}{\to}\R\Hhom_{S^{\rm top},\Q_p}(\R\pi_*{\mathbb R}^{\rm alg}_{\proeet,*}(X_C,\Q_p),\Q_p).
$$
\end{corollary} 
\begin{proof} Immediate  from Theorem \ref{duck1}.
\end{proof}
 \begin{remark} (1)  Let $S\in{\rm sPerf}_C$, let $\sff^\bullet$ be a
complex of abelian sheaves on $S_{\proeet}$, and let
$\sg\in\sd(S_{\proeet},{\rm Ab})$. The proof of
Theorem~\ref{duck1} also gives a natural quasi-isomorphism
\[
\R\pi_*\R\Hhom_S(\sff^\bullet,\sg)
 \longrightarrow
\R\Hhom_{S^{\rm top}}
   (\pi_*\sff^\bullet,\R\pi_*\sg),
\]
where $\pi_*$ is applied termwise to the  complex
$\sff^\bullet$. The analogous statement holds with
$\underline{\Q}_p$-coefficients.

 Indeed, in the notation of that proof, set
$\sff_1:=\iota_*\sff^\bullet$, with $\iota_*$ applied termwise.
Since $\iota^*\sff_1=\sff^\bullet$, the derived adjunction
for sheafification gives
\[
\R\Hhom_{\spp}(\sff_1,\R\iota_*\sg)
 \simeq
\R\iota_*\R\Hhom_S(\sff^\bullet,\sg).
\]
Applying $\pi_*^{\rm psh}$ and using \eqref{cicho1},
together with
\[
\pi_*^{\rm psh}\sff_1=\pi_*\sff^\bullet,
\qquad
\pi_*^{\rm psh}\R\iota_*\sg\simeq\R\pi_*\sg,
\]
proves the claim.
 
(2) Corollary \ref{duck2} is the reason why we work with strictly totally disconnected spaces in the definition of ${\rm TVS}$'s: we need vanishing of pro-\'etale cohomology of 
${\mathbb G}_a$ and $\underline{\Q}_p$ in nontrivial degrees\footnote{Since the pro-\'etale cohomology of ${\mathbb G}_a$ vanishes in degrees~$\geq 1$ on any perfectoid affinoid and 
the pro-\'etale cohomology of $\underline{\Q}_p$ vanishes on sympathetic perfectoids, we could have also used here sympathetic spaces.}.

(3) We note however that we can not go "higher", that is, we can not   replace strictly totally disconnected spaces $Y$ with w-contractible spaces. This is because in the proof of Theorem \ref{duck1} we need that $Y$ remains strictly totally disconnected  after base change to $T$ and an analogous  permanence property does  not hold for w-contractible $Y$.
\end{remark}

\subsection{Perfect complexes on the Fargues-Fontaine curve and Topological Vector Spaces}
Here we will prove a fully-faithfulness result between perfect complexes on the Fargues-Fontaine curve and Topological Vector Spaces. Via Theorem \ref{duck1} we will reduce it to an analogous result for Vector Spaces (proved by Ansch\"utz-Le Bras in \cite{AnLB}).
\subsubsection{Quasi-coherent sheaves on the Fargues-Fontaine curve}
Recall the definition of the relative Fargues-Fontaine curve (see \cite[Lecture 12]{SW}). Let  $S = {\rm Spa}(R,R^+)$ be an affinoid perfectoid space over the finite field $C^{\flat}$.  Let
$$Y_{\FF,S} := {\rm Spa}(W(R^+),W(R^+)) \setminus V (p[p^{\flat}])
 $$
 be the relative mixed characteristic punctured unit disc. 
It is an analytic adic space over $\Q_p$. The
Frobenius on $R^+$ induces the Witt vector Frobenius and hence  a Frobenius  $\phi$ on $Y_{\FF,S}$ with  free and
totally discontinuous action. The Fargues-Fontaine curve relative to $S$
(and $\Q_p$) is defined as 
$$X_{\FF,S} := Y_{\FF,S}/\phi^{\Z}.
 $$
 
   For an interval $I = [s, r]\subset  (0,\infty)$ with rational endpoints, we have the open subset
$$Y_{\FF,S,I} := \{|\cdot|: |p|^r \leq |[p^{\flat}]| \leq  |p|^s\} \subset  Y_{\FF,S}. 
 $$
  It is a rational open subset of ${\rm Spa}(W(R^+),W(R^+))$ hence 
an affinoid space,
$$
Y_{\FF,S,I}:={\rm Spa}(\B_{S,I},\B^+_{S,I}).
$$ One can form $X_{\FF,S}$ as the quotient of $Y_{\FF,S,[1,p]}$ via the identification $\phi: Y_{\FF,S,[1,1]}\stackrel{\sim}{\to} Y_{\FF,S,[p,p]}$.

Let  $Y$ be an analytic adic space over $\Q_p$. We denote by ${\rm QCoh}(Y )$ the $\infty$-category of (solid) quasi-coherent sheaves
on $Y$, and by ${\rm Nuc}(Y)$ its full $\infty$-subcategory of nuclear objects. Nuclearity is understood affinoid-locally in the sense of \cite[Def.~5.30]{And21}; the internal $\Hom$ formulation used below is supplied by \cite[Prop.~5.35]{And21}. For $Y={\rm Spa}(R,R^+)$ over $\Q_p$, we have ${\rm QCoh}(Y)\simeq\sd(R^{\rm an}_{\Box})$, where $R^{\rm an}_{\Box}:=(R,R^+)_{\Box}$ is the analytic ring associated with the complete Huber pair $(R,R^+)$ in \cite[Sec.~3.3, Th.~3.28]{And21}.
See \cite{And21} for more properties of the category ${\rm QCoh}(Y)$.
\subsubsection{Topological projection functor}
 For $S\in {\rm sPerf}_C$, we have  the functor 
\begin{align}\label{essen1a}
 & \R\overline{\tau}_*: {\rm QCoh}(X_{\FF,S^{\flat}})  \to {\sd}(S,\Q_{p,\Box}),\\
 & \sff\mapsto \{(f_Y: Y\to S)\mapsto \R\Gamma(X_{\FF,Y^{\flat}},{\LL}f^*_Y\sff)\}.\notag
\end{align}
Here ${\sd}(S,\Q_{p,\Box})$ denotes the derived category of the abelian category ${\rm PSh}(S,\underline{\Q}_{p,\Box})_0$ of (non-enriched) topological presheaves on ${\rm sPerf}_S$ with values in solid $\underline{\Q}_p$-modules (see Remark~\ref{koczkodan1}), i.e., the non-enriched analog of $\underline{\sd}(S,\underline{\Q}_{p,\Box})$ from Section~\ref{plane1}; similarly, ${\sd}(S,\Q_{p})$ denotes its analog with condensed coefficients.  We have $\R\Gamma(X_{\FF,Y^{\flat}},{\LL}f^*_Y\sff)\in \sd(\Q_{p}(Y)_{\Box})$.  
We claim that  
 formula \eqref{essen1a} actually defines a solid enriched presheaf. 
 \begin{lemma}\label{essen1b}
 The functor $\R\overline{\tau}_*$ factors canonically via the $\infty$-derived category $\underline{\sd}(S,\Q_{p,\Box})_0$. That is, we have  a functor 
\begin{align}\label{essen1}
  \R{\tau}_*: {\rm QCoh}(X_{\FF,S^{\flat}})  \to \underline{\sd}(S,\Q_{p,\Box})_0
\end{align}
that fits into a commutative diagram
$$
\xymatrix{
  {\rm QCoh}(X_{\FF,S^{\flat}})  \ar[r]^{\R{\tau}_* }\ar[rd]^{\R\overline{\tau}_*} & \underline{\sd}(S,\Q_{p,\Box})_0\ar[d]^{(-)^{\rm cl}}\\
  & {\sd}(S,\Q_{p,\Box})
}
$$
 \end{lemma}
\begin{remark} \label{essen1N} (1) The natural definition of the functor $\R\tau_*$ should proceed along the following lines. We take the formula from \eqref{essen1a}.  We need to explain why this formula can be upgraded to 
  define an  enriched presheaf. That is, we need to define  structure maps
\begin{equation}\label{essen21}
\underline{\sff}^{\prime}_{Y,T}: \quad \R\Gamma(X_{\FF,(Y\times T)^{\flat}},{\LL}f^*_{Y\times T}\sff)\to \R\uHom_{\Q_{p,\Box}}(\Q_p[T]^{\Box},\R\Gamma(X_{\FF,Y^{\flat}},{\LL}f^*_Y\sff)),
\end{equation}
for $Y\in {\rm sPerf}_S$ and a profinite set $T$. 
That this can be done  follows (working on the $Y_{\FF}$-curve and then taking $\phi$-eigenspaces) from the fact  that, for a stably  uniform analytic adic space ${\rm Spa}(R,R^+)$ over $\Q_p$ and  any $M\in \sd(R^{\rm an}_{\Box})$, we have 
\begin{align}\label{derived}
\R\uHom_{R^{\rm an}_{T,\Box}}(R^{\rm an}_{T,\Box}, &\, M\otimes^{\LL}_{R^{\rm an}_{\Box}}R^{\rm an}_{T,\Box})  \simeq \R\uHom_{R^{\rm an}_{T,\Box}}(R^{\rm an}_{T,\Box},\scc(T,R)\otimes^{\LL}_{R^{\rm an}_{\Box}} M)\\
 & \lomapr{\delta} \R\uHom_{R^{\rm an}_{T,\Box}}(R^{\rm an}_{T,\Box},\R\uHom_{R^{\rm an}_{\Box}}(R^{\rm an}_{\Box}[T],M))
 \simeq  \R\uHom_{R^{\rm an}_{\Box}}(R^{\rm an}_{\Box}[T], M)\notag \\
 & \simeq \R\uHom_{R^{\rm an}_{\Box}}(\Q_p[ T]^{\Box}\otimes^{\LL}_{\Q_{p,\Box}}R^{\rm an}_{\Box}, M),\notag
\end{align}
where $(R_T,R_T^+)$ is the coordinate Huber pair of the product ${\rm Spa}(R,R^+)\times T$ and $R^{\rm an}_{T,\Box}:=(R_T,R_T^+)_{\Box}$. 

 However this approach would require us to work in the context of enriched $\infty$-presheaves and check an infinite number of coherencies. We  would like to avoid this so we will describe  a more down to earth definition in the 
 proof below.
 
  (2)  If $\sff$ is nuclear, by \cite[Prop. 5.35]{And21}, the map $\delta$ above  is a quasi-isomorphism hence so are the structure maps \eqref{essen21}.  
\end{remark}
 \begin{proof}{\rm(Proof of Lemma \ref{essen1b})} 
 We start with defining the functors
 \begin{equation}\label{lato25}
 \R\tau_{I,*}: {\rm QCoh}(Y_{\FF,S^{\flat},I})\to \underline{\sd}(S,\Q_{p,\Box})_0, 
 \end{equation}
 for  intervals $I\subset (0,\infty)$ with rational endpoints. 

  Set $\sa:=\B^{\rm an}_{S^{\flat},I,\Box}$ and identify
${\rm QCoh}(Y_{\FF,S^{\flat},I})$ with $\sd(\sa)$.
The abelian category of $\sa$-complete modules is Grothendieck
and has compact projective generators $\sa[U]$, where $U$ runs
through extremally disconnected profinite sets
\cite[Th.~2.7]{And21}.
Hence every complex admits a $K$-projective resolution
$\spp(\sff)\stackrel{\sim}{\longrightarrow}\sff$
(see \cite[Cor.~3.5]{Spal}).
The construction in \cite[Sec.~3.2(a)--3.5]{Spal}, using these
generators, allows us to take each term of $\spp(\sff)$ to be
a direct sum of modules $\sa[U]$.

    For $f:Y\to S$, put $\sa_Y:=\B^{\rm an}_{Y^{\flat},I,\Box}$.
Analytic extension of scalars
$-\otimes_{\sa}\sa_Y$ has left derived functor
$-\otimes^{\LL}_{\sa}\sa_Y$ by
\cite[Prop.~2.9]{And21}. Hence, termwise analytic
extension of scalars gives
\[
\spp(\sff)\otimes_{\sa}\sa_Y
 \simeq \sff\otimes^{\LL}_{\sa}\sa_Y.
\]
Moreover, analytic extension of scalars preserves
$K$-projective complexes, since its right adjoint, restriction
of scalars on complete modules, is exact.
We define the complex of presheaves
 \begin{equation}\label{def-kol}
  \R\tau_{I,*}(\sff):= \{(f:Y\to S) \mapsto \Gamma(Y_{\FF, Y^{\flat},I}, f^*\spp(\sff))\}.
 \end{equation}
 Its value on $Y\to S$ is in complexes of $\underline{\Q}_{p}(Y)_{\Box}$-modules.

   We claim that the complex of presheaves \eqref{def-kol} is naturally enriched. Indeed, for that we can just use a nonderived version of \eqref{derived} (we set  
 here $R^{\rm an}_{\Box}=(R,R^+)_{\Box}:=\B^{\rm an}_{Y^{\flat},I,\Box}$ to simplify the notation):
 for any $R^{\rm an}_{\Box}$-complete module $M$, we have the enriching structure maps
\begin{align}\label{kol-def11}
\underline{M}^{\prime}_{Y,T}:\quad   M\otimes_{R^{\rm an}_{\Box}}R^{\rm an}_{T,\Box}  & \simeq \scc(T,R)\otimes_{R^{\rm an}_{\Box}} M
  \lomapr{\delta_{Y,T}} \uHom_{R^{\rm an}_{\Box}}(R^{\rm an}_{\Box}[T],M)\\
 & \simeq \uHom_{R^{\rm an}_{\Box}}(\Q_p[ T]^{\Box}\otimes_{\Q_{p,\Box}}R^{\rm an}_{\Box}, M).\notag
\end{align}
These maps are functorial in $M$, $Y$, and $T$. They satisfy
the identity and composition compatibilities of
Lemma~\ref{goraco1}: under
\[
\scc(T',\scc(T,R))\simeq\scc(T\times T',R),
\]
successive evaluation agrees with evaluation on $T\times T'$.
Thus
\[
M\longmapsto\{Y\mapsto M\otimes_{\sa}\sa_Y\}
\]
is an additive functor from $\sa$-complete modules to enriched
presheaves, and applying it termwise to $\spp(\sff)$ gives the
enrichment on \eqref{def-kol}.
A quasi-isomorphism between $K$-projective complexes is a
homotopy equivalence, and this additive functor preserves
chain homotopies. The termwise construction is a dg functor on
$K$-projective complexes, so its total left derived functor
is well-defined on the derived $\infty$-categories and gives \eqref{lato25}; the preceding
base-change calculation identifies its value at $Y$ with
$\R\Gamma(Y_{\FF,Y^{\flat},I},\LL f^*\sff)$.

   Set now $I=[1,p], I^{\prime}=[1,1]$.  Let $\sff\in {\rm QCoh}(X_{\FF,S^{\flat}})$ and identify it functorially with a pair $(\sff_I,\phi_\sff)$, where $\phi_\sff:\LL\phi^*\sff_I\simeq\LL j^*\sff_I$ is a quasi-isomorphism. Here $j: Y_{\FF,S^{\flat},I^{\prime}}\hookrightarrow Y_{\FF,S^{\flat},I}$ is the canonical open immersion and 
  $$\phi: Y_{\FF,S^{\flat},[1,1]}\hookrightarrow Y_{\FF,S^{\flat},[1/p,1]}  \lomapr{\phi} Y_{\FF,S^{\flat},[1,p]}$$ is  the canonical open immersion followed by Frobenius. Set $\sff_{I^{\prime}}:=\LL j^*\sff_I$ and consider   the induced maps 
\begin{align*}
j^*:  \R\tau_{I,*}(\sff_I)\to  \R\tau_{I^{\prime},*}(\sff_{I^{\prime}}),\quad  \phi^*: \R\tau_{I,*}(\sff_I)\to  \R\tau_{I^{\prime},*}(\sff_{I^{\prime}}).
\end{align*}
These maps are in $\underline{\sd}(S,\Q_{p,\Box})_0$ by functoriality of the structure maps \eqref{kol-def11}. 
 Hence we get an object in $\underline{\sd}(S,\Q_{p,\Box})_0$
\begin{equation}\label{cone1}
 \R\tau_{*}(\sff):=[ \R\tau_{I,*}(\sff_I)\lomapr{j^*-\phi^*} \R\tau_{I^{\prime},*}(\sff_{I^{\prime}})]
\end{equation}
 yielding a functor \eqref{essen1} wanted in the lemma.
Also, we clearly have a natural transformation
$$
 \R\tau_{*}(-)^{\rm cl}\simeq  \R\overline{\tau}_{*}(-),
$$
as wanted. 
 \end{proof}
\subsubsection{Topological and algebraic  projection functors}
The algebraic  projection takes values in the usual derived category of $\Q_p$-module sheaves on $S_{\proeet}$; the topological projection takes values in the derived category of enriched presheaves with condensed coefficients. We will explain how they are related. 

 Let $S\in {\rm sPerf}_C$. Consider the morphism of sites
$$
\tau': (X_{\FF, S^{\flat}})_{\proeet}\to S_{\proeet}. 
$$   
induced by sending $Y\in {\rm sPerf}_{S}$ to $X_{\FF,Y^{\flat}}$. 
 It yields the pushforward functor
\begin{align}\label{simons1}
& \R\tau^{\prime}_*: {\rm Nuc}(X_{\FF,S^{\flat}})\to \sd(S_{\proeet},{\Q}_p),\quad \sff\mapsto \{Y\mapsto \R\Gamma(X_{\FF,Y^{\flat}},{\LL}f^*\sff)| f:Y\to S\}.
\end{align}
Here the  cohomology is seen in the derived $\infty$-category $\sd(\Q_{p}(Y))$  of $\Q_{p}(Y)$-modules.
 $ \R\tau^{\prime}_*(\sff)$ is a sheaf. Indeed,   by \cite[Rem.~2.2.2]{ALBM}, nuclear objects in the
category of solid quasi-coherent complexes on the Fargues-Fontaine curve identify
with nuclear objects in the modified solid category
used there. Under this identification, the assignment
in \eqref{simons1} is the restriction to $S_{\proeet}$
of the functor $\sigma_{\Q_p}$ constructed in
\cite[Lem.~5.3.1, Rem.~5.3.2(ii)]{ALBM}.
Hence it satisfies pro-\'etale hyperdescent. 

 We will need the following fact: 
  \begin{lemma}\label{essen-cafe1} The following diagram commutes
$$
\xymatrix{
 {\rm Nuc}(X_{\FF,S^{\flat}})\ar[r]^-{\R\tau^{\prime}_*}\ar[rd]^{\R\tau_*} &  \sd(S_{\proeet},{\Q}_p)\ar[d]^{\R\pi_*}\\
  & \underline{\sd}(S,{\Q}_p)_0.
  }
  $$
 \end{lemma}
 \begin{proof}  We have the functors from $ {\rm Nuc}(X_{\FF,S^{\flat}})$ to ${\sd}(S,{\Q}_p)$ given by sending $\sff\in {\rm Nuc}(X_{\FF,S^{\flat}})$ to the topological  presheaves:
 \begin{align*}
& (\R\pi_* \R\tau^{\prime}_*)^{\rm cl}\sff: \{(Y,T)\mapsto \R\Gamma(X_{\FF,(Y\times T)^{\flat}},{\LL}f^*_{Y\times T}\sff)\},\\
& \R\overline{\tau}_*\sff: \{(Y,T)\mapsto \R\Hom_{\Q_p}(\Q_p[T],\R\Gamma(X_{\FF,Y^{\flat}},{\LL}f^*_{Y}\sff))\}.
 \end{align*}
 Moreover, we  have  a natural transformation 
 \begin{equation}\label{twarog1}
  (\R\pi_* \R\tau^{\prime}_*)^{\rm cl}\to  \R\overline{\tau}_*
 \end{equation}
 given by analytic descent from the structure morphisms  $\underline{\sff}^{\prime}_{Y,T}$ described in \eqref{essen21}. By Remark \ref{essen1N},  they are quasi-isomorphisms when restricted to ${\rm Nuc}(X_{\FF,S^{\flat}})$.
 
      Both sides of \eqref{twarog1} are enriched and it suffices now to show that the natural transformation \eqref{twarog1} preserves this enrichment. That is, that it lifts to a natural transformation
$$  \R\pi_* \R\tau^{\prime}_*\to  \R{\tau}_*
 $$
 For that we can pass via the mapping fiber \eqref{cone1} and its algebraic analog to showing that the natural transformation
 $$
( \R\pi_* \R\tau^{\prime}_{I,*})^{\rm cl}\to  \R\overline{\tau}_{I,*}
 $$
 is compatible with enrichment. Using the structure maps \eqref{kol-def11}, we check the
compatibility on the total complex associated with the
chosen $K$-projective resolution; for nuclear $\sff$, this
complex identifies with
$\R\iota_*\R\tau'_{I,*}\sff$. The compatibility diagrams commute
as diagrams of complexes, before passage to the derived
category. 
 After unwinding the definitions, this boils down to the identifications 
 $$\scc(T^{\prime},R_{Y\times T})\simeq \scc(T^{\prime},\scc(T, R_{Y}))\simeq\scc(T^{\prime}\times T,R_{Y})$$ 
 and, via them,  we have $$\delta_{Y,T}\delta_{Y\times T,T^{\prime}}=\delta_{Y,T\times T^{\prime}}.$$
 \end{proof}

\subsubsection{Topological fully-faithfulness for perfect complexes}

  Let $S\in {\rm sPerf}_C$. For  $\sff_1,\sff_2\in {\rm QCoh}(X_{\FF, S^{\flat}})$, we have a natural map in $\underline{\sd}(S,\Q_{p,\Box})$:
\begin{equation}
\R\tau_*\R\Hhom_{{\rm QCoh}(X_{\FF,S^{\flat}})}(\sff_1,\sff_2)\to \R\Hhom^{\Box}_{S,\Q_{p}}(\R\tau_*\sff_1,\R\tau_*\sff_2),
\end{equation}
which is  constructed in the usual way: 
By adjointness, it suffices to construct a map
$$
\R\tau_*\R\Hhom_{{\rm QCoh}(X_{\FF,S^{\flat}})}(\sff_1,\sff_2)\otimes^{\LL_{\Box}}_{\underline{\Q}_p}\R\tau_*\sff_1\to \R\tau_*\sff_2.
$$
For this, we use the composition 
$$
\R\tau_*\R\Hhom_{{\rm QCoh}(X_{\FF,S^{\flat}})}(\sff_1,\sff_2)\otimes^{\LL_{\Box}}_{\underline{\Q}_p}\R\tau_*\sff_1\to \R\tau_*(\R\Hhom_{{\rm QCoh}(X_{\FF,S^{\flat}})}(\sff_1,\sff_2)\otimes^{\LL}_{\so}\sff_1)\to\R\tau_*\sff_2,
$$
where  the second arrow is $\R\tau_*$ applied to the canonical map
$$
\R\Hhom_{{\rm QCoh}(X_{\FF,S^{\flat}})}(\sff_1,\sff_2)\otimes^{\LL}_{\so}\sff_1\to \sff_2
$$
and the first arrow is the relative cup product defined in the following way. For  $\sg_1,\sg_2\in {\rm QCoh}(X_{\FF, S^{\flat}})$,  the relative cup product map in $\underline{\sd}(S,\Q_{p,\Box})$
\begin{equation}
\R\tau_*\sg_1\otimes^{\LL_{\Box}}_{\underline{\Q}_p}\R\tau_*\sg_2\to \R\tau_*(\sg_1\otimes^{\LL}_{\so}\sg_2)
\end{equation}
is induced by the functorial maps  in $\underline{\sd}(S,\Q_{p,\Box})$:
\begin{equation*}
\R\tau_{I,*}\sg_1\otimes^{\LL_{\Box}}_{\underline{\Q}_p}\R\tau_{I,*}\sg_2\to \R\tau_{I,*}(\sg_1\otimes^{\LL}_{\so}\sg_2), 
\end{equation*}
where $\R\tau_{I,*}$ is the functor from \eqref{lato25}. And the latter comes from the 
identifications
$$
f_Y^*\sg_1\otimes_{\so} f_Y^*\sg_2 \stackrel{\sim}{\to} f_Y^*(\sg_1\otimes_{\so} \sg_2)
$$
on the level of $K$-projective resolutions of the type considered in the proof of Lemma \ref{essen1b}.

    The following result is a topological version of  \cite[Cor. 3.10]{AnLB}, which proved fully-faithfulness for $\R\tau^{\prime}_*$ and follows from Theorem \ref{duck1}.
\begin{corollary}\label{lebrasII}
Let $S\in {\rm sPerf}_C$.  The functor 
$$
\R\tau_*: {\rm Perf}(X_{\FF, S^{\flat}})\to \underline{\sd}(S,\Q_{p,\Box})
$$
is fully faithful. That is, for $\sff_1,\sff_2\in {\rm Perf}(X_{\FF, S^{\flat}})$, the natural morphism in $\underline{\sd}(S,\Q_{p,\Box})$
$$\R\tau_*\R\Hhom_{{\rm QCoh}(X_{\FF,S^{\flat}})}(\sff_1,\sff_2)\to \R\Hhom^{\Box}_{S,\Q_{p}}(\R\tau_*\sff_1,\R\tau_*\sff_2)
$$
is a quasi-isomorphism. In particular,  the natural map in ${\sd}(\Q_p(S)_{\Box})$
$$
\R\uHom_{{\rm QCoh}(X_{\FF,S^{\flat}})}(\sff_1,\sff_2)\to \R\uHom^{\Box}_{S,\Q_{p}}(\R\tau_*\sff_1,\R\tau_*\sff_2)
$$
is a quasi-isomorphism. 
\end{corollary}
\begin{proof}  We may pass to the condensed setting, that is, we want to show that, for  $\sff_1, \sff_2 \in {\rm Perf}(X_{\FF, S^{\flat}})$,   the canonical map   (we note that $\R\Hhom_{{\rm QCoh}(X_{\FF,S^{\flat}})}(\sff_1,\sff_2)\in {\rm Perf}(X_{\FF, S^{\flat}})$):
$$
\R\tau_*\R\Hhom_{{\rm QCoh}(X_{\FF,S^{\flat}})}(\sff_1,\sff_2)\to \R\Hhom_{S^{\rm top},\Q_{p}}(\R\tau_*\sff_1,\R\tau_*\sff_2)
$$
is a quasi-isomorphism in $\underline{\sd}(S,\Q_{p})$.

  For that, consider its factorization:
  \begin{align*}
\R\pi_*\R\tau^{\prime}_*\R\Hhom_{{\rm QCoh}(X_{\FF,S^{\flat}})}(\sff_1,\sff_2) &\stackrel{\sim}{ \to }
\R\pi_*\R\Hhom_{S,\Q_{p}}(\R\tau^{\prime}_*\sff_1,\R\tau^{\prime}_*\sff_2)\\
 & \to 
\R\Hhom_{S^{\rm top},\Q_{p}}(\R\tau_*\sff_1,\R\tau_*\sff_2).
\end{align*}
The first morphism is a quasi-isomorphism by \cite[Cor. 3.10]{AnLB}. It remains  to show that the map
\begin{equation}
\R\pi_*\R\Hhom_{S,\Q_{p}}(\R\tau^{\prime}_*\sff_1,\R\tau^{\prime}_*\sff_2)
  \to 
\R\Hhom_{S^{\rm top},\Q_{p}}(\R\tau_*\sff_1,\R\tau_*\sff_2)
\end{equation}
is a quasi-isomorphism. But this follows from Theorem \ref{duck1} and Lemma \ref{essen-cafe1}. 
\end{proof}
\begin{example}  
(1). We first record the canonical quasi-isomorphisms
\[
\R\tau_*\so\simeq\underline{\Q}_p,
\qquad
\R\tau_*(i_{\infty,*}R_S)\simeq{\mathbb G}_a
\]
in $\underline{\sd}(S,\underline{\Q}_{p,\Box})_0$.
Indeed, in $\sd(S_{\proeet},\underline{\Q}_p)$,
we have
\[
\R\tau'_*\so\simeq\underline{\Q}_p,
\qquad
\R\tau'_*(i_{\infty,*}R_S)\simeq{\mathbb G}_a.
\]
The first identification follows from
\cite[Prop.~II.2.5(ii)]{FS}.
For the second, the pullback of $i_{\infty,*}R_S$
to $X_{\FF,Y^\flat}$ is $i_{\infty,Y,*}R_Y$,
whose derived global sections are $R_Y[0]$.

Both $\so$ and $i_{\infty,*}R_S$ are perfect, hence
nuclear. Lemma~\ref{essen-cafe1} therefore identifies
their topological realizations with the images under
$\R\pi_*$ of the preceding algebraic realizations.
Lemma~\ref{acyclic1} and Example~\ref{ex1} give
\[
\R\pi_*\underline{\Q}_p\simeq\underline{\Q}_p,
\qquad
\R\pi_*{\mathbb G}_a\simeq{\mathbb G}_a,
\]
with their canonical enrichments, proving the asserted
identifications. These also hold in the solid category
by Proposition~\ref{nyear1b}.

    (2) We also have the following quasi-isomorphisms  in ${{\rm QCoh}(X_{\FF,S^{\flat}})}$:
\begin{align*}
&  \R\Hhom_{{\rm QCoh}(X_{\FF,S^{\flat}})} (\so ,\so  )   \simeq \so,\\
& \R\Hhom_{{\rm QCoh}(X_{\FF,S^{\flat}})} (\so  ,i_{\infty,*}R_S )  \simeq i_{\infty,*}R_S,\\
& \R\Hhom_{{\rm QCoh}(X_{\FF,S^{\flat}})} (i_{\infty,*}R_S, i_{\infty,*}R_S )   \simeq i_{\infty,*}R_S \oplus i_{\infty,*}R_S(-1)[-1],\\
&  \R\Hhom_{{\rm QCoh}(X_{\FF,S^{\flat}})} (i_{\infty,*}R_S,\so  )  \simeq  i_{\infty,*}R_S(-1)[-1].
\end{align*}
Via Corollary \ref{lebrasII}, this can be transferred to the following quasi-isomorphisms in $\underline{\sd}(S,\Q_{p,\Box})$:
\begin{align*}
& \R\Hhom^{\Box}_{S,\Q_p}(\Q_p,\Q_p)  \stackrel{\sim}{\leftarrow}\Q_p,\quad \R\Hhom^{\Box}_{S,\Q_p}(\Q_p,{\mathbb G}_a)\stackrel{\sim}{\leftarrow}{\mathbb G}_a,\\
&  \R\Hhom^{\Box}_{S,\Q_p}({\mathbb G}_a,{\mathbb G}_a)  \simeq{\mathbb G}_a\oplus{\mathbb G}_a(-1)[-1],\quad 
  \R\Hhom^{\Box}_{S,\Q_p}({\mathbb G}_a,\Q_p)  \simeq{\mathbb G}_a(-1)[-1].
\end{align*}
\end{example}
\appendix
\section{Big solid sheaves and Topological Vector Spaces}
We compare here  the big solid sheaves of \cite[Section  5.4]{ALBM} with Topological
Vector Spaces. The comparison functor is fully faithful and is compatible with
tensor products on the nuclear subcategories.

Let $S\in {\rm sPerf}_C$. Write $S_{\mathrm{Q}\proeet}$ for the big
quasi-pro-\'etale site of spatial diamonds over $S$, and set 
\[
 \sbb_S:=\sd_{\Box}(S_{\mathrm{Q}\proeet},\Q_p),
 \qquad
 \stt_S:=\underline{\sd}(S,\underline{\Q}_{p,\Box})_0.
\]
For a spatial diamond $Y/S$, let
$\rho_Y\sff=\epsilon_{Y,*}(\sff|_Y)$ be restriction to the small site
$Y_{\proeet}$. Following \cite[Def.~5.4.3]{ALBM}, we call $\sff$ big
solid if
$\rho_Y\sff\in\sd_{\Box}(Y,\Q_p)$
for every spatial diamond $Y/S$.
Here $\sd_{\Box}(Y,\Q_p)$ denotes the category of
solid complexes on the small quasi-pro-\'etale site,
as in \cite[Def.~VII.1.10 and the subsequent paragraph,
Def.~VII.1.17]{FS}.   The functor $\R\pi_*$, applied after restriction to
the perfectoid test objects, takes values in
$\stt_S$ on $\sbb_S$, as verified in
Proposition~\ref{app:big-solid-ff} below.
We denote this functor by
\[
\Phi:\sbb_S\longrightarrow\stt_S.
\]
It is characterized by 
$$
 \Phi(\sff)(Y)(T)
 =\R\Gamma_{\proeet}(Y\times\wh T,\sff), \quad  Y\in {\rm sPerf}_S,  \mbox{ profinite set }  T.
$$
Set $A_Y=\scc(\pi_0Y,\Q_p)=\underline{\Q}_p(Y)$. The tensor product in
$\stt_S$ is computed objectwise over the solid algebras $A_Y$.

\begin{proposition}\label{app:big-solid-ff}
The functor $\Phi:\sbb_S\to\stt_S$ is fully faithful. For
$\sff,\sg\in\sbb_S$, the canonical map
\[
 \Phi\R\Hhom_{\sbb_S}(\sff,\sg)
 \longrightarrow
 \R\Hhom_{\stt_S}(\Phi\sff,\Phi\sg)
\]
is a quasi-isomorphism.
\end{proposition}
\begin{proof}
For $Y\in {\rm sPerf}_S$, the condensed complex $\Phi(\sff)(Y)$ is
the small-site pushforward of $\rho_Y\sff$ along $Y\to\Spa C$.
It is solid by \cite[Lemma 5.4.2]{ALBM}. Hence $\Phi$ takes values
in $\stt_S$. The assertions now follow from Theorem~\ref{duck1} and
Proposition~\ref{nyear1b}. For the internal Hom, we also use
\cite[Lemma 5.4.4]{ALBM}: internal Hom into a big solid object is
big solid, and agrees with the internal Hom in the ambient derived
category of sheaves.
\end{proof}

We write $\otimes_{\mathrm{big}}^{\LL_{\Box}}$ for the solid tensor
product in $\sbb_S$. The usual cup product gives a canonical lax
symmetric monoidal structure on $\Phi$, with comparison maps
\begin{equation}\label{late1}
 \Phi\sff\otimes_{\underline{\Q}_p}^{\LL_{\Box}}\Phi\sg
 \longrightarrow
 \Phi(\sff\otimes_{\mathrm{big}}^{\LL_{\Box}}\sg).
\end{equation}
To compare these tensor products, we first record the following:

\begin{lemma}\label{app:big-solid-localization}
For every spatial diamond $Y/S$, restriction $\rho_Y$ commutes with
solidification. That is, for $\sff,\sg\in\sbb_S$,
\[
 \rho_Y(\sff\otimes_{\mathrm{big}}^{\LL_{\Box}}\sg)
 \simeq
 \rho_Y\sff\otimes_{\underline{\Q}_p}^{\LL_{\Box}}\rho_Y\sg,
\]
where the tensor product on the right is taken on $Y_{\proeet}$.
\end{lemma}
\begin{proof}
We start with the first claim. Small-site solidification commutes with pullback along every morphism of
spatial diamonds \cite[Prop. VII.1.13]{FS}. For an arbitrary
$\sff\in\sd(S_{\mathrm{Q}\proeet},\Q_p)$, consider 
\[
 Y\longmapsto
 \R\Gamma_{\proeet}\bigl(Y,(\rho_Y\sff)^{\LL_{\Box}}\bigr).
\]
The transition maps are induced by
$f^*(\rho_Y\sff)^{\LL_{\Box}}
\simeq(f^*\rho_Y\sff)^{\LL_{\Box}}
\to(\rho_{Y'}\sff)^{\LL_{\Box}}$, for $f:Y'\to Y$.
For quasi-pro-\'etale $f$, these are the usual restriction isomorphisms.
Thus this is a derived sheaf on the big site, whose restriction to
$Y_{\proeet}$ is $(\rho_Y\sff)^{\LL_{\Box}}$.
It is big solid. The unbounded descent used here is justified  because the
big site is replete and its derived category is left-complete (see the
paragraph following \cite[Definition 5.4.3]{ALBM}).
The universal properties of the small solidifications,
compatible with all pullbacks, identify it with the big solidification
of $\sff$. This proves the first claim.

   Ordinary restriction to the
small site preserves the ordinary sheaf tensor product, since coverings
of its objects are the same in the two sites. Applying solidification
proves the second claim.
\end{proof}

In this appendix, we call $\sff\in\sbb_S$ \emph{big nuclear} if
$\rho_Y\sff\in\sd_{\mathrm{nuc}}(Y,\Q_p)$ for every
$Y\in {\rm sPerf}_S$, and write $\sbb_S^{\mathrm{nuc}}$ for this full
subcategory (see \cite[Definition 3.26 and Lemma 3.28]{AM}).

\begin{proposition}\label{pen1}
The subcategory $\sbb_S^{\mathrm{nuc}}$ contains the unit and is closed
under $\otimes_{\mathrm{big}}^{\LL_{\Box}}$. The restriction
\[
 \Phi:\sbb_S^{\mathrm{nuc}}\longrightarrow\stt_S
\]
is fully faithful and symmetric monoidal: the canonical map
\eqref{late1} is an isomorphism for
$\sff,\sg\in\sbb_S^{\mathrm{nuc}}$.
\end{proposition}
\begin{proof} 
For $Y\in {\rm sPerf}_S$, condensed global sections induce an equivalence
\begin{equation}\label{app:big-solid-small-nuclear}
 \R\Gamma_{\rm cond}(Y,-):
 \sd_{\mathrm{nuc}}(Y,\Q_p)
 \simeq \sd_{\mathrm{nuc}}(A_Y)
\end{equation}
by \cite[Lemma 3.29]{AM}. This equivalence is symmetric monoidal for the
solid tensor products, over $\Q_p$ on the small site and over $A_Y$ on
the right. Indeed, the equivalence with the component space in loc.
cit. is symmetric monoidal, and the inverse to condensed global
sections on that space is the symmetric monoidal extension-of-scalars
functor. Also
$\sd_{\mathrm{nuc}}(A_Y)
\simeq {\rm Mod}_{A_Y}(\sd_{\mathrm{nuc}}(\Q_p))$,
since $A_Y$ is nuclear.

Lemma~\ref{app:big-solid-localization} and
\eqref{app:big-solid-small-nuclear} imply both the claimed closure and
the canonical identifications
\begin{align*}
 \Phi(\sff\otimes_{\mathrm{big}}^{\LL_{\Box}}\sg)(Y)
 &\simeq
 \R\Gamma_{\rm cond}
   (Y,\rho_Y\sff\otimes_{\Q_p}^{\LL_{\Box}}\rho_Y\sg)\\
 &\simeq
 \Phi(\sff)(Y)\otimes_{A_Y}^{\LL_{\Box}}\Phi(\sg)(Y).
\end{align*}
Their inverses are the objectwise components of
\eqref{late1}, so the isomorphism is compatible with
enrichment and with all monoidal constraints. The unit is sent to
$\underline{\Q}_p$. Full faithfulness follows from
Proposition~\ref{app:big-solid-ff}.
\end{proof}
\begin{corollary}
Let $V\in\sd_{\mathrm{nuc}}(\Q_p)$, and let $V_S$ be its pullback to
the big site. 
\begin{enumerate}
\item For $\sff\in\sbb_S^{\mathrm{nuc}}$, there is a canonical
isomorphism
\[
 \Phi(V_S\otimes_{\mathrm{big}}^{\LL_{\Box}}\sff)
 \simeq V\otimes_{\Q_p}^{\LL_{\Box}}\Phi(\sff),
\]
where the right-hand side is computed objectwise.
\item Let $\sigma=\R\tau'_*$ denote the ordinary realization of
\eqref{simons1}. If $M$ is a nuclear object on the relative Fargues--Fontaine
curve, then $\sigma(M)$ is big nuclear. In particular, for nuclear
$M,N$,
\begin{equation}\label{app:big-solid-curve-images}
 \R\tau_*M\otimes_{\underline{\Q}_p}^{\LL_{\Box}}\R\tau_*N
 \simeq
 \Phi\bigl(\sigma(M)\otimes_{\mathrm{big}}^{\LL_{\Box}}\sigma(N)\bigr).
\end{equation}
The coefficient action is compatible with realization:
\[
 \R\tau_*(V\otimes_{\Q_p}^{\LL_{\Box}}M)
 \simeq V\otimes_{\Q_p}^{\LL_{\Box}}\R\tau_*M.
\]
\end{enumerate}
\end{corollary}
\begin{proof}
The restrictions of $V_S$ correspond under
\eqref{app:big-solid-small-nuclear} to $A_Y\otimes_{\Q_p}^{\LL_{\Box}}V$.
The first assertion follows from
Proposition~\ref{pen1}.
The proof of \cite[Lemma 5.4.6]{ALBM}, together with compatibility of
$\sigma$ with pullback, shows that every small restriction of
$\sigma(M)$ on a strictly totally disconnected test object is nuclear.
Now apply the Proposition \ref{pen1} and Lemma~\ref{essen-cafe1}, which identifies
$\R\tau_*M$ with $\Phi\sigma(M)$. The last assertion follows also from
\cite[Lemma 5.4.7]{ALBM}.
\end{proof}


\begin{thebibliography}{99}
 \bibitem{AGN} P.~Achinger, S.~Gilles, W.~Nizio\l, 
{\em Compactly supported $p$-adic pro-\'etale cohomology of analytic varieties.}
Nagoya Math. J. 261 (2026), e56.

\bibitem{And21} G.~Andreychev, 
{\em Pseudocoherent and perfect complexes and vector bundles on analytic adic spaces.}
\url{arXiv:2105.12591 [math.AG]}.

\bibitem{AnLB}  J.~Ansch\"utz, A-C.~Le Bras, 
{\em A Fourier Transform for Banach-Colmez spaces.}  	
J. Eur. Math. Soc. (JEMS) 27 (2025), 3651--3712.

\bibitem{ALBM} J.~Ansch\"utz, A-C.~Le Bras, L.~Mann, 
{\em  A 6-functor formalism for solid quasi-coherent sheaves on the Fargues-Fontaine curve.}
\url{arXiv:2412.20968 [math.AG]}.

\bibitem{AM} J.~Ansch\"utz, L.~Mann,  {\em Descent for solid quasi-coherent sheaves on perfectoid spaces.}
 \url{https://arxiv.org/abs/2403.01951}.
 
\bibitem{BAC}  F.~Borceux, C.~Quinteiro, 
{\em A theory of enriched sheaves.}
  Cahiers Topologie G\'eom. Diff\'erentielle Cat\'eg. 37 (1996), 145--162.

\bibitem{BMS24} S.~Ben-Moshe, T.~M.~Schlank, 
{\em Higher semiadditive algebraic $K$-theory and redshift.}
 Compositio Math. 160 (2024), 237--287.

\bibitem{Cas1} J.M.F. Castillo, 
{\em The Hitchhiker Guide to Categorical Banach Space Theory. Part I.}
 Extracta Math. 25 (2010) 103--149.

\bibitem{Cas2} J.M.F. Castillo, 
{\em The Hitchhiker guide to Categorical Banach space theory. Part II.}
\url{arXiv:2110.06300 [math.FA]}.

\bibitem{Sch19} D.~Clausen, P.~Scholze, 
{\em Lectures on condensed mathematics.} 
\url{https://www.math.uni-bonn.de/people/scholze/Condensed.pdf}, 2019.

\bibitem{CF} P.~Colmez, 
{\em Espaces de Banach de dimension finie.}
 J. Inst. Math. Jussieu 1 (2002), 331--439.

\bibitem{CGLN}  P.~Colmez, S.~Gilles, Z.~Li, W.~Nizio\l, 
{\em $ C_{\st}$-conjecture for compactly supported pro-\'etale cohomology.}
\url{arXiv:2611.11589 [math.NT]}.


\bibitem{CGN3} P.\,Colmez, S.\,Gilles, W.\,Nizio\l,
{\em Une conjecture $C_{\st}$  pour la cohomologie \`a support compact.}
\url{arXiv:2511.14675v2 [math.AG] }.

\bibitem{CGN2} P.~Colmez, S.~Gilles, W.~Nizio\l, 
{\em Duality for $p$-adic geometric pro-\'etale cohomology.}
\url{arXiv:2411.12163 [math.AG]}.

\bibitem{CN4} P.~Colmez, W.~Nizio\l, 
{\em On the cohomology of $p$-adic analytic spaces, I: The basic comparison theorem.}
  J. Algebraic Geometry~{\bf 34} (2025), 1--108.
 
\bibitem{CN5} P.~Colmez, W.~Nizio\l,  
{\em On the cohomology of $p$-adic analytic spaces, II: The $C_{\st}$-conjecture}.  
Duke Math. J. 174 (2025), 2203--2301.

\bibitem{FS} L.~Fargues, P.~Scholze, 
{\em Geometrization of the local Langlands correspondence.}   
Ast\'erisque 466 (2026) 1--434.

\bibitem{GH15} D.~Gepner,  R.~Haugseng, 
{\em Enriched $\infty$--categories via non-symmetric $\infty$--operads.}
 Adv. Math. 279 (2015), 575--716.

\bibitem{GM} B. J.~Guillou, P.~May, 
{\em Enriched model categories and presheaf categories.}
 New York J. Math. 26 (2020), 37--91.

\bibitem{HP} C.~Herz, J. W.~Pelletier, 
{\em Dual functors and integral operators in the category of Banach spaces.}
J. Pure Appl. Algebra 8 (1976), 5--22.

\bibitem{HM} C.~Heyer, L.~Mann, 
{\em $6$-Functor Formalisms and Smooth Representations.}  	
\url{arXiv:2410.13038 [math.CT]}.

\bibitem{Hin20} V.~Hinich, 
{\em Yoneda lemma for enriched $\infty$-categories.}
 Adv. Math. 367 (2020), 107--129.

\bibitem{Gr} H.~Al Hwaeer, G.~Garkusha, 
{\em Grothendieck categories of enriched functors.}
 J. Algebra 450 (2016), 204--241.

\bibitem{Kelly} G.-M.~Kelly, 
{\em Basic concepts of enriched category theory.}
London Math. Soc. Lecture Note Ser., 64
Cambridge University Press, Cambridge-New York, 1982, 245 pp.

\bibitem{KL15}K.~S.~Kedlaya, R.~Liu,  
{\em Relative $p$-adic Hodge theory: foundations.} 
Ast\'erisque 371 (2015), 239pp.

\bibitem{HA} J.~Lurie, 
{\em Higher algebra}. https://www.math.ias.edu/~lurie/papers/HA.pdf, 2017.

\bibitem{Mac21}A.~W.~Macpherson, 
{\em The operad that corepresents enrichment.}
  Homology Homotopy Appl. 23 (2021), 387--401.

\bibitem{MW} L.~Mann, A.~Werner, 
{\em Local systems on diamonds and $p$-adic vector bundles.}
 Int. Math. Res. Not. IMRN (2023), 12785--12850.

\bibitem{Pel} J. W.~Pelletier, 
{\em Dual functors and the Radon-Nikodym property in the category of Banach spaces.}
J. Austral. Math. Soc. Ser. A 27 (1979), 479--494.

\bibitem{PGS} C.~Perez-Garcia,  W. H.~Schikhof, 
{\em Locally convex spaces over non-Archimedean valued fields}. Cambridge
Studies in Advanced Mathematics, vol. 119, Cambridge University Press, Cambridge, 2010.

\bibitem{RZ} D.~Reutter, M.~Zetto, 
{\em Enriched $\infty$-categories as marked module categories.}
\url{arXiv:2501.07697 [math.AT]}.

\bibitem{RR}  J. Rodrigues Jacinto, J. E. Rodriguez Camargo, {\em Solid locally analytic representations of $p$-adic Lie groups.} Representation Theory 26 (2022), 962--1024.

\bibitem{Ros} A.~Rosenfield, 
{\em Enriched Grothendieck topologies under change of base.}  	
\url{arXiv:2405.19529 [math.CT]}.

\bibitem{Ser} C.~Serp\'e, 
{\em Resolution of unbounded complexes in Grothendieck categories}.  
J. Pure Appl. Algebra 177 (2003), 103--112.

\bibitem{Schn} J.-P. Schneiders, 
{\em Quasi-abelian categories and sheaves.}
 M\'em. Soc. Math. Fr. 76, 1999.

\bibitem{Sch13} P.~Scholze, 
{\em $p$-adic Hodge theory for rigid-analytic varieties.}
 Forum Math. Pi 1 (2013), e1.

\bibitem{SchD} P.~Scholze, 
{\em  \'Etale cohomology of diamonds.}
Ast\'erisque 466 (2026) 439--672.

\bibitem{SW} P.~Scholze, J.~Weinstein, 
{\em Berkeley lectures on $p$-adic geometry.}
 Ann. of Math. Stud., 207,  2020.

\bibitem{Spal} N.~Spaltenstein 
{\em Resolutions of unbounded complexes.}
 Compositio Math. 65 (1988), 121--154.

\bibitem{Stck} The Stacks Project Authors, The Stacks Project. 
\url{http://stacks.math.columbia.edu}.

\bibitem{MO} J.~Van Name (\url{https://mathoverflow.net/users/22277/joseph-van-name}), 
"Which spaces are inverse limits of discrete spaces ?", URL (version: 2012-03-29): 
\url{https://mathoverflow.net/q/92608}
\end{thebibliography}
\end{document}